\documentclass[11pt,leqno]{amsart}
\usepackage{amsmath,amsfonts,amsthm,amscd,amssymb,mathrsfs,graphicx,xcolor,bm}
\usepackage{epstopdf}
\numberwithin{equation}{section}
\usepackage{fullpage}

\newcommand\s{\sigma}

\renewcommand{\L}{\Lambda}

\renewcommand\d{\partial}
\renewcommand\a{\alpha}

\def\l{\lambda}
\def\eps{\varepsilon }
\def\e{\varepsilon}

\renewcommand\d{\partial}
\renewcommand\a{\alpha}

\newcommand\R{\mathbb R}

\def\eps{\varepsilon}
\def\e{\varepsilon}

\def\l{\lambda}

\newcommand\br{\begin{remark}}
	\newcommand\er{\end{remark}}
\newcommand\bp{\begin{pmatrix}}
	\newcommand\ep{\end{pmatrix}}
\newcommand{\be}{\begin{equation}}
\newcommand{\ee}{\end{equation}}
\newcommand\ba{\begin{equation}\begin{aligned}}
\newcommand\ea{\end{aligned}\end{equation}}

\newcommand{\bap}{\begin{app}}
	\newcommand{\eap}{\end{app}}
\newcommand{\begs}{\begin{exams}}
	\newcommand{\eegs}{\end{exams}}
\newcommand{\beg}{\begin{example}}
	\newcommand{\eeg}{\end{exaplem}}
\newcommand{\bpr}{\begin{proposition}}
	\newcommand{\epr}{\end{proposition}}
\newcommand{\bt}{\begin{theorem}}
	\newcommand{\et}{\end{theorem}}
\newcommand{\bc}{\begin{corollary}}
	\newcommand{\ec}{\end{corollary}}
\newcommand{\bl}{\begin{lemma}}
	\newcommand{\el}{\end{lemma}}
\newcommand{\bd}{\begin{definition}}
	\newcommand{\ed}{\end{definition}}
\newcommand{\brs}{\begin{remarks}}
	\newcommand{\ers}{\end{remarks}}

\newcommand{\B }{\mathcal{B}}

\newcommand{\RR}{{\mathbb R}}

\newcommand{\CC}{{\mathbb C}}

\newcommand{\const}{\text{\rm constant}}
\newcommand{\Id}{{\rm Id }}

\newcommand{\diag}{{\rm diag }}

\newcommand{\sgn}{\text{\rm sgn}}
\newtheorem{theorem}{Theorem}[section]
\newtheorem{proposition}[theorem]{Proposition}
\newtheorem{corollary}[theorem]{Corollary}
\newtheorem{lemma}[theorem]{Lemma}

\theoremstyle{remark}
\newtheorem{remark}[theorem]{Remark}
\newtheorem{remarks}[theorem]{Remarks}
\theoremstyle{definition}
\newtheorem{definition}[theorem]{Definition}

\newtheorem{example}[theorem]{Example}

\newcommand\cE{{\mathcal  E}}

\newcommand\cP{{\mathcal  P}}

\newcommand\cO{{\mathcal O}}

\newcommand\cT{{\mathcal T}}
\newcommand\cS{{\mathcal S}}

\newcommand{\spec}{\operatorname{spec}}

\newcommand{\beq}{\begin{equation}}
\newcommand{\eeq}{\end{equation}}

\newcommand{\Hx}{\hat{x}}
\newcommand{\Ht}{\hat{t}}

\newcommand{\abs}[1]{\lvert#1\rvert}

\usepackage{color,soul}

\title{Linear stability analysis for a system of singular amplitude equations arising in biomorphology}
\author{Aric Wheeler}
\address{Indiana University, Bloomington, IN 47405}
\email{awheele@iu.edu }
\thanks{Research of A.W. was partially supported
under NSF grants no. DMS-1700279 and DMS-2038056.}
\author{Kevin Zumbrun}
\address{Indiana University, Bloomington, IN 47405}
\email{kzumbrun@indiana.edu} 
\thanks{Research of K.Z. was partially supported
under NSF grants no. DMS-0300487 and DMS-0801745.}
\begin{document}
\begin{abstract}
	We study linear stability of exponential periodic solutions of a system of singular amplitude equations 
	associated with convective Turing bifurcation in the presence of conservation laws,
	as arises in modern biomorphology models, binary fluids, and elsewhere.
	Consisting of a complex Ginzburg-Landau equation coupled with a singular convection-diffusion
	equation in ``mean modes'' associated with conservation laws, these were shown previously by the authors
	to admit a constant-coefficient linearized stability analysis as in the classical Ginzburg-Landau case-
	albeit now singular in wave amplitude $\eps$-
	yielding useful \emph{necessary} conditions for stability, both of the exponential functions as solutions
	of the amplitude equations, and of the associated periodic pattern solving the underlying PDE.
	Here, we show by a delicate two-parameter matrix perturbation analysis that (strict) satisfaction of
	these necessary conditions is also {\it sufficient} for diffusive stability in the sense of Schneider,
	yielding a corresponding result, and nonlinear stability, for the underlying PDE.
	Moreover, we show that they may be interpreted as stability along a non-normally hyperbolic
	slow manifold approximated by Darcy-type reduction, together with attraction along transverse mean modes,
	connecting with finite-time approximation theorems of H\"acker-Schneider-Zimmerman.
\end{abstract}

\maketitle

\tableofcontents
\section{Introduction}\label{s:intro}
In this paper, we carry out a linearized stability analysis 
for periodic traveling-wave solutions
\be\label{tw}
(A,B)(\hat x,\hat t)=(A_0 e^{i(\kappa \hat x - \omega \hat t)}, B_0), 
\ee
$A_0, B_0 = \const$, of a system of singular amplitude equations
\be\label{ampeq}\tag{mcGL}
\begin{aligned}
	A_{\hat t}&=aA_{\hat x\hat x}+bA+c|A|^2A+dAB,\\
	B_{\hat t}&=\e^{-1}(f B_{\hat x}+h |A|^2_{\hat x})+e_B B_{\hat x\hat x}+\Re(gA\bar{A}_{\hat x})_{\hat x},
\end{aligned}
\ee
with $A, a,b,c\in \CC$, $B\in \R^m$, $d\in \CC^{1\times m}$, $g, h\in\R^{m\times 1}$, $f,e_B\in\R^{m\times m}$, 
$\Re a,\Re b>0$, $\Im\spec(f)=0$, and $\Re \spec( e_B)>0$,
in the limit as $\eps\to 0$.

The first example of such a system, and the striking phenomenon of singular convection in $B$,
was derived via multi-scale expansion by H\"acker, Schneider, and Zimmerman \cite{HSZ} 
in the context of weakly unstable B\'enard-Marangoni and thin-film flow, for which $d=0$ and
the equations partially decouple.
The singular convection they observed is a surprising consequence of the interaction of conserved quantities
and convective forces in the underlying pde.
For, in the absence of conserved quantities, the relevant amplitude equations are the famous (nonsingular)
complex Ginzburg-Landau equations \cite{E,vH,KSM,M3} consisting of the $A$ equation alone, 
whereas, in the absence of convection, the $B$ equation becomes purely diffusive, and
\eqref{ampeq} reduces to a nonsingular system discovered previously by Matthews and Cox \cite{MC}.

The more general, fully coupled version described above was derived in \cite{WZ3} 
by multi-scale expansion as a set of amplitude equations formally
governing weakly unstable behavior near Turing bifurcation of a general family of convection reaction diffusion system
\be \label{bl}
\hbox{$u_t + f(u,\nu)_x-(B(u,\nu)u_x)_x=R(u,\nu):=\bp  R_1(u,\nu)\\ 0_m \ep$, \qquad $R_1$ full rank.}
\ee
$u\in \R^n$, indexed by bifurcation parameter $\nu$ centered about $0$,
in the presence of conservation laws: that is, for $R$ of co-rank $0<m\leq n$. 
As described in \cite{WZ1,WZ3}, the main motivation for this study was from modern biomorphology models
incorporating mechanical/hydrodynamical effects, in particular vasculogenesis models as in \cite{AGS,SBP}.

In particular, the periodic solutions \eqref{tw} of \eqref{ampeq} correspond to approximate solutions
\ba\label{eq:ansatz}
U^\e(\xi,\Hx,\Ht)&=\frac{1}{2}A(\Hx,\Ht)e^{i\xi}r+c.c.+\e^2(B(\Hx,\Ht)+ c.c. + H.O.T. 
\ea
of \eqref{bl},
where $\xi=k_*(x-d_*t)$, $\Hx=\e(x-(d_*+\delta)t)$, and $\Ht=\e^2t$ for $d_*,\delta$ real numbers determined
by formal expansion, $c.c.$ denotes complex conjugate, and $H.O.T$ denotes higher-order terms.
Here, $\nu\sim \eps^2$, while $r$ corresponds to the bifurcating
neutral direction for the Fourier symbol of the linearized equations about the constant state $u(x,t)\equiv U_0$
from which Turing bifurcation occurs.
For relations of model parameters of \eqref{ampeq} to the form of \eqref{bl}, see \cite{WZ3}.

We refer to \eqref{ampeq} as {\it modified complex Ginzburg-Landau equations} by analogy to the standard complex 
Ginzburg-Landau equations playing a corresponding role for convective Turing bifurcation without conservation laws
\cite{M3,WZ1,WZ2}.

\begin{example}\label{AGS}
	An example of \eqref{bl} with $m=1$, arising in biomorphology, 
	is the variation on the hydrodynamic/chemotactic vasculogenesis
	models of \cite{GAC,AGS}:\footnote{We note that the nonconservative form of term $\beta \rho \nabla c$, 
	appearing in the second, \emph{nonconservative} equation of \eqref{MOeq}, does not
	change the analysis, here or in \cite{WZ1,WZ2,WZ3}, in any appreciable way.}
\ba\label{MOeq}
	\partial_t \rho + \nabla \cdot (\rho u)&=\mu \Delta \rho,\\
	\partial_t (\rho u) + \nabla \cdot (\rho u\otimes u)+ \nabla P(\rho) - \beta \rho \nabla c 
	&= \nu \Delta u -\gamma \rho u,\\
	\partial_t (c) &=  D\Delta c + \alpha \rho -\tau^{-1}c;
\ea
	see also \cite{LW}.
Here, $\rho$ and $u$ are density and velocity of endothelial cells,
$P$ pressure, $c$ chemo-attractant concentration, $\alpha >0$ and $\beta>0$ release and cell response rates,
and $\tau>0$ the half-life of the chemo-attractant.
The $\gamma\geq 0$ and $\nu \geq 0$ terms model in different ways drag against the extracellular matrix.
They are set to $\gamma>0$, $\nu=0$ in \cite{AGS} and $\gamma=0$, $\nu>0$ in \cite{GAC};
here, we take $\gamma, \nu > 0$. The term $\mu \Delta \rho$ is a nonphysical ``artificial viscosity'',
set to $\mu=0$ in \cite{GAC,AGS} and here $\mu>0$.
	In \eqref{AGS}, the pressure is taken to be zero below some critical density; in \cite{GAC} it is
	taken to be zero. Here, following \cite{LW}, we take $P(\rho)=A\rho^2$, or, more generally, 
	$P=A\rho^p$, $p\geq 1$.
\end{example}

\smallskip \noindent
\textbf{Existence}. We recall briefly the existence theory for \eqref{ampeq}, both
for context/general interest, and to introduce a supercriticality condition that will be
important in the later stability analysis.
Substituting $|A|=A_0$, $B\equiv B_0$ into \eqref{ampeq} satisfies the second equation automatically
and in the first equation gives a shifted complex Ginzburg-Landau equation
\be\label{shifted}
A_{\hat t}=aA_{\hat x\hat x}+\tilde b  A+c|A|^2A, \qquad \tilde b:= b+dB_0
\ee
for which periodic existence may be treated in standard fashion. 

Specifically, substituting $A=A_0e^{i(\omega \hat t+\kappa \hat x)}$
gives
\be\label{A0soln}
A_0^2=\frac{\Re( \tilde b) - \Re(a)\kappa^2}{-\Re(c)}.
\ee
Here, $\Re(a)>0$ by parabolicity/well-posedness and $\Re(b)>0$ by standard Turing assumptions/exchange of stability
\cite{WZ2}. We will assume also the supercriticality condition
\be\label{supercrit}
\Re(c)<0
\ee
allowing (as is standard) existence of solutions with $\kappa=0$, yielding finally the domain of existence
\be\label{solndom}
\kappa^2<\kappa_{exist}^2:= \frac{\Re(\tilde b)}{\Re(a)} = \frac{\Re( b) + \Re(d) B_0}{\Re(a)}.
\ee

Finally, we assume
\be\label{compat}
\hbox{\rm $\Re \spec(iz f + z^2 e_B)<0$ for nonzero $z \in \R$,}
\ee
as follows from a more detailed analysis of the Turing assumptions, namely, the assumption that
the spectra of the linearized operator for \eqref{bl} about the bifurcating constant solution
be strictly negative for $\sigma\neq 0$ except for a single complex-conjugate pair with real part order $\eps^2$.
Taking account of the asymptotics relating \eqref{ampeq} to \eqref{bl}, this corresponds to the assumption 
on the linearization 
\ba\label{zerolin}
	A_{\hat t}&=aA_{\hat x\hat x}+bA+dAB,\\
	B_{\hat t}&=\e^{-1}f B_{\hat x}+e_B B_{\hat x\hat x}
	\ea
of \eqref{ampeq} about $(A,B)\equiv (0, B_0)$ that spectra be negative except 
for a single complex-conjugate pair with real part order $1$.  
By inspection, the spectra consist of two copies of $\spec(b+ a \sigma^2)$, $b>0$, union with
$\spec(\eps^{-1}f\sigma + \sigma^2 e_b)$, whence \eqref{compat} follows by setting $\sigma=\eps^{-1}z$ to
obtain
$$
\spec(\eps^{-1}f\sigma + \sigma^2 e_b)= \eps^{-2}\spec(zf\sigma + z^2 e_b).
$$

\br\label{B0limrmk}
Note that \eqref{solndom} limits $B_0$ to range
\be\label{B0dom}
B_0\in \sgn \Re(d)\times (-\Re(b)/|\Re(d)|,+\infty).
\ee
In particular, the choice $B_0=0$ is always feasible.
Contrarily, by the change of coordinate $B\to B-B_0$, one could take without loss of generality $B_0=0$
for any existing periodic wave, thereby changing $b$ to $\tilde b>0$.
\er

\smallskip \noindent
\textbf{Stability.} As noted in \cite{WZ3},
just as for the classical complex Ginzburg-Landau equation, the linearized stability
problem for \eqref{ampeq} about exponential waves \eqref{tw} may, by the exponentially weighted change of coordinates
$(A,B)= \Big((A_0+u+iv)e^{i(\kappa \hat x-\omega \hat t)}, B_0+w)$, $u$, $v$, $w$ real, be converted to 
{\it real, constant-coefficient} form
\ba\label{lineq}
u_{\hat t}&=\Re(a)u_{\hat x\hat x}-\Im(a)v_{\hat x\hat x}-2\kappa\big(\Im(a)u_{\hat x}
+\Re(a)v_{\hat x} \big)+2A_0^2\Re(c)u+A_0\Re(d)w,\\
v_{\hat t}&=\Im(a)u_{\hat x\hat x}+\Re(a)v_{\hat x\hat x}+2\kappa\big(\Re(a)u_{\hat x}-\Im(a)v_{\hat x}\big)
+2A_0^2\Im(c)v+A_0\Im(d)w,\\
w_{\hat t}&=\e^{-1}(f w_{\hat x}+2h A_0 u_{\hat x} )+e_B w_{\hat x\hat x}
+2A_0\big(u_{\hat x}\Re(g)+v_{\hat x}\Im(g)+\kappa u\Im(g) \big)_{\hat x},
\ea
or, setting $\mathcal{U}=(u,v,w^T)^T$,
\be\label{Msys}
\mathcal{U}_t = \mathcal{L}\mathcal{U}:= 
M_0 \mathcal{U} + M_1(\eps) \mathcal{U}_{\hat x} + M_2 \mathcal{U}_{\hat x\hat x}
\ee
with $M_j$ appropriately defined (see Section \ref{s:nec}). Note that neither $B_0$ nor $b$ appears.

The spectrum of the linearized operator $\mathcal{L}$ about the wave may thus be determined by linear
algebraic computations, via the dispersion relation for \eqref{Msys}, or
\be\label{Mdisp}
\hat \lambda(\sigma)\in \spec M(\eps,\hat \sigma), \qquad
M(\eps,\hat \sigma):= M_0  + M_1(\eps)\hat \sigma + M_2\hat \sigma^2
\ee
with $\hat \sigma $ denoting Fourier wave number, a {\it global two-parameter matrix perturbation problem.}

Recall that {spectral stability}, necessary for linearized stability, corresponds to 
\be\label{ss1}
\Re \hat \lambda \leq 0 \; \hbox{\rm for $\hat \lambda=\hat \lambda(\hat \s)\in \spec(M(\e,\hat \s))$.}
\ee
Meanwhile {diffusive spectral stability} in the sense of Schneider \cite{S1,S2}, sufficient for linearized and
nonlinear stability, corresponds in this context to 
\be\label{dss1}
\Re \hat \lambda \leq c(\eps)|\hat \sigma|^2/(1+|\hat \sigma|^2) \; \hbox{\rm for 
$\hat \lambda=\hat \lambda(\hat \s)\in \spec(M(\e,\hat \s))$.}
\ee

For general periodic waves of \eqref{bl}, Schneider's diffusive stability condition is defined in terms
of Bloch-Floquet  spectrum of the wave, as discussed in Section \ref{s:validation}, below, 
with the Fourier wave number $\hat \sigma$ replaced by a Bloch-Floquet number. See Section \ref{s:validation}.

\smallskip\noindent

{\bf Darcy reduction.} A natural further reduction of \eqref{ampeq} \cite{WZ3,HSZ} suggested by the 
singular structure of \eqref{ampeq}
is to make the Ansatz $B+f^{-1}h|A|^2=\tilde B_0$ for some fixed constant $\tilde B_0$. To relate the stability criteria of the Darcy reduction to the original model, we choose the constant $\tilde B_0$ by $\tilde B_0=B_0+f^{-1}h|A_0|^2$ for some fixed choice of periodic solution $(A_0e^{i(\kappa\hat x+\omega \hat t) },B_0)$ of \eqref{ampeq}. Note then that $\hat b$ as in \eqref{Darcy} is then a function of $\kappa$. Symbolically, we then obtain for $B=B(A)$
\be\label{Dansatz}
B(A)= B_0 + f^{-1} h |A_0|^2 - f^{-1} h |A|^2,  
\ee
canceling the singular term in \eqref{ampeq}(ii), and giving for $A$ the shifted complex Ginzburg-Landau equation
\be\label{Darcy}
\hbox{\rm $A_{\hat t}=aA_{\hat x\hat x}+\hat bA+\hat c|A|^2A$, where 
$\hat b= b+ d (B_0+  f^{-1} h A_0^2)$, \, $\hat c:= c-d f^{-1} h$,}
\ee
denoted the {\it Darcy} approximation by analogy to related approximations in hydrodynamics.
See, e.g., \cite{DJMR} for similar approximations in weakly nonlinear geometric optics and MHD.
We note for later reference the relations (cf. \eqref{shifted})
\be\label{rels}
\hat b=\tilde b+df^{-1}h A_0^2, \quad \hat c= c- df^{-1}h.
\ee

This does not represent an invariant subflow, but could be viewed, heuristically, as
an approximate ``slow manifold.'' 
We note that the corresponding fast flow
$$
( B + f^{-1} h |A|^2)_{\hat t}=  \e^{-1}\partial_{\hat x}f ( B + f^{-1} h |A|^2)_{\hat t}+ H.O.T.
$$
linearized about a constant state has for $\spec(f)$ real, dispersion relation
$\lambda(k)=i\spec{f}\eps^{-1} + O(1)$ with real part of order $1$ rather than $\eps^{-1}$, 
so that the slow manifold is typically {\it non-normally hyperbolic}.

Noting further that periodic solutions, with $|A|=A_0$, $B=B_0$ constant, are contained entirely within the
Darcy flow, we may think of these as lying within an exact slow manifold, with local dynamics about
the periodic solutions given to first order by those of the Darcy flow \eqref{Darcy}.
Hence, one may guess that {\it Darcy stability} of periodic waves, or stability as solutions of \eqref{Darcy} is
necessary- but not sufficient- for stability as solutions of \eqref{ampeq}.
Stability within the full model \eqref{ampeq} evidently involves also the question whether fast flow is
attracting or repelling, which, by non-normal hyperbolicity, may be expected to be somewhat delicate.

\subsection{Previous work}\label{s:previous}
The approximate solutions \eqref{eq:ansatz} may be shown in by now standard fashion to be 
close to exact periodic solutions of \eqref{bl}, with sharp error bounds; see \cite{WZ3} for further details.
Notably, this existence problem is {\it nonsingular}, as one can see by setting $|A|,B$ constant, thus
eliminating the $B$-equation, and noting that the resulting dimensional solution count remains correct.
Though we do not display it here, the analog of \eqref{ampeq} in the O(2)-symmetric reaction diffusion
case is nonsingular in both $A$ and $B$ to begin with, and so the issue of singularity does not ever arise.
For this reason we were able to treat in standard fashion not only existence but also stability completely for
that case in \cite{WZ3}, using classical methods of Mielke and Schneider \cite{M1,M2,S1,S2}, showing
under generic conditions that diffusive stability of periodic waves of \eqref{bl} is equivalent
to diffusive stability of the associated periodic waves \eqref{tw} of \eqref{ampeq} in \eqref{eq:ansatz}.

However, in the convective case modeled by \eqref{ampeq}, the stability problem is not only singular, but the
neutral eigenstructure of the Fourier symbol about the constant state $U_0$ features a Jordan block that would
at first sight appear to prevent completely an analytic Taylor expansion about zero frequency as is the
first step in the analysis of \cite{M1,M2,S1,S2}.
Using a key observation of \cite{JZ1,JZ2,JNRZ}, we were able to overcome this apparent obstacle and perform Taylor
expansion on a ball of $\eps$ times smaller order than in the standard case, thus obtaining in \cite{WZ3}
{\it necessary but not sufficient} conditions for diffusive spectral stability in the sense of Schneider \cite{S1,S1}
corresponding to low-frequency diffusive stability of the associated periodic waves \eqref{tw} of \eqref{ampeq} in \eqref{eq:ansatz}.

We mention also the earlier results of \cite{HSZ} of a different, \emph{bounded-time approximation} type, 
for specific decoupled ($d=0$) versions 
of \eqref{ampeq} arising in B\`enard-Marangoni and thin film flow, in which they showed that a corresponding
Darcy model for \emph{localized} rather than periodic solutions $(A,B)$ of \eqref{ampeq}, 
with the modified relation $B=- f^{-1} h |A|^2$, accurately predicts bounded-time behavior, in the sense that
localized solutions of the resulting Darcy model lie near exact solutions of the underlying PDE \eqref{bl}.
In the decoupled case considered there, it was clear that this ``approximate slow manifold'' was
attracting, suggesting at least heuristically formal link between bounded time approximation and 
time-asymptotic behavior.  A natural question posed in \cite{WZ3} was, in the fully coupled ($d\neq 0$)
and perturbed periodic case considered there, to what extent the Darcy model \eqref{Darcy} 
can shed light on \emph{time-asymptotic behavior} as studied here.

\subsection{Main results}\label{s:results}
The purpose of the present paper is to go beyond the restrictive regime imposed by Taylor expansion and 
complete a full linearized stability analysis, for both periodic solutions of the amplitude equations \eqref{ampeq} 
and the exact periodic solutions of \eqref{bl} that the former approximate, yielding {\it sufficient as
well as necessary conditions for diffusive stability.}

\medskip\noindent
\textbf{Main result 1.} Our first main result is, under mild genericity assumptions,
the derivation of $m+1$ simple necessary and sufficient conditions
for diffusive spectral stability \eqref{dss} of periodic solutions \eqref{tw} of the amplitude
equations \eqref{ampeq} for $0<\eps \leq \eps_0$ sufficiently small, where $m=\dim B$ is the number of mean modes.
See Proposition \ref{stabcrit} in the scalar case $m=1$ and Proposition \ref{stabcritvec} in the vector case $m>1$.
Here, $\eps_0$ may be chosen uniformly on compact parameter-sets for which the above-mentioned genericity
conditions are satisfied.

The first condition, associated with the neutral translational mode is
\be\label{tcond}
\kappa^2< \kappa_{stab}^2 := \frac{2\frac{\Re(\tilde b)}{\Re(c)}(\Im(a)\Im(\hat c)+\Re(a)\Re(\hat c) )}{ 4\Re(a)^2(1+\hat q^2)+2\frac{\Re(a)}{\Re(c)}\big(\Im(a)\Im(\hat c)+\Re(a)\Re(\hat c) \big)},
\ee
deriving from the well-known Eckhaus condition \cite[Eqn. (2.13)]{WZ3}:
\be\label{modeq}
0>\mu_t^0:= \frac{(2\kappa^2\Im \hat c ^2\Re a^2+A_0^2\Im a\Im \hat c\Re \hat c^2+
\Re a\Re \hat c^2(2\kappa^2\Re a+A_0^2\Re \hat c))2}{(A_0^2\Re \hat c^3)}
\ee
for the complex Ginzburg-Landau equation, with $b$, $c$ replaced by $\hat b$, $\hat c$:
that is, the translational stability condition for the Darcy approximation.

We note that \eqref{tcond} is {\it not} simply the $\kappa$ boundary for the Darcy equation with
$b$ and $c$ suitably replaced, since the relation
$A_0^2=\frac{\Re( \tilde b) - \Re(a)\kappa^2}{-\Re(c)}$ from \eqref{A0soln} involves the original variable
$\tilde b$ and $c$ rather than $\hat b$ and $\hat c$.
To put this a different way, from \eqref{rels}, we have that 
\be\label{subtle}
\hat b=\tilde b+df^{-1}h A_0^2
\ee
depends on $A_0^2$, so is not a fixed parameter. 
This subtle point is perhaps the main difference between the singular and the classical case.
In the decoupled case $d=0$, or for $\kappa=0$, this distinction disappears.

We refer to condition \eqref{tcond} as the {\it generalized Eckhaus criterion}.
The domain $\kappa^2<\kappa_{stab}^2$ is nonempty under the Benjamin-Feir-Newell criterion
	\be\label{BFN}\tag{BFN}
			\Im a\Im \hat c\Re \hat b\Re \hat c+\Re a\Re \hat b\Re \hat c^2>0;
	\ee
otherwise, diffusively stable waves do not exist.
This corresponds to \eqref{modeq} with $\kappa=0$, hence, unlike \eqref{tcond}, is identical
with the classical \eqref{BFN} condition with $b$, $c$ replaced by $\hat b$, $\hat c$.

The remaining $m$ conditions, associated with ``conserved,'' or ``mean'' modes are both
first-order and second-order. The first-order condition, automatic in the scalar case $m=1$, is
	\be\label{preccond}
	\hbox{\rm $\spec\Big( f- h \frac{\Re(d)}{\Re(c)}\Big)$ real}.
	\ee
We make the additional nondegeneracy assumption 
\be\label{splitass}
\hbox{\rm $\spec\Big( f- h \frac{\Re(d)}{\Re(c)}\Big)$ distinct,}
\ee
denoting by $\ell_j$ and $r_j$ associated left and right eigenvector pairs of $f-h\Re(d)/\Re(c)$. Then, the second-order
conditions are
	\be\label{ccond}
\mu_{c,j}^0:=  \ell_j \frac{h\Re(d) }{2A_0^2 \Re(c)}\Big( f- h \frac{\Re(d)}{\Re(c)}\Big)r_j <0,
\ee
with no conditions on $\kappa$. We refer to these as {\it auxiliary Benjamin-Feir-Newell criteria}
for the similar role they play to that of \eqref{BFN} in determination of stability.

{\it The scalar case ($m=1$).} Specialized to a single conservation law, $m=1$, as in the case of
example system \eqref{MOeq}, the results simplify considerably, reducing to the pair of conditions
\be\label{scalarconds}
\hbox{\rm $\Re(\mu_t)<0$ and $\Re(c)< \Re(\hat c)<0$.}
\ee
Noting, as observed just above, that $\mu_t<0$ agrees with the condition for translational stability
of the Darcy approximation, while $\hat c<0$ is necessary for supercriticality/fast mode stability of
the same, we see that, apart from information obtainable from the Darcy ``slow manifold'' approximation,
the only additional requirement is the condition $\Re (c)< \Re (\hat c)$, which we will see in the
computations later on is equivalent to stability of the transversal mean mode, or, heuristically,
attraction of the Darcy slow manifold.

\medskip\noindent
\textbf{Main result 2.} Our second main result is to connect stability of periodic solutions of the 
\eqref{ampeq} to stability of exact periodic solutions of \eqref{bl}
nearby the associated approximate solutions \eqref{eq:ansatz}, the existence of which was established in
\cite{WZ3}.
Namely, we show, under appropriate genericity assumptions, that for $0<\eps\leq \eps_0$ sufficiently small,
diffusive spectral stability in the sense of Schneider of exact periodic waves of \eqref{bl}
is equivalent to diffusive spectral stability \eqref{dss} of approximating periodic waves \eqref{tw}
of the associated amplitude equations \eqref{ampeq} derived in \cite{WZ3}; see Theorem \ref{main2}.
Here, again, $\eps_0$ may be chosen uniformly on compact parameter-sets for which the required genericity
conditions are satisfied.
Thus, {necessary and sufficient conditions for
diffusive spectral stability of $\eps$-amplitude exact periodic solutions of \eqref{bl}, with $0<\eps\leq \eps_0$
sufficiently small, are given by the Eckhaus and auxiliary Benjamin-Feir-Newell conditions \eqref{tcond}
and \eqref{preccond}-\eqref{ccond}.}

\medskip

{\it These complete the program of \cite{WZ1,WZ2,WZ3,Wh}, extending the theory of Eckhaus-Mielke-Schneider
for classical Turing bifurcation to the case of Turing bifurcations with conservation laws}, such
as occur in binary fluids and, most importantly for us, in biomorphogenesis models incorporating hydrodynamic
and mechanical effects, the latter having received considerable recent interest.

\medskip

\medskip\noindent
\textbf{Darcy approximation vs. the full model.} 
In passing, we answer the question posed in \cite{WZ3} of the relation between
the Darcy approximation \eqref{Darcy} and the full model \eqref{ampeq}, showing that, indeed,
Darcy stability is {\it necessary} (but not sufficient) for stability with respect to the full model 
of periodic solutions of \eqref{ampeq}, \eqref{Darcy}.
More precisely, we can see by expansion in $\eps$ with $\check \sigma:=\eps^{-1}\hat \sigma$ held fixed 
(carried out in
Section \ref{s:Dnec}) that Darcy stability is necessary for stability of dispersion relation \eqref{Mdisp}
with wave numbers in intermediate frequency range $1/C\leq |\s|\leq C$, corresponding heuristically
to behavior in the intermediate time range $0<t_0\leq t\leq 1$.
By contrast, the Eckhaus conditions above are found by expanding in $\s$ with $\eps$ fixed and bounded from zero.
But, either by abstract matching arguments or by direct comparison, we see further that the leading order
ultra-slow mode coefficient $\mu_t^0$ obtained by either of these methods is identical. 
That is, the (bounded-time) Darcy and (time-asymptotic) Eckhaus slow modes agree to leading order,
i.e., {\it the Darcy reduction behaves like an approximate slow manifold in both finite-time approximation 
and spectral sense}, similarly as shown in the classical (nonconservative) case by the analysis of Mielke
and Schneider.

The fact that the order of limits does not affect the value of $\mu_t$ can be understood, more generally,
from the fact proved in the course of the analysis of Section \ref{s:suff} (regions (ii)-(iii)),
that {\it the slow mode $\mu_t$ is in fact jointly analytic} in $\eps$, $\hat \sigma$ for $\eps$, $|\hat \sigma|$
sufficiently small, similarly as in the classical case; see Remark \ref{lambda_texp}.
By contrast, the fast modes $\mu_{c,j}$, as noted
in \cite{WZ3} are typically analytic in $\hat \sigma$ only on the reduced range $|\hat \sigma|\ll \eps$.

Concerning the values of $\mu_{c,j}$, the fact that $\Re(c)<0$, $\mu_t^0<0$, and $\mu_{c,j}^0<0$
are sufficient for stability, 
together with the facts that the Darcy stability conditions $\Re(\hat c)$, $\Re (\mu_t^0)$ are necessary 
for stability, implies that the former imply $\Re(\hat c)<0$ also in the vector case $m>1$,
as was seen by direct computation in the scalar case $m=1$.
However, in the vector case, we do not see any easy way to show this directly.

\subsection{Discussion and open problems}\label{s:disc}
System \eqref{ampeq} generalizes amplitude equations derived for particular examples by Matthews-Cox \cite{MC} in 
the O(2)-symmetric reaction diffusion case and H\"acker-Schneider-Zimmermann \cite{HSZ} in the SO(2)-symmetric
reaction convection diffusion case.
A major difference in the O(2) case is that the $\e^{-1}(f B_{\hat x}+h |A|^2_{\hat x})$ convection term
does not appear in \eqref{ampeq}(ii), canceling due to reflective symmetry.
This leaves a nonsingular reaction diffusion system in place of the coupled reaction diffusion/singular convection diffusion system of \eqref{ampeq}, for which the analytic issues are essentially different.
See \cite{MC,S,WZ3} for further discussion.

The singular mean-mode convection studied here was first pointed out in \cite{HSZ}
in the context of weakly nonlinear asymptotics for B\'enard-Marangoni convection and inclined-plane flow,
with rather different motivations and results from those of the present analysis.
More precisely, they sought to establish \eqref{eq:ansatz} together with (a localized version of) 
\eqref{Darcy} as an infinite-dimensional 
approximate center manifold, lying $\eps^2$-close to associated exact solutions in bounded time.
To this end, they derived mode-filtered equations \cite[Eqs. (22)-(23) and (33)-(34]{HSZ} for exact
solutions isolating neutrally growing (i.e., ``slow'') modes, corresponding under appropriate
rescaling to a forced version of \eqref{ampeq},
then showed by rigorous stability estimates that these remain $\eps^2$ close to \eqref{eq:ansatz}
under (localized) Darcy flow \eqref{Darcy} for time interval $0\leq t\leq C\eps^{-2}$.

Note that the time-interval $0\leq t\leq C\eps^{-2}$
corresponds to a bounded time interval $0\leq \hat t\leq C$ for \eqref{ampeq}.
Thus, the above stability estimates correspond to bounded-time, or well-posedness properties of
\eqref{ampeq}, rather than the asymptotic stability questions we consider here.
Moreover, the class of localized (roughly, $H^s(\R)$) solutions considered in \cite{HSZ} does not
include the periodic ones we study here.
In addition, the solutions they construct are for ``prepared'' initial data, with $B$
appropriately coupled to $A$, whereas we are concerned with stability with respect
to general data and perturbations.
On the other hand, the results of \cite{HSZ} apply to a much larger class of solutions, whereas ours apply only
to the periodic solutions \eqref{tw} associated with Turing bifurcation.

A further major technical difference between the analysis of \cite{HSZ} and that carried out here
is that, for the particular models considered in \cite{HSZ}, the associated amplitude equations 
\eqref{ampeq} feature vanishing coupling coefficients $d=0$ and simultaneously diagonal matrices $f$ and $e$.
Hence, the $A$ equation \eqref{ampeq}(i) {\it decouples as a standard complex Ginzburg-Landau equation},
and the spectra of $\mathcal{L}$ in \eqref{Msys} decouples into the well-known spectra for the linearized
complex Ginzburg-Landau equation plus that of the (already linear) decoupled $B$-equation
$B_{\hat t}=\e^{-1}f B_{\hat x}+e_B B_{\hat x\hat x}$, or 
\be\label{inspect}
\hat \lambda_j(\hat \sigma)= \e^{-1}i\hat \sigma f_j -\hat \sigma^2(e_B)_{j},
\ee
where $f_j$, $(e_B)_j$ denote the eigenvalues of $f$ and $e_B$.
Thus, time-asymptotic spectral stability is straightforward in that case, with the main technical issue
being {\it stiffness}, or accounting of transient effects in $B$, with principal behavior given 
by the usual complex Ginzburg-Landau equation in $A$.

Indeed, the singular $B$-equation was not included in \cite{HSZ} as part of the amplitude equations, but subsumed
in the underlying analysis supporting the description of behavior by
complex Ginzburg-Landau approximation, an aspect reflecting an important difference in point of view.
For the general amplitude equations \eqref{ampeq}, derived in \cite{WZ3} in the context of biomorphogenesis
models such as Murray-Oster and related equations for vasculogenesis \cite{AGS,SBP}, such a decoupling generically
does not occur, and the determination of spectral stability changes from inspection as in \eqref{inspect},
to the complicated singular two-parameter matrix perturbation problem studied here,
in terms of $\eps$ and the Fourier (Bloch) number $\sigma$. 
At the same time, mean modes $B$ are no longer necessarily decaying transients, but
through coupling interactions may play an important role in dynamics.
Thus, in the exposition of \cite{WZ3}, the mean modes $B$ are ``promoted'' to elements of the
formal amplitude equations, the latter becoming therefore {\it singular} with respect to $\eps$.


It is truly remarkable that in the originally-motivating context of Turing bifurcation, 
the complicated singular two-parameter stability problem for \eqref{ampeq} in the end yields simple stability
conditions \eqref{tcond} and \eqref{preccond}-\eqref{ccond} analogous to the classic Eckhaus and Benjamin-Feir-Newell
stability conditions for periodic solutions of the complex Ginzburg-Landau equation.
This both justifies the introduction of singular amplitude equations \eqref{ampeq}
and sets the stage for systematic exploration of biological pattern processes like vasculogenesis
to which they apply, starting from initiation at Turing bifurcation. 
{\it The latter program may be
expected to play the same powerful role in exploration of global bifurcation in biomorphogenesis 
as has the corresponding program in the classical pattern-formation case arising in elasticity, 
reaction diffusion, and myriad other settings.}

Many biomorphology models used in applications feature partial parabolicity and mixed hyperbolic parabolic type.
An important further open problem to extend our results for complete parabolicity to this more general domain,
as been done for related settings in \cite{JZN} and references therein.
A second very important followup problem is to actually carry out the above-described program for
specific biomorphology models, systematically determining associated Turing bifurcations and their stability;
that is, studying the practical {\it initiation problem} for periodic pattern formation with conservation laws.
In particular, it would be very interesting to carry out a full analysis
for the example model discussed in Appendix \ref{s:Teg}.
A third, more speculative followup, as suggested in \cite{WZ1,WZ3}, is to study modulation of these patterns
as a possible model for {\it emergent dynamics}; see \cite{JNRZ,MetZ}.

\subsubsection{Relations to thin-film flow}\label{s:thinrel}
More generally, our results in a sense also complete/complement the programs of \cite{JNRZ}, 
\cite{JNRZ2,BJNRZ}, and \cite{BJZ1}.  
The analysis of \cite{JNRZ} completely analyzes stability and asymptotic behavior given spectral diffusive
stability condition of Schneider, while those of \cite{JNRZ2,BJNRZ} verifies Schneider's condition in a 
certain degenerate small-amplitude limit arising in inclined shallow water flow.
It was pointed out in \cite{BJZ1} that nondegenerate Turing type bifurcations can also occur, as a complementary
and in general perhaps more frequent case, and the determination of their stability, even partially or
just heuristically, was cited as an important open problem.
Moreover, it was observed that the numerics associated with spectral stability were quite delicate
and computationally expensive for this problem.

Our present results both resolve this open problem in passing, giving simple conditions for 
diffusive spectral stability and justifying the heuristic approximation \eqref{ampeq},
and explain the observed delicacy of numerics.  For, the singularity $\eps^{-1}$ in \eqref{ampeq}
makes this a stiff system in the $\eps \to 0$ limit, for which the spectrum is inherently difficult to resolve.
Nonetheless, the equations are theoretically well-posed \cite{HSZ,WZ3}; the efficient resolution of the
associated time-evolution problem is thus an interesting open problem.  See \cite{BLWZ} for work in this direction.

In regard to the related earlier work on thin film flow, it is worth noting that rescaling 
$\check x=\eps \hat x$ (slow variable) removes $\eps^{-1}$ from \eqref{Msys}, giving form
$$
\mathcal{U}_t= \tilde M_1(\eps) \mathcal{U}_{\check x} + M_0 \mathcal{U}
+ \eps^2 M_2 \mathcal{U}_{\check x \check x},
$$
where $\tilde M_1(\eps)= \tilde M_1^0 + \eps \tilde M_1^1$, $\tilde M_1^j$ constant, that is,
a singularly perturbed relaxation system with vanishing viscosity $\eps^2\partial_{\check x}^2$,
the form ultimately studied in the two-parameter matrix bifurcation analysis of Section \ref{s:suff}.
It is interesting that a similar relaxation structure was encountered in the studies
\cite{JNRZ2,BJNRZ} of a quite different bifurcation occurring in the small-amplitude
limit for periodic waves in inclined shallow-water flow in the ultra-small frequency regime, 
and a reminiscent two-parameter bifurcation analysis successfully carried out.
See Remark \ref{remrmk} for further discussion.

\subsubsection{The Darcy approximation revisited}\label{s:Darcy}
We emphasize that the bounded-time estimates of \cite{HSZ}, though they apply to more general types of
solution, are not relevant to the special question of Turing patterns and their asymptotic stability,
and so give essentially complementary information. We discuss this point further in Section \ref{s:disc2}.
Likewise, the necessary conditions derived in \cite{WZ3} were theoretical and not
explicitly computed in the generic case considered here.
As we shall see in Section \ref{s:nec}, to go beyond establishing analyticity of neutral spectra
and actually compute the first nonvanishing real parts of the associated Taylor series
costs several further levels of expansion and substantial additional effort.
However, the end result is quite simple, consisting of the conditions for stability 
within the formal Darcy approximation, plus stability of transverse mean modes:
precisely as suggested by the heuristic picture of the Darcy model as an approximate slow manifold.

An interesting open problem related to the approximating manifold approach of \cite{HSZ} is to
extend their results to the generic case that coupling constant $d\neq 0$.  As pointed out
in \cite[\S 4.1]{WZ3}, a corresponding ansatz in that case is $ B= -h |A|^2/f$, canceling the
singular term in \eqref{ampeq}(ii), and giving for $A$ the modified complex Ginzburg-Landau equation
\eqref{Darcy} described in the introduction.
This gives a similar approximation error in the full equations \eqref{ampeq} to that given in \cite{HSZ}
for the complex Ginzburg-Landau ansatz in the mode-filtered equations, suggesting that the latter
should hold for the Darcy approximation in the general case $d\neq 0$ as well.
Likewise there was established in \cite[\S 4.1]{WZ3} a bounded-time/well-posedness result for \eqref{ampeq}
in class $H^s(\R)$, verifying in the generic case $d\neq 0$ the second main ingredient in the analysis
of \cite{HSZ}.

\subsubsection{Whitham modulation, amplitude equations, and the Darcy approximation}\label{s:disc2}
We close with some comments of a general nature contrasting the various results obtained here and in
\cite{HSZ} and \cite{JNRZ}.
Here, we determine diffusive spectral stability in terms of the $m+1$ neutral (translational and
conserved)
modes of the linearized amplitude equations \eqref{ampeq} about periodic waves \eqref{tw}, 
and show that this is equivalent to diffusive stability of bifurcating periodic solutions of \eqref{bl}
for $\eps>0$ sufficiently small.
The results of \cite{JNRZ} give for {\it fixed} $\eps>0$ that low-frequency diffusive spectral stability 
of periodic solutions is equivalent to diffusive stability of the $(m+1)$-dimensional 
{\it Whitham modulation equations} for \eqref{bl}, which in turn implies linearized and nonlinear stability
with respect to $H^s\cap L^1$ perturbations.
The latter estimates, however, are quite $\eps$-dependent, potentially blowing up as $\eps\to 0$.

The Whitham modulation equations, likewise, are associated with the same neutral translational and 
conserved 
modes as considered in our spectral stability analysis. It is natural to conjecture that these should
agree with the reduced equations that we obtain; at least the spectral expansions must agree to lowest
order in appropriate domains of common validity.
It would be very interesting to make this connection precise.
Likewise, a difficult but extremely interesting open problem would be to give a description of behavior of
perturbed Turing patterns like that of \cite{JNRZ} but uniformly valid in $\eps$.

These issues appear also related to the finite-time complex Ginzburg-Landau approximation of \cite{HSZ},
or, more generally, the Darcy approximation conjectured in Section \ref{s:Darcy},
governed asymptotically by a single modulation equation describing translation, or ``phase shift.''
In particular, one may ask whether the Darcy approximation could be not only valid for bounded time
and vanishing $\eps$, but also for small enough $\eps$ and time going to infinity, or both?

In this regard, we recall from \cite{JNRZ} the related result that time-asymptotic behavior
is dominated by phase shift precisely under a certain decoupling condition implying the absence of
a Jordan block in the eigenstructure of neutral modes,
corresponding to failure of our genericity condition \eqref{gencase}.
\be\label{decoup}
\Im(d)=\Re(d) \Im(c)/\Re(c),
\ee
Thus, it is apparently not possible that the Darcy approximation, accounting for phase shift,
can represent time-asymptotic behavior in the standard sense that remaining terms decay at faster
time-asymptotic rate.

We conjecture, rather, that derivative of the phase and transverse, mean modes, decay at the
same diffusive rate $t^{-1/2}$ in $L^\infty$ (see \cite{JNRZ}), but with coefficients of
the latter of order $\eps$, uniformly as $\eps\to 0$. Note that this agrees with the intuition
of Fourier modes with slow decay $e^{-\eps^2 k^2 t}$, corresponding to a heat equation
$u_t=\eps^2 u_{xx}$ decay for $L^1$ initial data at rate $(\eps^2 t)^{-1/2}=\eps^{-1}t^{-1/1}$.
If correct, this would yield that the Darcy approximation governs time-asymptotic behavior under general 
(not just ``prepared'' type) initial perturbations,
a conclusion that would evidently be very useful if true.
The resolution of this and the above related questions appear to be very interesting directions
for future investigation.

\medskip
{\bf Acknowledgement.} 
Thanks to Miguel Rodrigues, Olivier Lafitte, and Benjamin Melinand for their interest in the project and
related stimulating and helpful discussions.
Thanks also to Paul Milewski for helpful feedback and comments regarding framing of the discussion surrounding
\eqref{Darcy}.

\section{Necessity of Eckhaus conditions}\label{s:nec}
We now begin the main task of the paper, analyzing the linear stability problem by
reduction to constant coefficients followed by a 2-parameter matrix perturbation analysis.
In this section, we carry out for fixed $\eps$ a Taylor expansion in $\s$ about $0$, determining 
necessary conditions for stability. In the following section, we examine the remaining regions of $\s$-$\e$ space,
showing that these conditions are in fact sufficient for stability as well.

For clarity, we first carry out the analysis completely in the case $m=1$ of a single conservation law,
indicating after the straightforward 
generalization to the vector case $m\geq 2$.

\subsection{Reduction to constant coefficients}\label{s:redcc}
Following \cite{WZ3}, we perturb $(A,B)=(A_0e^{i(\kappa x-\omega t)},B_0)$ as $A\to (A_0+u+iv)e^{i(\kappa x-\omega t)}$ and $B\to B_0+w$, with $u,v,w$ real valued, giving linearized equations
\eqref{lineq} in constant coefficient form. 
Here, we have used the identity
\be\label{B0ident}
-i\omega U=-a\kappa^2U+bU+cA_0^2U+dB_0U
\ee
for $U:=u+iv$, coming from satisfaction of \eqref{ampeq} by the periodic solution.

To study stability, we compute the dispersion relations associated with \eqref{lineq}.
Namely, writing $(u,v,w)=(u_0,v_0,w_0)e^{i\hat \s x-\l t}$ we arrive at the eigenvalue problem
\be\label{eval}
\lambda (u_0,v_0,w_0)^T= M(\e,\hat\s)(u_0,v_0,w_0)^T,
\ee
where
\ba\label{eq:truncatedmatrix}
M(\e, \hat \s)&=
\bp  2A_0^2\Re(c) & 0& A_0\Re(d)\\ 2A_0^2\Im(c) & 0 & A_0\Im(d)\\ 0 & 0  & 0\ep 
+i\hat \s\bp -2\kappa \Im(a)& -2\kappa \Re(a) & 0\\
+2\kappa \Re(a)& -2\kappa \Im(a) & 0\\ 2A_0h \e^{-1}+2A_0\kappa \Im(g) & 0 & \e^{-1}f\ep\\
&\quad
-\hat \s^2\bp \Re(a)& -\Im(a) &0\\ \Im(a)& \Re(a) &0\\ 2A_0\Re(g) & 2A_0\Im(g) & e_B \ep
=: M_0 + \hat \sigma M_1 + \hat \sigma^2 M_2.
\ea

{\it Spectral stability}, necessary for linearized stability,
corresponds to 
\be\label{ss}
\Re \lambda \leq 0 \; \hbox{\rm for $\lambda=\lambda(\hat \s)\in \spec(M(\e,\hat \s))$.}
\ee
{\it Diffusive spectral stability} in the sense of Schneider \cite{S1,S2}, sufficient for linearized and
nonlinear stability, corresponds to 
\be\label{dss}
\Re \lambda \leq c(\eps)|\hat \sigma|^2/(1+|\hat \sigma|^2) \; \hbox{\rm for 
	$\lambda=\lambda(\hat \s)\in \spec(M(\e,\hat \s))$.}
\ee

\br\label{rmkrmk}
We note the remarkable fact that $B_0$ does not appear in the linearized equations \eqref{lineq},
having been removed using \eqref{B0ident}.
Paradoxically, it thus appears not to enter stability considerations, despite its role through \eqref{B0dom}
in the existence theory.
The resolution of this paradox is that the restriction \eqref{B0dom} on existence
stems from our convention \eqref{supercrit}, which will be seen to be necessary for stability.
But apart from this indirect connection, indeed $B_0$ disappears in stability computations.
\er

\subsection{Taylor expansion of the dispersion relations}\label{s:disp}
Our first task, carried out in the remainder of this section,
is to compute an analytic expansion for the neutral dispersive curves about the origin,
thereby determining {\it necessary conditions} for spectral and diffusive spectral stability.
Our analysis follows to a large extent the corresponding analysis in \cite{WZ3}, in which a
Taylor expansion was shown to exist but not explicitly computed.
Here, we go substantially further, however,
computing the Taylor coefficients to sufficiently high order in $\eps$
to obtain explicit low-frequency stability conditions.

To minimize the number of fractions in the following computation, define two parameters
\be\label{p_q}
p:=-\frac{\Re(d)}{2A_0\Re(c)}\quad\quad q:=-\frac{\Im(c)}{\Re(c)}.
\ee
so that
\be\label{pq}
pq:=\frac{\Re(d) \Im(c)}{2A_0\Re(c)^2}.
\ee

\begin{remark}
	We recall from \cite{WZ3} and \eqref{A0soln} that $p$ can be identified as $$p=\frac{\d A_0}{\d B_0},$$ and that $q$ is normalized measure of the nonlinear dispersion in the $A$-equation. Hence, in the vectorial case $p$ will become a row vector and $q$ will remain a scalar.
\end{remark}

\subsubsection{Preliminary diagonalization}\label{s:prelim}
At $\s=0$, we have the (generalized) eigenvectors of $M(\e,0)$ given by
\be\label{eq:truncatedeigenvectors}
R_s:=\bp 1 \\ -q \\ 0\ep, \qquad
R_t:=\bp 0 \\ 1 \\ 0\ep, \quad\quad R_c:=\bp p \\ 0 \\ 1\ep. 
\ee
In addition, we have corresponding left (generalized) eigenvectors
\be\label{eq:truncatedleft}
L_s:=\bp 1 & 0 & -p \ep, \qquad L_t:=\bp q & 1 & -pq\ep,  \qquad L_c:=\bp 0 & 0 & 1 \ep, 
\ee
where $(L_s,R_s)$ are the left/right eigenvector pair for the unique stable eigenvalue and $(L_t,R_t)$, $(L_c,R_c)$ 
are associated to the zero eigenvalue. 
We note that $L_t$ and $R_c$ are generalized left and right eigenvectors 
respectively, provided that $\Re(c)\Im(d)-\Im(c)\Re(d)\not=0$. 

Following \cite{WZ3}, this gives an initial block-diagonalization
\be\label{M0}
\tilde M_0(\eps):=
\bp L_s\\L_t\\L_c\ep M_0(\eps) \bp R_s & R_t & R_c\ep=
\bp m & 0\\ 0 & \hat M_0(\eps)\ep,
\ee
where
\be\label{m's}
m_0:= 2A_0^2\Re(c), \qquad 
\hat M_0:= 
\bp  0 & A_0 ( \Im(d)+ q \Re(d))\\  0 & 0 \ep.
\ee

Similarly, we find
\ba\label{M1}
\tilde M_1(\eps)&= i\e^{-1}\bp -2A_0ph & 0 & -p(f+2A_0hp)\\
-2A_0pqh & 0 & -pq(f+2A_0hp)\\
2A_0h & 0 & f+2A_0hp\ep\\
&\quad+2i\kappa \bp -\Im(a)-A_0p\Im(g)+q\Re(a) & -\Re(a) & -p(\Im(a)+A_0p\Im(g) )\\
\Re(a)-A_0pq\Im(g)+q^2\Re(a) & -q\Re(a)-\Im(a) & p(-q\Im(a)+\Re(a)-A_0pq\Im(g) )\\
A_0\Im(g) & 0 & A_0\Im(g)p \ep\\
&=\bp m_1 (\e)& s_1(\e)\\ s_2(\e) & \hat M_1(\eps)
\ep,
\ea
where
\ba\label{m0}
m_1(\e) &= i\big( -2\kappa \Im(a) -2A_0\kappa p\Im(g) - 2\eps^{-1}A_0ph + 2\kappa \Re(a) q\big)\\
&=  - 2i\eps^{-1}A_0ph + H.O.sT. \, ,
\ea
\ba\label{hatM1}
\hat M_1(\e) &=i\bp -2\kappa(\Re(a)q+\Im(a) ) & -pq\e^{-1}(f+2A_0hp )+2p\kappa\big(-\Im(a)q+\Re(a)-A_0pq\Im(g)\big)\\
0 & \e^{-1}(f+2A_0hp)+2A_0\kappa\Im(g)p\ep\\
&= i\bp -2\kappa(\Re(a)q+\Im(a) ) & -pq\e^{-1}(f+2A_0hp )\\
0 & \e^{-1}(f+2A_0hp)\ep  + H.O.T.\, ,\\
\ea
\ba\label{s1}
s_1(\e)&=i \bp -2\kappa \Re(a) &-2\kappa\Im(a)p- 2\kappa p A_0\Im(g) p -\eps^{-1}p( f+2A_0hp)\ep\\
&= -i\eps^{-1} \bp 0 & p( f+2A_0hp) \ep +
i\bp -2\kappa \Re(a) & O(1)\ep   \, , 
\ea
and
\ba\label{s2}
s_2(\e)&= 
i \bp -2\e^{-1}A_0pqh+2\kappa\big(\Re(a)-A_0pq\Im(g)+q^2\Re(a) \big) \\ \eps^{-1} 2A_0 h+ 2A_0 \kappa \Im (g))\ep \\ 
&= 
2A_0 i\eps^{-1} \bp    -pqh \\ h \ep + i\bp O(1) \\ 2A_0 \kappa \Im (g))\ep \, .
\ea

\begin{remark}
	In anticipation of the vectorial case, $m\geq 2$, we've carefully placed the $h$'s, $p$'s and $\Im(g)$'s, so that the corresponding expressions make sense when $p$ is an $1\times m$ row vector, $f$ is an $m\times m$ matrix, and $h$, $g$ are $m\times 1$ column vectors.
\end{remark}

At next order, we obtain, likewise,
\be\label{M2}
\tilde M_2(\eps)= -\bp m_2 & X_1\\ X_2 &\hat M_2\ep,
\ee
where
\ba\label{WXYZ}
m_2 &= \bp Re(a)+\Im(a)q -2A_0p(\Re(g)-\Im(g)q) \ep, \\
X_1&=\bp \Im (a) - 2A_0p\Im(g)& p\Re(a)-2A_0p\Re(g)p -pe_B \ep,\\
X_2&=\bp \Im(a)-q^2\Im(a)  -2A_0pq\big( \Re(g)-\Im(g)q\big)\\  2A_0\big(\Re(g) - \Im(g)q\big) \ep, \\
\hat M_2 &= \bp -q\Im(a) + \Re(a) -2A_0pq\Im(g) & \big(q\Re(a)+\Im(a)-2A_0pq\Im(g)\big)p -pqe_B\\
2A_0 \Im(g) & 2A_0 \Re(g)p + e_B \ep.
\ea

\subsubsection{Some organizing discussion}
As described in \cite{Wh,WZ3}, the eigenvector $R_t$ corresponds to a ``translational mode'' coming
from rotational invariance in the complex Ginzburg-Landau equation for $A$,
a consequence of translational invariance in the underlying PDE,
while the eigenvector $R_c$ corresponds to a ``conservative mode'' associated with the conservation
law for $B$.
The mode $R_s$ corresponds to a (weakly) ``stable mode'' coming from supercritical Hopf bifurcation,
as is familiar from the classical complex Ginzburg-Landau equation arising in the case without conservation 
laws.

These classifications may be helpful both in following the analysis, and in connecting to other work.
For example, the neutral $\hat N$ block may be seen to correspond with the more general Whitham
modulation approximation about arbitrary-amplitude waves, which, as described in \cite{JNRZ},
is precisely an expansion in neutral translational and conservative modes.
As noted in \cite{JZ1,JZ2,BJZ2} in the context of the Whitham equations, one finds from 
study of the corresponding existence problem that the lowest-order part $\hat N(0)$ generically 
consists of a nontrivial Jordan block, as we see also here, a consequence of the conservation
structure of the equations. 

Here as there, this leads to much of the difficulty in the 
study of stability; indeed, it is counterintuitive at first glance that spectra should expand
analytically in $\hat \sigma$ instead of in a Puiseaux expansion in $\sqrt{\hat \sigma}$, as we shall see.
The reason for this, as pointed out in \cite{JZ2}, is that the	lower lefthand entry of the next-order 
block 
$\hat N_{1}$ 
opposite to the nonvanishing entry in the Jordan block necessarily vanishes,
also by conservation principles, so that the matrix perturbation problem can be converted by a ``balancing''
transformation to a standard matrix perturbation problem without a nontrivial Jordan block.
The latter procedure is described in Section \ref{s:balance} below.

Alternatively, looking at the exact spectral expansion of neutral modes
for the underlying PDE problem, corresponding to variations along the manifold of periodic traveling-wave
solutions, one finds \cite{JZ2} that there is a Jordan block whenever speed of profiles is nonstationary,
and that the corresponding generalized right zero eigenfunction lies in a Jordan chain over the genuine
eigenfunction consisting of the spatial derivative $\partial_x \bar u$ of the background 
traveling wave.  Meanwhile, associated left eigenfunctions are constant functions $\ell\equiv \const$, 
by conservation form;
the lower lefthand corner of the Jordan block, corresponding to 
$\langle \ell, \partial_x \bar u \rangle= \ell\cdot \bar(x)|_0^X$, $X$ the period
of the background wave, thus vanishes by periodicity of $\bar u$.

That is, {\it the same conservation structure leading to appearance of a Jordan block enforces constraints on
	the perturbation leading to analytic expansion nonetheless}: i.e., ``the disease is also the cure.''
The importance of analytic vs. Puiseaux expansion is that the latter involves ill-conditioning of 
the spectral expansion/diagonalization procedure in the form of blowup of associated eigenprojectors.
In the context of \cite{JZ2}, this would have wrecked the linearized stability estimates, leading to transient
time-algebraic growth instead of decay.  Here, it would have increased sensitivity to higher-order errors,
preventing our estimates from closing.

\subsubsection{First-order diagonalization}\label{s:first}
Following \cite{MZ1,MZ}, we apply the method of ``successive diagonalization'', 
defining a coordinate change 
\ba\label{calT}
\mathcal{T}=\bp 1 & 0\\ \hat \sigma t_2 &\Id_2\ep \bp 1 & \hat \sigma t_1\\ 0 &\Id_2\ep
&=\bp 1  & \hat \sigma t_1\\ \hat \sigma t_2&\Id_2+  \hat \sigma^2 t_2t_1\ep, \\
\mathcal{T}^{-1}= \bp 1 & - \hat \sigma t_1\\0 &\Id_2\ep \bp 1 & 0 \\ - \hat \sigma t_2 &\Id_2\ep &=
\bp 1  +  \hat \sigma^2 t_1 t_2 & - \hat \sigma t_1\\- \hat \sigma t_2&\Id_2\ep,
\ea
with 
$t_1 \hat M_0 - m_0 t_1= -s_1$, $\hat M_0 t_2 - t_2 m_0 = s_2$, or equivalently
\ba\label{t's}
t_1&= -s_1( \hat M_0 - m_0)^{-1} ,\\
t_2  &= (\hat M_0-m_0)^{-1} s_2,
\ea
$\hat M_0$, $m_0$ as in \eqref{M0}, in order to diagonalize to order $O( \hat \sigma)$.
With this choice, we obtain
\be\label{N}
N( \hat \sigma, \eps):= \mathcal{T}\tilde M\mathcal{T}^{-1}=N_0(\eps)+ \hat \sigma N_1(\eps) 
+  \hat \sigma^2 N_2(\eps) +  \hat \sigma^3 N_3( \hat \sigma)  + O( \hat \sigma^4),
\ee
where
\ba\label{dN's}
N_0=\tilde M_0 &= \bp m_0 &0\\ 0& \hat M_0\ep, \qquad
N_1= \bp m_1 &0\\ 0& \hat M_1 \ep, \\
N_2&= \bp * &*\\ *& \hat N_2 \ep, \qquad
N_3= \bp * &*\\ * & \hat N_3\ep, 
\ea
with 
\ba\label{N2}
\hat N_2&= -\hat M_2 + (t_2s_1- s_2t_1) + t_2t_1 \hat M_0 - t_2 m_0 t_1 \\
&= -\hat M_2 + (t_2s_1- s_2t_1) -t_2 s_1 \\
&= -\hat M_2 - s_2t_1 \\
&= -\hat M_2 +  s_2s_1(\hat M_0-m_0 )^{-1} \\
\ea
and
\ba\label{N3}
\hat N_3&=  -t_2 m_1 t_1 + t_2t_1\hat M_1-t_2X_1+X_2t_1= t_2t_1(\hat M_1-m_1)-t_2X_1+X_2t_1\\
&=
-(\hat M_0-m_0)^{-1}s_2s_1(\hat M_0-m_0)^{-1}(\hat M_1-m_1)-t_2X_1+X_2t_1.
\ea
Here, the recurring term $s_2s_1$ by direct computation is
\ba\label{s2s1}
s_2s_1&= 2\eps^{-2}A_0\bp 0 & -qphp(f+2A_0hp) \\0 & hp(f+2A_0hp)\ep + \eps^{-1} \bp -4A_0\kappa\Re(a)pqh & O(1)\\ 
4A_0\kappa h \Re(a) & O(1)\ep.
\ea

Applying standard spectral perturbation theory \cite{K}, we have that there exists an {\it exact}
analytic in $\hat \sigma$ change of coordinates decoupling $n$ and $\hat N$ blocks for which the above approximate
series are correct to $O(\hat \sigma^4)$ error. We thus have that the ``stable'' eigenvalue $\lambda_s$ corresponding
to the $n$ block satisfies
\be\label{muc}
\lambda_s(\eps)= m_0 + O(\hat \sigma)= 2A_0^2\Re(c)+ O(\hat \sigma).
\ee
The ``translational'' and ``conservative'' eigenvalues $\lambda_t$ and $\lambda_c$ can then be determined
by analysis of the reduced matrix perturbation problem
\be\label{hatNprob}
\hat N(\eps, \hat \sigma):= \hat N_0(\eps) + \hat \sigma \hat N_1(\eps)
+ \hat \sigma^2 \hat N_2 (\eps) + \hat \sigma^3 \hat N_3 (\eps).
\ee

Note, by supercriticality condition \eqref{supercrit}, that $m_0=2A_0^2\Re(c)<0$, so
that indeed the ``stable'' eigenvalue is stable.
This observation indirectly justifies our convention in assuming \eqref{supercrit} in the existence
problem, since otherwise the resulting solutions would be {\it exponentially unstable} by \eqref{muc}.

\begin{remark}
	The analogous computation in \cite{WZ3} uses a more traditional spectral perturbation argument. In order to connect this argument to that argument, we note that the first order diagonalization can also be interpreted as finding the first correctors of the left/right eigenvectors $(L_s,R_s)$ associated to the stable eigenvalue $\lambda_s(0)=m_0$, as we recall from \cite{K} that simple eigenvalues have smooth eigenvectors.
\end{remark}

\subsubsection{Balancing transformation}\label{s:balance}
In the generic case 
\be\label{gencase}
\Im(d)\neq \Re(d) \Im(c)/\Re(c),
\ee
$\hat M_0$ takes the form of a (nonzero multiple of a) Jordan block
\ba\label{Jblock}
\hat M_0= rJ; \qquad J=\bp 0 & 1 \\ 0 & 0\ep, \quad r:= 
A_0 (\Im(d)- \Re(d) \Im(c)/\Re(c))\neq 0,
\ea
giving
\ba\label{Jinv}
(\hat M_0-m_0)^{-1}&= \big(-m_0(I-(r/m_0)J)\big)^{-1}= -m_0^{-1}\big(I+(r/m_0)J\big)
\ea
and, by \eqref{s2s1}
\ba \label{N2block}
s_2s_1(\hat M_0-m_0)^{-1}&= -s_2s_1 m_0^{-1}\big(I+(r/m_0)J\big).\\
&=-\frac{1}{m_0}s_2\bp -2i\kappa\Re(a) & -i\e^{-1}p(f+2A_0hp)+O(1)\ep\bp 1 & \frac{r}{m_0}\\ 0 & 1\ep\\
&=-\frac{1}{m_0}s_2s_1+\bp 0 & O(\e^{-1})\\0 & O(\e^{-1})\ep,\\
\ea \label{t1t2}
\ba \label{Mhatinv}
(\hat M_1-m_1)&= i\eps^{-1} 
\bp  2A_0ph & -pq(f+2A_0hp )\\ 0 & f+2A_0h p+2A_0ph \ep  + O(1), 
\ea
and
\ba \label{N3block}
(\hat M_0-m_0)^{-1} & s_2s_1(\hat M_0-m_0)^{-1}(\hat M_1-m_1)=
i\frac{\e^{-2}}{m_0^2}\bp 1 & r/m_0\\ 0 & 1\ep\\
&\quad\times \bp -4A_0\kappa\Re(a)pqh & -2A_0qphp(f+2A_0hp)\e^{-1} \\4A_0\kappa h \Re(a) & 2A_0hp(f+2A_0hp)\e^{-1}\ep\\
&\quad\times \bp 1 & r/m_0\\ 0 & 1 \ep\bp  2A_0ph & -pq(f+2A_0hp)\\ 0 & f+2A_0hp+2A_0ph \ep+O(\e^2)\\
&\quad=i\frac{\e^{-2}}{m_0^2}\bp O(\e^{-1}) & O(\e^{-1})\\ 8A_0^2\kappa\Re(a)hph & O(\e^{-1})\ep+ H.O.T.
\ea

\br
The distinction between $2A_0hp$ and $2A_0ph$ in \eqref{Mhatinv} is being kept in order to more easily adapt this calculation to the systems case.
\er

We note here that the crucial fact that makes this entire calculation of dispersion relations possible is that, while $\hat N_3$ is of order $\e^{-3}$, $\hat N_3$ is upper triangular to leading order. Hence the bottom left corner of $\hat N_3$ is of order $\e^{-2}$, matching the order of $\hat N_2$. The importance of this observation comes from the ``balancing'' transformation which informally takes the top right corner and moves it up an order and similarly takes the bottom left corner and moves it down an order. We also note here that \eqref{Mhatinv} is upper triangular, and that the $O(1)$ terms are proportional to $\kappa$.\\

Thus, 
\ba\label{N1again}
\hat N_1=\hat M_1= i\bp -2\kappa(\Re(a)q+\Im(a) ) & -pq\e^{-1}(f+2A_0hp)\\
0 & \e^{-1}(f+2A_0hp)\ep  + H.O.T.\, ,\\
\ea
\ba\label{comp2}
\hat N_2 &= -\hat M_2 +  s_2s_1(\hat M_0-m_0 )^{-1} \\
&= -m_0^{-1}\e^{-1}\bp -4A_0\kappa\Re(a)pqh & -2A_0qphp(f+2A_0hp)\e^{-1} \\4A_0\kappa h \Re(a) & 2A_0hp(f+2A_0hp)\e^{-1}\ep+ H.O.T.\ ,
\ea
and
\ba\label{comp3}
\hat N_3 &= -(\hat M_0-m_0)^{-1}s_2s_1(\hat M_0-m_0)^{-1}(\hat M_1-m_1 )-t_2X_1+X_2t_1
= \bp * & *\\ O(\eps^{-2}) & * \ep.
\ea

\br\label{noJordanrmk}
Evidently, for the decoupled case $d=0$ considered in \cite{HSZ}, \eqref{gencase} fails,
hence that case is degenerate from this point of view.
The case that \eqref{gencase} fails but $d\neq 0$ is treated in \cite{WZ3}.
\er

To remove the Jordan block, following \cite{MZ1,MZ}, we perform the ``balancing'' transformation 
$\hat N\to O:= \mathcal{S}\hat N\mathcal{S}^{-1}$, where 
\be\label{bal}
\mathcal{S}:= \bp i\hat \sigma & 0\\ 0 & 1 \ep, \quad
\mathcal{S}^{-1}= \bp (i\hat \sigma)^{-1} & 0\\ 0 & 1 \ep,
\ee
yielding
\be\label{O}
O(\eps,\hat \sigma)= \hat \sigma O_1 + \hat \sigma^2 O_2 + O(\hat \sigma^3),
\ee
with 
\ba\label{O's}
O_1 &= i \bp 
-2\kappa(\Re(a)q+\Im(a) ) & r\\
4\kappa  \eps^{-1}m_0^{-1}  A_0 h\Re(a) & \eps^{-1}( f+2A_0hp)\ep  + H.O.T. \\
&=
i \bp  0 & 0\\ 4\kappa m_0^{-1}\eps^{-1}  A_0 h\Re(a) & \eps^{-1}(f+2A_0hp)\ep  + H.O.T.\, ,\\
O_2 &= \bp 4\e^{-1}m_0^{-1}A_0\kappa \Re(a)pqh & pq \eps^{-1} (f+2A_0hp )\\
-8\e^{-2}m_0^{-2}A_0^2\kappa (hph\Re(a)) & -\eps^{-2}m_0^{-1} 2A_0 hp( f+2A_0hp)\ep+H.O.T.
\ea

This can now be expanded up to order $\hat \sigma^2$ by standard perturbation of distinct eigenvalues, as
the eigenvalues of $O_1$ are $\sim 1, \eps^{-1}$ and split for $\eps>0$ sufficiently small under the 
generic condition
\be\label{fluxcond}
f+ 2A_0h p \neq 0.
\ee
Moreover (by splitting), we may conclude {\it analyticity in $\hat \sigma$} on some sufficiently small ball.

Namely, taking the eigenvectors to lowest order in $\eps$ of $O_1$, we have 
\ba\label{desing}
\ell_t&= \bp 1 & 0\ep + O(\eps),\quad \ell_c= \bp 4\kappa m_0^{-1}\Re(a) A_0 (f+2A_0hp)^{-1}h  & 1\ep+ O(\eps)\\
r_t&= \bp 1 \\ -4\kappa m_0^{-1} A_0 \Re(a)(f+2A_0hp)^{-1}h  \ep + O(\eps),\quad r_c=\bp 0 \\ 1\ep + O(\eps),
\ea
giving the analytic expansion
\ba\label{lamex}
\lambda_t&= \alpha_t \hat \sigma + \mu_t \hat \sigma^2 + O(\hat \sigma^3),\\
\lambda_c&= \alpha_c \hat \sigma + \mu_c \hat \sigma^2 + O(\hat \sigma^3),
\ea
where $\alpha_j$ are pure imaginary
\ba\label{alpha's}
\alpha_t&= \ell_t O_1  r_t=  O(1)  \, ,\\ 
\alpha_c&= \ell_c O_1  r_c=  i \eps^{-1}( f+2A_0hp) +O(1),
\ea
are {\it pure imaginary} and
\ba\label{genmu's}
\mu_t&= \ell_t O_2  r_t= O(\eps^{-1})=: \eps^{-1} \mu_t^0 + O(1)  \, ,\\ 
\mu_c&= \ell_c O_2  r_c= -2\eps^{-2} m_0^{-1}A_0  hp( f+2A_0hp) +O(\eps^{-1}) =: \eps^{-2} \mu_c^0 + H.O.T. \, .
\ea
Thus, {\it the signs of the real parts of $\lambda_t$ and $\lambda_c$ are determined
	by the signs of $\Re\mu_t$ and $\Re \mu_c$}.\\

\br\label{splitrmk}
In the vectorial case $m>1$, the lower righthand block of $O_1$ is an $m\times m$ matrix with
leading order $\eps^{-1}( f+2A_0hp)i$, hence requires the additional, first-order stability condition
\eqref{preccond} in order to ensure that its spectra are pure imaginary. The second order stability conditions also require modification, see Section \ref{s:ext} for details.
\er

\br
In the systems setting, $\ell_c\to\ell_{c,i}$ with first entry corresponding to the $i$-th entry of the vector $4\kappa m_0^{-1}\Re(a) A_0 (f+2A_0hp)^{-1}h$ and $1$ replaced with the $i$-th standard basis of $\RR^m$, thought of as a row vector. Similarly, $r_c\to r_{c,i}$ with $1$ replaced with the $i$-th standard basis element of $\RR^m$, now thought of as a column vector. Finally, we note that to complete the eigenvectors we have that $\ell_t\to \bp 1 & 0\ep$ and $r_t\to \bp 1 \\ -4\kappa m_0^{-1}\Re(a) A_0 (f+2A_0hp)^{-1}h\ep$.
\er

However, our expansion is not fine enough to determine $\mu_t$, or even its order, i.e., whether
$\mu_t^0$ is vanishing or nonvanishing.
This could be remedied by applying the method of repeated diagonalization in powers of $\eps$ to
expand \eqref{desing} to as many orders as required to obtain a nonvanishing coefficient in the expansion
of $\mu_t$. In \cite{WZ3}, $\mu_t$ was computed through a Kato-style expansion when $r=0$, i.e. $\hat{M}_0$ is not a Jordan block. In that case, $\mu_t^0=0$, or equivalently, $\mu_t(\e)=O(1)$.\\

While a similar expansion is possible here, it is significantly more complicated due to $\a_t$ and $\mu_t$ being primarily determined by error terms in the $O_j$. Instead, we elect to work with \eqref{eq:truncatedmatrix} directly and use the method of matched asymptotics. It is at this point that our argument specializes to the case of a single conservation law, see Section \ref{s:sysmatched} for a variation on this argument for the case $m\geq 2$. Suppose that \eqref{eq:truncatedmatrix} has an eigenvalue of the form $\l_t(\hat{\s})=i\a_t\hat\s+\mu_t\hat\s^2$ with $\a_t=O(1)$ and $\mu_t=O(\e^{-1})$. Recalling $M(\e,\hat\s)$ from \eqref{eq:truncatedmatrix}, we find
\begin{equation*}
	\begin{aligned}
		M(\e,\hat\s)&=-\hat\s^2\bp \Re(a)& -\Im(a) &0\\ \Im(a)& \Re(a) &0\\ 2A_0\Re(g) & 2A_0\Im(g) & e_B \ep
		+i\hat\s\bp -2\kappa \Im(a)& -2\kappa \Re(a) & 0\\
		+2\kappa \Re(a)& -2\kappa \Im(a) & 0\\ 2A_0h \e^{-1}+2A_0\kappa \Im(g) & 0 & \e^{-1}f\ep\\
		&\quad+\bp  2A_0^2\Re(c) & 0& A_0\Re(d)\\ 2A_0^2\Im(c) & 0 & A_0\Im(d)\\
		0 & 0  & 0\ep.
	\end{aligned}
\end{equation*}
Let us make the important observation that because $\l_t=O(\hat\s)$, the second column and third row of $M(\e,\hat\s)-\l(\hat\s)Id$ are both proportional to $i\hat\s$, with the intersection being order $\hat\s^2$, and so can be divided out. Hence, to determine $\a_t$, we need to work with the now $O(1)$ terms in the characteristic polynomial given by
\be\label{matchedO1}
P_0:=\det\bp 2A_0^2\Re(c) & -2\kappa\Re(a) & A_0\Re(d)\\ 2A_0^2\Im(c) & -2\kappa\Im(a)-\a_t & A_0\Im(d)\\ 2A_0h\e^{-1}+2A_0\kappa\Im(g) & 2A_0\Im(g) & \e^{-1}f-\a_t\ep.
\ee
Expanding $P_0$ along the first column gives us
\ba\label{P0exp}
P_0&=2A_0^2\Re(c)\a_t^2+\big(\e^{-1}2A_0^2(h\Re(d)-f\Re(c))+O(1) \big)\a_t\\
&\quad+\e^{-1}4A_0^2\kappa\big(f\Re(a)\Im(c)-f\Im(a)\Re(c)+h\Im(a)\Re(d)-h\Re(a)\Im(d) \big)+O(1)\\
&\quad=:c_2\a_t^2+c_1\a_t+c_0.
\ea
Now, if $\l_t$ is to be an eigenvalue, we need $P_0=0$. Applying the quadratic formula to \eqref{P0exp}, we find that the two roots of $P_0$ are
$$
\a=\frac{-c_1\pm\sqrt{c_1^2-4c_2c_0}}{2c_2}.
$$
As $c_2=O(1)$ and $c_1,c_0=O(\e^{-1})$, we see that the $c_1^2$ term in the square root dominates. Hence applying the binomial theorem we find the two roots are
$$
\a_-=-\frac{c_1}{c_2}+O(1),\qquad \a_+=-\frac{c_0}{c_1}+O(\e).
$$

Recalling the expressions in \eqref{P0exp}, we find that
$$
\a_-=\e^{-1}\Big(f-\frac{h\Re(d)}{\Re(c)} \Big)+O(1).
$$
Upon recalling the definition of $p$ from \eqref{p_q}, we notice that $\a_-=\a_c$ for $\a_c$ as in \eqref{alpha's}. Evidently, we then conclude that $\a_+=\a_t$. Returning to \eqref{P0exp}, we find that
\be\label{alphatmatched}
\a_t=-2\kappa\frac{f\Re(a)\Im(c)-f\Im(a)\Re(c)+h\Im(a)\Re(d)-h\Re(a)\Im(d)}{h\Re(d)-f\Re(c)}+O(\e).
\ee
On the other hand, using \eqref{desing}, \eqref{O's}, and a Kato-style expansion, we can also find
\be\label{alphatkato}
\a_t=-2\kappa(\Re(a)q+\Im(a))-\frac{4\kappa A_0\Re(a)}{m_0(f+2A_0hp)}rh+O(\e).
\ee
To get from \eqref{alphatmatched} to \eqref{alphatkato}, we start by observing that
$$
h\Re(d)-f\Re(c)=-\Re(c)\big(f+2A_0hp).
$$
We then see that the two central terms in the numerator of \eqref{alphatmatched} are $\Im(a)$ times the denominator and so we obtain
$$
\a_t=-2\kappa\Big(\Im(a)-\Re(a)\frac{f\Im(c)-h\Im(d) }{\Re(c)(f+2A_0hp)}  \Big)+O(\e).
$$
In the numerator, we now add and subtract $2A_0hp\Im(c)$ and cancel with the denominator to obtain
$$
\a_t=-2\kappa\Big(\Im(a)-\Re(a)\frac{\Im(c)}{\Re(c)}+h\Re(a)\frac{\Im(c)2A_0p+\Im(d)}{\Re(c)(f+2A_0hp)}  \Big)+O(\e).
$$
Now, plugging in the value of $p$ and recalling the definition of $r$, we come to
$$
\a_t=-2\kappa\Big(\Im(a)-\Re(a)\frac{\Im(c)}{\Re(c)}+h\Re(a)\frac{r}{2A_0\Re(c)(f+2A_0hp)} \Big)+O(\e).
$$
We note that upon plugging in $m_0=2A_0^2\Re(c)$ that this expression matches \eqref{alphatkato}.

Turning to $\mu_t$, we start by finding the coefficient, $P_1$, of $i\hat\s$ in $(i\hat\s)^{-2}\det(M(\e,\hat\s)-\l_t Id )$. Using multilinearity of $\det$, we find by choosing the $O(\hat\s)$ terms in each column and the $O(1)$ terms in the remaining columns
\ba\label{matchedOs}
P_1&=\det\bp -2\kappa\Im(a)-\a_t & -2\kappa\Re(a) & A_0\Re(d)\\ 2\kappa\Re(a) & -2\kappa\Im(a)-\a_t & A_0\Im(d)\\ 2A_0\Re(g) & 2A_0\Im(g) & \e^{-1}f-\a_t\ep+\det\bp 2A_0^2\Re(c) & -\Im(a) & A_0\Re(d)\\ 2A_0^2\Im(c) & \Re(a)+\mu_t & A_0\Im(d) \\2A_0h\e^{-1}+2A_0\kappa\Im(g) & 0 & \e^{-1}f-\a_t\ep\\
&\quad+\det\bp 2A_0^2\Re(c) & -2\kappa\Re(a) & 0\\ 2A_0^2\Im(c) & -2\kappa\Im(a) & 0 \\ 2A_0h\e^{-1}+2A_0\kappa\Im(g) & 2A_0\Im(g) & e_B+\mu_t\ep.
\ea

The first determinant in \eqref{matchedOs} is to leading order in $\e$ given by
\be\label{P1const}
\det\bp -2\kappa\Im(a)-\a_t & -2\kappa\Re(a) & A_0\Re(d)\\ 2\kappa\Re(a) & -2\kappa\Im(a)-\a_t & A_0\Im(d)\\ 2A_0\Re(g) & 2A_0\Im(g) & \e^{-1}f-\a_t\ep=\e^{-1}f\Big((-2\kappa\Im(a)-\a_t )^2+4\kappa^2\Re(a)^2 \Big)+O(1).
\ee

For the second determinant in \eqref{matchedOs}, we obtain
\ba\label{P1mut}
\det\bp 2A_0^2\Re(c) & -\Im(a) & A_0\Re(d)\\ 2A_0^2\Im(c) & \Re(a)+\mu_t & A_0\Im(d) \\2A_0h\e^{-1}+2A_0\kappa\Im(g) & 0 & \e^{-1}f-\a_t\ep&=\mu_t\big(2A_0^2\e^{-1}(\Re(c)f-h\Re(d) ) +O(1)\big)\\
&\quad+2\e^{-1}A_0^2\Im(a)\big(f\Im(c)-h\Im(d) \big)\\
&\quad+2\e^{-1}A_0^2\Re(a)\big(\Re(c)f-h\Re(d) \big)+O(1).
\ea
The third determinant is linear in $\mu_t$ and $O(1)$, so we will not need it. So, setting $P_1=0$, we obtain a linear equation of the form
$$
c_3\mu_t+c_4=0,
$$
with $c_3,c_4=O(\e^{-1})$. Applying \eqref{P1const} and \eqref{P1mut}, we are led to
\ba\label{mut}
\mu_t&=-\frac{f\Big((-2\kappa\Im(a)-\a_t )^2+4\kappa^2\Re(a)^2 \Big)+2A_0^2\Im(a)\big(f\Im(c)-h\Im(d)\big)+2A_0^2\Re(a)\big(f\Re(c)-h\Re(d)\big) } { 2A_0^2(\Re(c)f-h\Re(d) )}\\
&\quad+O(\e).
\ea

\br
The same kind of calculation can also be used to find $\mu_c$. Indeed, as $\a_c\sim\e^{-1}$, one replaces \eqref{P1const} with
$$
(\e^{-1}f-\a_c)\a_c^2+O(\e^{-2}).
$$
We must also adapt \eqref{P1mut} to
$$
\mu_c2A_0^2\Big(\Re(c)\big(\e^{-1}f-\a_c \big)-\e^{-1}h\Re(d) \Big)+H.O.T.
$$
We notice that this expression vanishes to leading order. In order to match the $O(\e^{-3})$ term 
from \eqref{P1const}, we consider the third matrix in \eqref{matchedOs}, which has an $O(\e^{-1})$ coefficient of $\mu_c$ coming from the third determinant in \eqref{matchedOs}, which gives us
$$
-2A_0^2\Re(c)\a_c \mu_c+ H.O.T.
$$
At this point, we then obtain
$$
\mu_c=\frac{(\e^{-1}f-\a_c)\a_c}{2A_0^2\Re(c) }+O(\e^{-1}).
$$
Plugging in the expression for $\a_c$, we obtain
$$
\mu_c=\e^{-2}\frac{\Re(d)h}{2A_0^2\Re(c)^2 }\Big(f-\frac{\Re(d)h}{\Re(c)} \Big)+O(\e^{-1}),
$$
which upon factoring a copy $f/\Re(c)$ out of the expression in parentheses yields
$$
\mu_c=\e^{-2}\frac{\Re(d)hf}{2A_0^2\Re(c)^3}\Big(\Re(c)-\frac{\Re(d)h}{f}\Big)+O(\e^{-1}), 
$$
which agrees with our previous expression.

\er

Let us make a few key observations about \eqref{mut}. The first is that, to leading order, $\mu_t$ is a function of $\kappa^2$ as $\a_t$ in \eqref{alphatmatched} is to leading order proportional to $\kappa$. Second, $\mu_t$ is, to leading order, {\it independent} of $e_B$ and $g$. Third, at $\kappa=0$, $\mu_t$ reduces to
\be\label{BFNmut}
\mu_t=-\Re(a)-\Im(a)\frac{\Im(\hat c)}{\Re(\hat c)},
\ee
where $\hat c=c-dh/f$ is the updated value of $c$ in the Darcy reduction. Hence the corresponding Benjamin-Feir-Newell criterion is the same as that for the Darcy reduction provide that $\Re(\hat b)>0$.\\

Let us start by finding necessary stability criteria, starting with $\mu_c$. We observe in the scalar case that the sign of $\mu_c^0$ is {\it independent} of $\kappa$, as by \eqref{genmu's}, \eqref{p_q}, \eqref{m0}, we see that
\be\label{mu0cexpanded}
\mu_c^0=-2m_0^{-1}A_0  hp( f+2A_0hp)=\frac{\Re(d)hf}{2A_0^2\Re(c)^3}\Big(\Re(c)-\frac{\Re(d)h}{f} \Big),
\ee
which exactly matches the corresponding expression in \cite{WZ3}. Indeed, this leads us to half of the stability criterion for $\mu_c$, extending Lemma 5.6 of \cite{WZ3} whose proof we also recall.
\begin{lemma}\label{lemma5.6}
	If $\Re(\hat c)>0$ with $\hat c=c-dh/f$, then $\mu_c>0$.
\end{lemma}
\begin{proof}
	If $\Re(\hat c)>0$, then $\Re(d)h/f<0$ since $\Re(c)<0$. Hence in the right most expression of \eqref{mu0cexpanded}, the term in parentheses is positive and the term outside parentheses is also positive since $\Re(d)h/f$ and $\Re(d)hf$ have the same sign and under the assumption $\Re(\hat c)>0$ has the same sign as $\Re(c)^3$.
\end{proof}
Let us proceed further to find necessary and sufficient criteria for $\mu_c^0<0$.
\begin{proposition}\label{mucconditions}
	There holds $\mu_c^0<0$ if and only if
	\be\label{dhfcond}
	\Re(c)<\frac{\Re(d)h}{f}<0.
	\ee
\end{proposition}
\begin{proof}
	The lower inequality in \eqref{dhfcond} was established in Lemma \ref{lemma5.6}. For the other inequality, assume $\Re(\hat c)<0$ and so $\mu_c^0<0$ if and only if $\Re(d)hf<0$ as $\Re(\hat c)/\Re(c)^3>0$ under the assumption $\Re(\hat c)<0$. This completes the proof.
\end{proof}

To obtain Eckhaus-style criteria from \eqref{mut}, we begin by multiplying and dividing by $f$. After doing so and using $\hat c=c-dh/f$, we obtain
\be\label{mut2}
\mu_t=-\frac{\Big((-2\kappa\Im(a)-\a_t)^2+4\kappa^2\Re(a)^2 \Big)+2A_0^2\Im(a)\Im(\hat c)+2A_0^2\Re(a)\Re(\hat c)}{2A_0^2\Re(\hat c)}+O(\e).
\ee
To determine the sign $\mu_t$, we observe that we can freely drop the $2A_0^2$ from the denominator as it is always positive. By Lemma \ref{lemma5.6}, we can further assume that $\Re(\hat c)<0$ as otherwise all waves are unstable. So for stability, we're left with determining for what $\kappa$
\be\label{eckhaus1}
\Big((-2\kappa\Im(a)-\a_t)^2+4\kappa^2\Re(a)^2 \Big)+2A_0^2\Im(a)\Im(\hat c)+2A_0^2\Re(a)\Re(\hat c)<0,
\ee
holds. To continue, we do the same to $\a_t$ as in \eqref{alphatmatched}, where we obtain
$$
\a_t= 2\kappa\Big(-\Im(a)+\Re(a)\frac{\Im(\hat c)}{\Re(\hat c)} \Big).
$$
Recalling the definition of $A_0^2$ from \eqref{A0soln}, we find that \eqref{eckhaus1} can be written as
\be\label{eckhaus2}
4\kappa^2\Re(a)^2(1+\hat q^2)-2\frac{\Re(\tilde b)-\Re(a)\kappa^2}{\Re (c)}(\Im(a)\Im(\hat c)+\Re(a)\Re(\hat c) )<0,
\ee
where $\hat q:=-\Im(\hat c)/\Re(\hat c)$. Collecting the $\kappa^2$ terms on the left hand side and the remainder on the right hand side, we find
\be\label{eckhaus3}
\kappa^2\Big(4\Re(a)^2(1+\hat q^2)+2\frac{\Re(a)}{\Re(c)}\big(\Im(a)\Im(\hat c)+\Re(a)\Re(\hat c) \big) \Big)<2\frac{\Re(\tilde b)}{\Re(c)}(\Im(a)\Im(\hat c)+\Re(a)\Re(\hat c) ).
\ee
For us, the Benjamin-Feir-Newell criterion is that \eqref{BFNmut} is negative. Or, after rearranging,
$$
\Re(a)+\Im(a)\frac{\Im(\hat c)}{\Re(\hat c)}>0.
$$
As we've assumed $\Re(\hat c)<0$, we then obtain
$$
\Re(a)\Re(\hat c)+\Im(a)\Im(\hat c)<0.
$$
Moreover, since we've assumed $\Re(c)<0$, $\Re(a)>0$, and $\Re(\tilde b)>0$, we find that the coefficients in \eqref{eckhaus3} are positive.

\begin{lemma}
	Let $\kappa_S^2$ be defined by
	\be\label{kappastab}
	\kappa_S^2:=\frac{2\frac{\Re(\tilde b)}{\Re(c)}(\Im(a)\Im(\hat c)+\Re(a)\Re(\hat c) )}{ 4\Re(a)^2(1+\hat q^2)+2\frac{\Re(a)}{\Re(c)}\big(\Im(a)\Im(\hat c)+\Re(a)\Re(\hat c) \big)}. 
	\ee
	Then $\kappa_S^2\leq\kappa_E^2$ for $\kappa_E^2$ as in \eqref{solndom} provided the Benjamin-Feir-Newell criterion \eqref{BFNmut} holds.
\end{lemma}
\begin{proof}
	In \eqref{eckhaus3}, we observe that the left hand side is bounded from below by
	$$
	\kappa^2\Big(2\frac{\Re(a)}{\Re(c)}\big(\Im(a)\Im(\hat c)+\Re(a)\Re(\hat c) \big) \Big),
	$$
	as $4\Re(a)^2(1+\hat q^2)\kappa^2$ is readily seen to be positive. As $\kappa_S^2$ is the right hand side of \eqref{eckhaus3} divided by the left hand side, we conclude that
	$$
	\kappa_S^2<\frac{2\frac{\Re(\tilde b)}{\Re(c)}(\Im(a)\Im(\hat c)+\Re(a)\Re(\hat c) )}{2\frac{\Re(a)}{\Re(c)}\big(\Im(a)\Im(\hat c)+\Re(a)\Re(\hat c) \big) }=\kappa_E^2.
	$$
\end{proof}

\begin{remark}\label{symmrmk}
	One can expect that $\mu_t$ and $\mu_c$ are real based on an observation which played a key role in \cite{WZ3} in showing that the wave could not destabilize to leading order. More precisely, each branch of the dispersion relation satisfies
	\begin{equation}\label{realsymmetry}
		\overline{\lambda(-\sigma)}=\lambda(\sigma),
	\end{equation}
	near $\sigma=0$. Hence by \eqref{realsymmetry} and the chain rule, all even derivatives of $\lambda$ are real at $\sigma=0$ and all odd derivatives of $\lambda$ are pure imaginary at $\sigma=0$.
\end{remark}

\begin{definition}\label{astabdef}
	We define \emph{asymptotic diffusive stability} by the pair of conditions 
	\be\label{assdiffstab}
	\hbox{\rm $\Re \mu_t^0>0$, $\Re \mu_c^0>0$, with $\mu_j$ as in \eqref{genmu's}.}
	\ee
	We define \emph{asymptotic instability} by failure of {\it asymptotic neutral stability}
	\be\label{assneutralstab}
	\hbox{\rm $\Re \mu_t^0\geq 0$, $\Re \mu_c^0\geq 0$,}
	\ee
	i.e., $\Re \mu_t^0 < 0$ or $\Re \mu_c^0 < 0$.\footnote{
		As $\mu_t^0$ and $\mu_c^0$ are real in this setting, these are sign conditions for $\mu_t^0$ and $\mu_c^0$ themselves.}
\end{definition}

Evidently, asymptotic diffusive stability is necessary for strict stability, 
$\Re \lambda(\eps, \delta)<0$ for all $\eps>0$, $\delta\neq 0$, while asymptotic instability
is sufficient for instability,
$\Re \lambda(\eps, \delta)>0$ for some $\eps>0$, $\delta \in \R$.

\br\label{couplermk}
In the decoupled case $d=0$ treated in \cite{HSZ}, we see that $\mu_c=0$ and diffusion is at the lower order 
$e_B$. That is, the $\eps^{-2}$ diffusion rate in the generic case is an example of ``convection-enhanced diffusion''.
\er 

\br\label{relaxrmk}
The key quantity $f+ 2A_0h p$ in \eqref{fluxcond} may be regarded as an effective convection obtained by Chapman-Enskog type reduction; note
that our equations have the form of a relaxation system.
The condition that $f+2A_0h p$ be nonzero is equivalent to noncharacteristicity of this effective convection.
Indeed, by the desingularizing rescaling $\delta\to \tilde \delta$, 
we are effectively converting the singular truncated model to a {\it singularly perturbed} model
\ba\label{spert}
A_t &= bA + c|A|^2 A +dAB \eps A_{xx},\\
B_t  + (f-h|A|^2)_x &=  \eps(\Re(gA\bar A_x)_x +\eps e_B B_{xx}
\ea
consisting of an inviscid limit of a relaxation system, 
similar in form to the viscous Saint Venant equations treated in \cite{BJNRZ}.
This may further explain points of similarity in the analyses remarked elsewhere.
We note in particular the remarkably simple form of $\mu_c$ as proportional to the effective convection,
similar to a key relation found in \cite{JNRYZ}, so that {\it second-order behavior in the ``mean mode'' 
	c may be deduced by study of more easily accessible first-order behavior in $\hat \sigma$}, e.g., 
by the first-order Whitham modulation approximation.
\er

\br\label{degenrmk}
The degenerate case $\Im(d)=q\Re(d)$ treated already in \cite{WZ3,Wh} may be treated by the above 
argument equally well, that is, this covers both degenerate and generic case in a common framework.
However, more terms must be included in deriving asymptotics.
\er

\section{Necessity of Darcy conditions}\label{s:Dnec}
We next establish necessity of the Darcy stability conditions $\Re(\hat \sigma)<0$,
$\Re \mu_t^0<0$, determining stability of periodic solutions with respect to the reduced Darcy system 
\eqref{Darcy}, for stability with respect to the full system \eqref{ampeq}.

\subsection{Relation of linearizations}\label{s:rellin}
Our first step is to observe the relation between the linearizations of \eqref{ampeq} and \eqref{Darcy}
after reduction to constant coefficients.
Recall that the linearization of the full system \eqref{ampeq}, reduced to constant coefficients, is
\ba\label{eq:truncatedmatrix2}
M(\e,\s)&=-\s^2\bp \Re(a)& -\Im(a) &0\\ \Im(a)& \Re(a) &0\\ 2A_0\Re(g) & 2A_0\Im(g) & e_B \ep
+i\s\bp -2\kappa \Im(a)& -2\kappa \Re(a) & 0\\
+2\kappa \Re(a)& -2\kappa \Im(a) & 0\\ 2A_0h \e^{-1}+2A_0\kappa \Im(g) & 0 & \e^{-1}f\ep\\
&\quad+\bp  2A_0^2\Re(c) & 0& A_0\Re(d)\\ 2A_0^2\Im(c) & 0 & A_0\Im(d)\\
0 & 0  & 0\ep =: M_0 + \hat \sigma M_1 + \hat \sigma^2 M_2.
\ea
Likewise (see, e.g., \cite{WZ2}), the linearization of the Darcy model \eqref{Darcy}, 
reduced to constant coefficients, is
\ba\label{eq:darcysys}
m(\e,\s)&=-\s^2\bp \Re(a)& -\Im(a) \\ \Im(a)& \Re(a) \ep
+i\s\bp -2\kappa \Im(a)& -2\kappa \Re(a) \\ +2\kappa \Re(a)& -2\kappa \Im(a) \ep\\
&\quad+\bp  2A_0^2\Re(\hat c) & 0 \\ 2A_0^2\Im(\hat c) & 0 \ep
=: m_0 + \hat \sigma m_1 + \hat \sigma^2 m_2.
\ea

Note that the linearization of \eqref{Darcy} is the linearization of the first two equations of the full model
\eqref{ampeq} subject to relation  
$$
B(A)= B_0 + f^{-1} h |A_0|^2 - f^{-1} h |A|^2=  B_0 + f^{-1} h |A_0|^2 - f^{-1} h(U^2+V^2)
$$
of \eqref{Dansatz}. By the chain rule, this is equivalent to
$$
\bp \Id_2 & 0\ep L(\e,\s,\hat x) \bp \Id_2\\ dB/d(U,V)\ep, \qquad
dB/d(U,V)|_{(\bar U,\bar V)}= \bp -2\bar U f^{-1}h& -2\bar Vf^{-1}h\ep, 
$$
where $L$ is the linearization about the periodic wave in the full model \eqref{ampeq}.
As the same exponential coordinate transformation taking $L$ to constant-coefficient takes $dB/d(U,V)$
to a constant-coefficient right multiplier $\bp \Id_2 & N\ep$ canceling the singularity in the third equation
of the full system \eqref{eq:truncatedmatrix}, we may conclude that
\be\label{Dfrel}
\hbox{\rm $m(\e,\s)=\bp \Id_2 & 0\ep M(\e,\s)\bp \Id_2\\ N\ep$, \; with $N=\bp -f^{-1}2A_0 h  & 0\ep$,}
\ee
or 
\ba\label{Dver}
m(\e,\s)&=-\s^2\bp \Re(a)& -\Im(a) \\ \Im(a)& \Re(a) \ep
+i\s\bp -2\kappa \Im(a)& -2\kappa \Re(a) \\ +2\kappa \Re(a)& -2\kappa \Im(a) \ep\\
&\quad+\bp  2A_0^2\Re( c) -f^{-1}2A_0^2h \Re(d) & 0 \\ 2A_0^2\Im( c)-f^{-1}2A_0^2h \Im(d) & 0 \ep.
\ea

Alternatively, comparing \eqref{Dver} against \eqref{eq:darcysys}, and using 
relation $\hat c= c- df^{-1}h$ from \eqref{rels}, we may verify \eqref{Dfrel} by direct computation.

\subsection{Matrix perturbation expansion}\label{s:mpx}
We complete our study by a matrix perturbation analysis of $M$.
Specifically, taking a matrix perturbation point of view, we observe that the change of coordinates
$M\to S^{-1}MS$ with
\be\label{S}
S= \bp \Id_2 & 0\\ N & \Id_m\ep
\ee 
block diagonalizes the singular portion
$ \eps^{-1} \bp 0 & 0 & 0\\ 0 & 0 & 0\\ 2A_0 h & 0 & f\ep $ of $M$, giving
\be\label{Dbd}
S^{-1}M(\e, \s)S= \bp m(\e, \s) & 0\\ 0_2 & i\sigma \eps^{-1}f\ep + \bp 0_2 & O(1)\\O(1) & O(1)\ep. 
\ee

\begin{proposition}\label{darcyprop}
	Let $\det f\neq 0$.
	Then, for $|\s|\in [1/C,C]$, for any fixed $C>0$, and $\eps$ sufficiently small,
	$\spec M(\e,\s)$ consists of $m$ eigenvalues lying within $o(\eps^{-1})$ of $\eps^{-1} \spec f$ together
	with $2$ eigenvalues lying within $o(1)$ of Darcy eigenvalues $\spec(m)$.
\end{proposition}

\begin{proof}
	The first, block-diagonal term of approximate block-diagonalization \eqref{Dbd} 
	under the assumption $\det f\neq 0$ has spectral separation of order $\s/\eps \gtrsim \eps^{-1}$
	between the upper left and lower right blocks.
	By standard matrix perturbation theory \cite{K,MZ,PZ}, it follows that that there exists a further {\it exact}
	near-identity diagonalizer
	\be\label{nearid}
	T:=\bp \Id_2 & \theta_1\\ \theta_2 & \Id\ep
	\ee
	with $|\theta_j|$ bounded by the size $O(1)$ of off-diagonal blocks divided by the spectral separation,
	or $|\theta_j|=O(\eps)$.
	This gives
	\be\label{xDbd}
	T^{-1}S^{-1}M(\e, \s)ST= \bp m(\e, \s)+ O(\eps) & 0\\ 0_2 & \eps^{-1}i\sigma f+ O(1)\ep,
	\ee
	whence the result follows by continuity of spectra under matrix perturbation.
\end{proof}

\bc\label{Darcycor}
For $\hat c\neq 0$, $\kappa^2\neq \kappa_{Dstab}^2$, and $\det f\neq 0$, stability of
the Darcy system \eqref{Darcy}--\eqref{rels}
is \emph{necessary} for stability of the full model \eqref{ampeq} for $\e$ sufficiently small.
\ec

\begin{proof}
	For $\hat c\neq 0$, $\kappa^2\neq \kappa_{Dstab}^2$, failure of stability of
	the Darcy system \eqref{Darcy}--\eqref{rels}
	implies that some Darcy eigenvalue in $\spec (m)$ must take on a strictly positive real part
	for some $|\s_*|>0$. Taking $C>0$ large enough in Proposition \ref{darcyprop} that 
	$\s_*\in [1/C,C]$, we find for $\e$ small enough that the corresponding eigenvalue of
	$M(\e,\s_*)$ must also have strictly positive real part, by continuity.
	Thus, Darcy stability is necessary for full stability, by contradiction.
\end{proof}

\br\label{complrmk}
In the above argument, we used crucially that $(\s/\e)f$ have eigenvalues of norm $\gtrsim \eps^{-1}$
assuming $\det f\neq 0$, or equivalently $\s \gtrsim 1$, and also uniform boundedness of all terms not
involving $\eps^{-1}$, in particular $\s^2$. Thus, the restriction $\s\in [1/C,C]$ is sharp.
Note that this region in which the Darcy model is relevant is disjoint from the region $|\s|\leq \e/C$
for which the Taylor expansion leading to full Eckhaus conditions is valid.
\er

\br\label{impDarcy}
Recalling that the Eckhaus stability condition $\mu_t^0<0$ in the translational mode,
already shown to be necessary for stability in Section \ref{s:nec},
agrees with the Darcy condition for stability of {\it its} translational mode, we see that the
new information contained in Corollary \ref{Darcycor} is precisely the condition for stability
of the order-one Darcy mode: $\Re(\hat c)<0$.
\er

\subsection{Refinement in the scalar case}
Let us show that Proposition \ref{darcyprop} can be refined in the scalar case.
\begin{proposition}\label{darcyproprefined}
	Suppose $\mu_c^0<0$. Then $\mu_t=\mu_{D,t}+O(\e)$ where $\l_{D,t}(\hat \s)=i\a_{D,t}\hat\s+\mu_{D,t}\hat\s^2$ is the expansion of the neutral eigenvalue of \eqref{eq:darcysys}.
\end{proposition}
\begin{proof}
	From Proposition \ref{mucconditions}, we can conclude that $\Re(\hat b(\kappa^2))$ is an increasing function of $\kappa^2$ as \eqref{A0soln} implies $A_0(\kappa^2)$ is a decreasing function and $\Re(d)h/f<0$. Moreover, $\hat b(0)$ is given by
	$$
	\hat b(0)=\tilde b-\frac{dh}{f}\frac{\Re(\tilde b)}{\Re(c)}=\tilde b-\frac{dh}{f\Re(c)}\Re(\tilde{b}).
	$$
	Taking real parts, we find
	$$
	\Re(\hat b(0))=\Big(1-\frac{\Re(d)h}{f\Re(c)} \Big)\Re(\tilde b)>0,
	$$
	by the lower bound on $\Re(d)h/f$ in Proposition \ref{mucconditions}.\\
	
	Hence the Darcy reduction gives a complex Ginzburg-Landau equation of the form studied in \cite{WZ2}, from which we can obtain an expression for $\mu_{D,t}$ by either Kato-style expansion as in \cite{WZ2} and then rearranging or by matched asymptotics as done in Appendix \ref{s:matched}. In either case, the expression we want is
	\be\label{darcymut}
	\mu_{D,t}=-\frac{\big(-2\kappa\Im(a)-\a_{D,t}\big)^2+4\kappa^2\Re(a)^2+2A_0^2\big(\Re(a)\Re(\hat c)+\Im(a)\Im(\hat c) \big) }{2A_0^2\Re(\hat c)},
	\ee
	with
	\be\label{darcyat}
	\a_{D,t}=-2\kappa\Im(a)+2\kappa\Re(a)\frac{\Im(\hat c)}{\Re(\hat c)}.
	\ee
	Comparing \eqref{darcymut} and \eqref{mut2} completes the proof, where we've used the critical observation that the Darcy reduction at the desired frequency has the same amplitude $A_0$ as the original amplitude system.
\end{proof}
Note that this refinement is still of the form ``Darcy stability is necessary but not sufficient'' as we did not need $\Re(d)h/f<0$, only the supercriticality condition $\Re(\hat c)<0$. To adapt the proof to $\Re(d)h/f>0$, we note that then $\Re(\hat b)$ is then a decreasing function of $\kappa^2$, with $\Re(\hat b(\kappa_E^2))=\Re(\tilde b)>0$.\\

Let us comment that the Darcy reduction at frequency $\kappa_*$ does not necessarily carry stability information for other frequencies, and in particular the Eckhaus condition for a typical $\kappa_*$ does not play a significant role. Indeed, for fixed $\kappa_*$ and assuming $\mu_c^0<0$, the existence range for the Ginzburg-Landau equation is smaller than that of the original amplitude system. In fact, by cleverly choosing $a,b,c,d,f,h$, one can make the existence range for Darcy at small $\kappa_*^2$ smaller than the stable range for the original amplitude system.

\begin{remark}
	One might call our Darcy reduction ``amplitude-adapted'' as it perfectly reconstructs the periodic wave it was generated by. The downside of this ``amplitude-adapted'' Darcy reduction is that it only carries information about that specific wave.\\
	
	An alternative way to perform the Darcy reduction, which one might call ``frequency-adapted'', is to choose $\tilde B_0$ to be $B_0$, so that the reduced Ginzburg-Landau equation is
	$$
	A_{\Ht}=aA_{\hat x\hat x}+\tilde bA+\hat c|A|^2A.
	$$
	This Ginzburg-Landau equation has periodic traveling waves in the same frequency range as the original amplitude system, hence the name ``frequency-adapted'' as it is generically the only way to preserve the range of frequencies, but the waves generically have different amplitudes compared to the corresponding frequency in the original amplitude system. An entirely similar computation to what we've done so far yields an Eckhaus condition for the ``frequency-adapted'' Darcy reduction of the form
	$$
	\kappa^2\leq\tilde{\kappa}^2_S:=\frac{2\frac{\Re(\tilde b)}{\Re(\hat c)}(\Im(a)\Im(\hat c)+\Re(a)\Re(\hat c) )}{ 4\Re(a)^2(1+\hat q^2)+2\frac{\Re(a)}{\Re(\hat c)}\big(\Im(a)\Im(\hat c)+\Re(a)\Re(\hat c) \big)}.
	$$
	We claim under the assumption that $\mu_c^0<0$ and the Benjamin-Feir-Newell criterion \eqref{BFN} that $\kappa^2_S\leq \tilde\kappa_S^2$ for $\kappa_S^2$ as in \eqref{kappastab}. Indeed, Proposition \ref{mucconditions} can be rephrased as $\Re(c)<\Re(\hat c)<0$, and so we readily obtain
	$$
	\Re(c)4\Re(a)^2(1+\hat q^2)+2\Re(a)\big(\Im(a)\Im(\hat c)+\Re(a)\Re(\hat c) \big)\leq \Re(\hat c)4\Re(a)^2(1+\hat q^2)+2\Re(a)\big(\Im(a)\Im(\hat c)+\Re(a)\Re(\hat c) \big).
	$$
	Taking the reciprocal then yields
	$$
	\begin{aligned}
		&\frac{1}{\Re(c)4\Re(a)^2(1+\hat q^2)+2\Re(a)\big(\Im(a)\Im(\hat c)+\Re(a)\Re(\hat c) \big)}\\
	&\quad \geq\frac{1}{\Re(\hat c)4\Re(a)^2(1+\hat q^2)+2\Re(a)\big(\Im(a)\Im(\hat c)+\Re(a)\Re(\hat c) \big)}.
	\end{aligned}
	$$
	Multiplying both sides by $\Re(\tilde b)(\Im(a)\Im(\hat c)+\Re(a)\Re(\hat c) )<0$ by \eqref{BFN} then yields
	$$
	\begin{aligned}
		&\frac{\Re(\tilde b)(\Im(a)\Im(\hat c)+\Re(a)\Re(\hat c) )}{\Re(c)4\Re(a)^2(1+\hat q^2)+2\Re(a)\big(\Im(a)\Im(\hat c)+\Re(a)\Re(\hat c) \big)}\\ &\quad \leq 
		\frac{\Re(\tilde b)(\Im(a)\Im(\hat c)+\Re(a)\Re(\hat c) )}{\Re(\hat c)4\Re(a)^2(1+\hat q^2)+2\Re(a)\big(\Im(a)\Im(\hat c)+\Re(a)\Re(\hat c) \big)}.
	\end{aligned}
	$$
	As this is a simple rearrangement of $\kappa_S^2\leq\tilde\kappa_S^2$, we obtain our desired conclusion. Hence, we conclude that the ``frequency-adapted'' Darcy reduction also carries necessary but not sufficient stability about {\it every} wave, however, it is less precise than the stability information of the ``amplitude-adapted'' Darcy reduction.
\end{remark}

\section{Sufficiency of Eckhaus conditions, case $m=1$}\label{s:suff}

Up to now, we have determined asymptotics for the second-order
Taylor expansion at the origin of the dispersion relation
for the singular \eqref{ampeq} model. This gives useful necessary conditions \eqref{assneutralstab}
for stability of exponential solutions in terms of the signs of the real parts of 
second-order coefficients $\mu_t$ and $\mu_c$.
We shall show later that the second-order Taylor expansions for \eqref{ampeq} well-approximate those for the key neutral
modes of the exact spectrum of the linearized operator around spatially periodic convective Turing patterns,
hence these conditions are necessary also for diffusive stability of the full periodic waves.
Indeed, the proof relies strongly on the spectral perturbation analysis done above for \eqref{ampeq}, the exact spectra
being shown to coincide to that of a small perturbation of the Fourier symbol analyzed above.

The above study, however, concerns only the radius of convergence of the Taylor series about the origin, which
can be seen (see Section \ref{s:suff}) to 
correspond to $|\hat\sigma|\leq \eps/C$ in the scaling for \eqref{ampeq}.
To obtain {\it sufficient} conditions for stability, we must show stability for all $\hat \sigma\in \R$.

In this section we continue our study, treating the two-parameter matrix perturbation in $\eps$ and $\hat \sigma$
as a one-parameter perturbation in various regimes, in order to complete all cases and obtain equivalent conditions
for stability: specifically, to show that (strict) {\it asymptotic diffusive stability \eqref{assdiffstab} 
	is sufficient for stability,} just as (nonstrict) {\it asymptotic neutral stability \eqref{assneutralstab} 
	is necessary.}

We start by introducing the rescaled parameter $\check \sigma:= \hat \sigma/\eps$, or $\hat \sigma=\eps \check \sigma$,
resulting in a desingularized matrix perturbation problem
\ba\label{eq:desing1}
\l \bp u_0 \\ v_0 \\ w_0\ep&= \Bigg(
-\check \sigma^2 \eps^2 \bp \Re (a) & -\Im (a) & 0 \\ -\Im(a) & \Re(a)& 0\\ 2A_0\Re(g) & 2A_0\Im(g) & e_B\ep 
+i\check \sigma\bp O(\eps) & O(\eps) &  0 \\ O(\eps) & O(\eps) &  0 \\ 2A_0h + O(\eps)  & 0 & f\ep \\ 
&\quad+\bp  2A_0^2\Re(c) & 0& A_0\Re(d)\\ 2A_0^2\Im(c) & 0 & A_0\Im(d)\\
0 & 0  & 0\ep  \Bigg)\bp u_0 \\ v_0 \\ w_0\ep
\ea
corresponding to a parabolic singular perturbation of a relaxation system.

\subsection{Case (i) $|\check \sigma|\leq 1/C$, $C\gg 1$ ($\hat \sigma\leq \eps/C$)}\label{s:step1check}
On this region, we consider \eqref{desing} as a family of matrix perturbation problems in $\check \sigma$
parametrized by $\eps$, with coefficients uniformly bounded, and Taylor expand around $\check \sigma=0$
as described just above. The balancing procedure described in Section \ref{s:disp} 
then yields an analytic expansion in $\check \sigma$ of associated eigenvalues that is
uniformly convergent in a small ball with respect to $\eps$ sufficiently small.

Comparing to the expansions \eqref{muc}, \eqref{lamex} 
in $\sigma$ derived in Section \ref{s:nec}, we see that these are
\ba\label{rescdisp1}
\hat \lambda_s(\eps,\check \sigma)&= 2A_0^2\Re(c)+ O(\check \sigma),\
\hat \lambda_t&= \check \alpha_t \check \sigma + \check \mu_t \check \sigma^2 + O(\check \sigma^3),\\
\hat \lambda_c&= \check \alpha_c \check \sigma + \check \mu_c \check \sigma^2 + O(\check \sigma^3),
\ea
where
\ba\label{talpha's1}
\check \alpha_t&= \eps \alpha_t=  O(\eps) \, ,\\ 
\check \alpha_c&= \eps \alpha_c= i ( f+2A_0hp) +O(\eps),
\ea
and
\ba\label{tgenmu's1}
\check \mu_t&=\eps^2 \mu_t= \eps \mu_t^0 + O(1)  \, ,\\ 
\check \mu_c&= \eps^2 \mu_c =  2A_0  ph( f+2A_0hp) +O(\eps). 
\ea

Thus the stability conditions are sufficient as well as necessary on this region.

\subsection{Case (ii) $1/C\leq |\check \sigma|\leq C$ ($\eps/C\leq |\hat \sigma|\leq C\eps$)}\label{s:step2check}
On this region, we take a different point of view, considering
\eqref{eq:desing1} as a compact family of matrix perturbation problems parametrized by 
$$
1/C\leq |\check \sigma|\leq C,
$$
$C>0$ arbitrary, with perturbation parameter $\rho:=\eps \check \sigma = o(\check \sigma)$,
with again all coefficients uniformly bounded.
That is, we consider 
$$
M(\check \sigma):= M_0(\check \sigma)+ \rho M_1(\check \sigma) + \rho^2 M_2(\check \sigma)
$$
with
\ba\label{Ms1}
M_0&= \bp 2A_0^2\Re(c) & 0& A_0\Re(d)\\ 2A_0^2\Im(c) & 0 & A_0\Im(d)\\
2A_0hi\check \sigma   & 0 & fi\check \sigma \ep,\\
M_1&= 
i \bp -2\kappa \Im(a)& -2\kappa \Re(a)&  0 \\ 2\kappa \Re(a)& -2\kappa \Im(a) & 0 \\ 2A_0\kappa \Im(g)  & 0 & 0\ep ,\\
M_2&= -\bp \Re (a) & \Im (a) & 0 \\ -\Im(a) & \Re(a)& 0\\ 2A_0\Re(g) & 2A_0\Im(g) & e_B\ep. 
\ea

Evidently, $\det M_0\equiv 0$, hence we may factor the characteristic polynomial 
\be\label{char1}
p(\lambda;\check \sigma):= \det (\lambda -M_0(\check \sigma))
\ee
as $p(\lambda;\check \sigma)= \lambda q(\lambda;\check \sigma)$, where $q$ is quadratic,
namely
\be\label{factor1}
q(\lambda;\check \sigma)=\lambda^2 + \beta \lambda + \gamma,
\ee
where
\be\label{coeffs1}
\beta= -(2A_0^2\Re(c)+ fi\check \sigma), \quad \gamma= 2A_0^2 if\check \sigma\Big(\Re(c)-\frac{h\Re(d)}{f}\Big)=2A_0^2if\check\sigma\Re(\hat c),
\ee
where, following \eqref{rels},
$\hat{c}:=c-\frac{dh}{f}$
is the updated value of $c$ in the Darcy reduction.
\begin{remark}
	The factorization \eqref{factor1} can also be seen by direct expansion of the determinant along the central column of $M_0(\check\sigma)$.
\end{remark}

Let us investigate the appearance of a pure imaginary root $\lambda=i\tau$. This gives from
$$
0=q(i\tau; \check \sigma)=-\tau^2 +\beta i\tau +\gamma
$$
the pair of equations
\ba\label{pair1}
0=\Re q(i\tau;\check \sigma)&= -\tau^2 - \tau \Im (\beta)+ \Re(\gamma)
= \tau^2 - \tau f\check \sigma ,\\
0=\Im q(i\tau;\check \sigma)&= \tau \Re(\beta) + \Im(\gamma)
=  2A_0^2\Big(-\Re(c) \tau + \check \sigma f\Re(\hat{c}) \Big).
\ea

Solving, we obtain for $\tau=0$ the trivial solution $\tau=0$, $\check \sigma=0$, which we have specifically excluded by taking $\check \sigma \neq 0$, or else $\tau=0$, $\hat c=0$, which we exclude by the genericity assumption
\be\label{genc}
\hat c\neq 0,
\ee
and, otherwise, the linear system
\ba\label{linpair1}
\tau = f\check \sigma, \quad \tau = \check \sigma f\frac{\Re(\hat{c})}{\Re(c)},
\ea
which is consistent for $\check \sigma\neq 0$ if and only if 
$$
\Re(c)=\Re(\hat{c})\iff 0= h\Re(d).
$$

We exclude the latter case by the genericity assumption
\be\label{genz1}
h\Re(d)\neq 0.
\ee
Recall that first-order coupling coefficients are $d$ and $h$, with
the hyperbolic model decoupling if either of these vanish. Thus, failure of \eqref{genz1}
is related to a degenerate (at least partial) decoupling of the system.

Assuming \eqref{genz1}, we find that there are no imaginary eigenvalues of $q$ on $1/C\leq |\check \sigma|\leq C$.
Thus, by continuity/compactness, there is a uniform spectral gap between $\lambda=0$ and the remaining two
eigenvalues of $p(\cdot;\check \sigma)$. By standard matrix perturbation theory, this yields an analytic
branch $\lambda_0(\rho;\check \sigma)$ with $\lambda_0(0;\sigma)=0$, convergent for $|\rho|=\eps |\check \sigma|$
sufficiently small.
By standard matrix perturbation theory, we obtain also
continuous expansions in $\rho$ of the remaining two nonzero eigenvalues of $M_0$.

Stability of large (nonzero real part) eigenvalues is straightforward, at least for $|\check \sigma|$ bounded, 
since continuity of spectrum is then enough to conclude stability or instability on the whole domain. 
Thus, the problem reduces to checking stability at the left 
endpoint $|\check \sigma|=1/C$, already determined by Taylor series analysis.  
Stability of the perturbed zero mode requires more care, as, being neutral to lowest order, it will be
determined by the second-order term in the Taylor expansion.

\subsubsection{Matched expansion of the zero mode}\label{s:matchcheck}
As often the case with matrix perturbation theory, the abstract theory (as above) is useful for establishing
analyticity, but it is easier to compute coefficients by positing an analytic expansion and finding coefficients by
matching terms at successive orders. Namely, defining 
$\lambda_0(\rho)= \lambda_1 \rho + \lambda_2 \rho^2 + H.O.T.$,
we may expand the characteristic polynomial equation $0= \det (M(\rho;\check \sigma)-\lambda_0)$, after
factoring out a common factor $\rho$ from the second column, as
\ba\label{fcharexp1}
0 &=\det \bp
2A_0^2\Re(c)+\rho(-2i\kappa\Im(a)-\lambda_1 ) & 
 \rho\Im (a)  - 2i\kappa \Re(a)& A_0\Re(d)\\
2A_0^2\Im(c)+2i\kappa \Re(a)\rho  &  -2i\kappa \Im(a) - \lambda_1+\rho(\Re(a)-\lambda_2)
& A_0\Im(d) \\
2A_0  i\kappa \Im(g)\rho+ 2A_0 hi\check \sigma & \rho 2A_0\Im(g) & -\rho \lambda_1  + fi\check \sigma
\ep + O(\rho^2)\\
& =
\det \bp 2A_0^2\Re(c)& - 2i\kappa \Re(a)& A_0\Re(d)\\
2A_0^2\Im(c) &  -  2i\kappa \Im(a) - \lambda_1 & A_0\Im(d) \\
2A_0hi\check \sigma & 0 & fi\check \sigma  \ep + O(\rho)\\
&= 
i\check \sigma \det \bp 2A_0^2\Re(c)& - 2i\kappa \Re(a)& A_0\Re(d)\\
2A_0^2\Im(c) &  -  2i\kappa \Im(a) & A_0\Im(d) \\
2A_0h& 0 & f\ep 
-i\check \sigma \lambda_1 
\det \bp 2A_0^2\Re(c)& 0& A_0\Re(d)\\
2A_0^2\Im(c) &   1 & A_0\Im(d) \\
2A_0 h& 0 & f \ep + O(\rho),\\
\ea
where in the third inequality we have factored out $i\check \sigma$ from the third row.
Thus,
\ba \label{l11}
\lambda_1&= i \det \bp 2A_0^2\Re(c)& - 2\kappa \Re(a)& A_0\Re(d)\\
2A_0^2\Im(c) &  -  2\kappa \Im(a) & A_0\Im(d) \\ 2 A_0h& 0 & f\ep  /
\det \bp 2A_0^2\Re(c)& 0& A_0\Re(d)\\
2A_0^2\Im(c) &   1 & A_0\Im(d) \\
2A_0 h& 0 & f \ep \\
&=
i\Big(-2\kappa\Im(a)+2\kappa\Re(a)\frac{\Im(\tilde{c})}{\Re(\tilde{c})} \Big).
\ea

We note that {\it $\lambda_1$ is pure imaginary} (as noted earlier by symmetry), {\it and independent
	of $\check \sigma$}. 

\medskip

To determine the second-order coefficient $\lambda_2$, we now expand to the next order in \eqref{fcharexp1}
and match first-order coefficients in $\rho$.
Expanding the first line of \eqref{fcharexp1} and collecting first-order terms, we have
the term 
$$
-i\check \sigma \lambda_2 \rho 
\det \bp 2A_0^2\Re(c)& 0& A_0\Re(d)\\
2A_0^2\Im(c) &   1 & A_0\Im(d) \\
2A_0 h& 0 & f \ep 
=
-i\check \sigma \lambda_2 \rho 2A_0^2\big(\Re(c)f-h\Re(d)\big),
$$
as the only term involving $\lambda_2$, similarly as in the previous calculation.
The remaining terms can be calculated from the $\cO(\rho)$ coefficient in \eqref{fcharexp1} after using multilinearity to remove the $\lambda_2$ in the central column. In particular, denoting the remaining terms as $\rho\Sigma$, where
\ba\label{1terms1}
\Sigma&:=\frac{\d}{\d \rho}\det\bp 2A_0^2\Re(c)+\rho(-2i\kappa\Im(a)-\lambda_1 ) & -2i\kappa\Re(a)+\rho\Im(a)& A_0\Re(d)\\
	2A_0^2\Im(c)+2i\kappa\Re(a)\rho & (-2i\kappa\Im(a)-\lambda_1 )+\rho \Re(a) & A_0\Im(d)\\
	2A_0hi\check{\sigma}+2A_0i\kappa\Im(g)\rho & 2A_0\Im(g)\rho & fi\check{\sigma}-\rho\lambda_1
 \ep\Bigg\rvert_{\rho=0}
\ea
we see that we have 
\be\label{l21}
\lambda_2= \frac{i\Sigma}{ \check \sigma 2A_0^2\big(\Re(c)f-h\Re(d)\big)}.
\ee
While the actual expression for $\Sigma$ in \eqref{1terms1} is quite complicated, it is important to note that it is \emph{affine} in $\check{\sigma}$ as can readily be seen by performing cofactor expansion along the third column. 
Hence, to determine stability by $\sgn \lambda_2$, we need only compute the sign of $\Im(\Sigma)$.\\

The expression for $\Im(\Sigma)$ is also rather complicated, but the important point is that it too is affine in $\check \sigma$,
hence, by \eqref{l21}, $\Re(\lambda_2)$ is affine in $(1/\check \sigma)$.

Thus, to check stability, it is sufficient to check only at the endpoints, $|\check \sigma|=1/C$- where it is
already known from the Taylor expansion- and $|\check \sigma|=C$, $C>0$ arbitrarily large.
At the latter boundary, $|\check \sigma|$ large, we may either check by throwing out $O(1/\check \sigma)$ terms
in the just-completed computation, or else by comparing to a separate computation for $|\check \sigma|\geq C$,
$C>0$ sufficiently large, to be done in the next steps.
As the first possibility is rather complicated, we will follow that latter approach and defer this to the following
sections.

\br\label{affinermk}
In the vectorial case $m=1$, the computations in \eqref{1terms1}-\eqref{l21} become significantly 
more complicated, another point in which the scalar and vectorial cases significantly differ.
However, the end result is the same, giving a quantity affine in $1/\check \sigma$.
Likewise, the computations \eqref{pair1}-\eqref{linpair1}, though a bit more complicated, extend
to give the same result for $m>1$.  We carry out these computations in detail in Section \ref{s:ext}.
\er

\subsection{Case (iii) $C\leq |\check \sigma|\leq 1/C\eps$ ($C\eps \leq |\hat \sigma|\leq 1/C$)}\label{s:step3check}
Here, $|\check \sigma|\gg 1$ but $|\rho|\ll 1$.
On this region, reviewing $M_0$ in \eqref{Ms1}, we see that, assuming as usual invertibility of $f$,
there is now a dominant large eigenvalue $i\check \sigma f$, with associated left and right eigenvectors
\be\label{newlr}
l= \bp 2f^{-1}A_0h  & 0& 1\ep, \qquad r=\bp 0 & 0& 1\ep^T.
\ee
Motivated by this observation, we make the global change of coordinates 
$\tilde M(\sigma):=T^{-1} M(\sigma)T$, where
\be\label{Tdef}
T=\bp 1 & 0 & 0\\ 0 & 1 & 0\\ -2f^{-1}A_0 h & 0 & 1\ep,
\ee
yielding 
\be\label{tildeM}
\tilde M(\check \sigma)= \bp \check M & \theta_1\\ \theta_2 & if\check \sigma + 2f^{-1}A_0^2 h\Re (d) - e_B \rho^2\ep,
\ee
with $|\theta_j|, |\rho|=O(1)$ and 
\ba\label{checkM}
\check M(\check \sigma)&= \bp 2A_0^2 (\Re(c)-\Re(d)f^{-1}h) & 0\\2A_0^2(\Im (c)-\Im(d)f^{-1}h)& 0\ep
+\rho i\bp -2\kappa \Im(a) & -2\kappa \Re(a)\\ 2\kappa \Re(a) & - 2\kappa \Im (a)\ep
-\rho^2 \bp \Re(a)& \Im(a)\\ -\Im(a) & \Re(a)\ep.
\ea

Here, the lower righthand entry is dominated by the $\check \sigma$ order term $if\check \sigma\gg 1$,
whereas remaining terms are $O(1)$; in particular, there is separation of order $|\check \sigma|$ between
this entry and the spectra of the upper lefthand block, and off-diagonal blocks of order one.
Thus, there exists an exact diagonalizing transformation
of form $S=\Id+ O(1/|\check \sigma|)$ showing that the real part of the largest eigenvalue 
is $2f^{-1}A_0 \Re (d)+o(1)$, which has sign to first order independent of $\check \sigma$.
We have from the analysis of case (ii) that this largest eigenvalue must have negative real part at the
innermost endpoint $|\check\sigma|=C$, hence
$$
\hbox{\rm $2f^{-1}A_0^2 h\Re (d)<0$ and therefore also $2f^{-1}A_0^2h \Re (d)-e_B\rho^2 <0$,}
$$
yielding by exact diagonalization that this eigenvalue is stable on the entire $|\check \sigma|$-interval.

On the complementary- i.e., upper left- exactly diagonal block, we obtain, similarly an $O(1/|\check \sigma)$
perturbation of \eqref{checkM}.
But, \eqref{checkM} may be recognized as the linearized Darcy system, with $\hat \sigma=\eps \check \sigma$
replaced by $\rho$, on the interval $C \eps\leq |\rho|\leq C$.
By standard analysis of the Darcy system, the associated eigenvalues are to leading order $2A_0^2\Re(\hat c)$
and $\mu_t^0 \rho^2$.
By comparison at the boundary $|\check \sigma|=C$, the former must have negative real part, as one of the $m+1$ stable
order-one eigenvalues. We note that we have just verified indirectly the fact shown by direct
computation for the scalar case $m=1$ in Section \ref{s:Dnec} that the stability conditions imply $\Re(\hat c)<0$.
Our abstract argument here has the advantage that it generalizes to the vector case $m>1$.

As the latter has $\Re \mu_t<0$ by assumption/agreement of translational expansions for Darcy and full model,
we may conclude that the reduced Darcy system is {\it stable} for all $\rho\in \R$, by standard stability
theory for complex Ginzburg-Landau as may be obtained readily by the quadratic formula:
in particular, on the $|\rho|\ll 1$ region under consideration, where it may be obtained by inspection of
the second-order Taylor series.

However, the order $|\rho|^2$ perturbation could change sign under $o(1)$ perturbation, hence must be treated
differently. Fortunately, having already determined that the two other eigenvalues are order $1$ and $|\hat \sigma|\gg1$, we have a spectral separation of order one between those and the order $|\rho|^2$ eigenvalue, hence can
repeat the argument of case (ii) to conclude that the analytic expansion in $\rho$ deduced there remains valid
up to $|\rho|\ll 1$, and so the coefficient of $\rho^2$ in that expansion is affine in $\check \sigma$
as already determined, with domain of validity $\eps/C\leq |\check \sigma|\leq 1/C\eps$ including both those of cases
(ii) and (iii), hence this eigenvalue is stable iff and only if it is stable at the endpoints
$|\check \sigma|=\eps /C$ and $|\check \sigma|=1/C\eps$.
As observed in the treatment of case (ii), it is stable
at the inner endpoint $|\check \sigma|=\eps /C$ as the outer endpoint of region (i) already determined.
At the outer endpoint $|\check \sigma|=1/C\eps$, on the other hand, the associated Darcy eigenvalue
has strictly negative value bounded away from zero, which property persists under $o(1)$ perturbation to
give stability of the exactly diagonalized eigenvalue. Thus, we may conclude stability also of this smallest,
order $|\rho|^2$ eigenvalue.

Note that this in passing gives stability on the regime of case (ii), completing the analysis there.

\subsection{Case (iv) $1/C\eps \leq |\check \sigma|\leq 1/C\eps^2$ ($1/C \leq |\hat \sigma|\leq 1/C\eps$)}\label{s:step4check}
On this region, $|\rho|\geq 1/C$ and $1+\rho^2\ll |\check \sigma|$, hence the lower righthand entry
$$
if\check \sigma + 2f^{-1}A_0^2h \Re (d) - e_B \rho^2
$$
of $\tilde M(\check \sigma) $ in \eqref{tildeM} still dominates all other terms, giving a spectral separation
of order $|\check \sigma|$ between this and the upper lefthand diagonal block, with off-diagonal terms
$\theta_j$ now of order $1+|\rho|^2$.
It follows that there is an analytic exact block diagonalization of form 
$$
\Id + O((1+|\rho|^2)/|\check \sigma|= \Id + o(1)
$$
yielding in the lower righthand entry a real part $2f^{-1}A_0^2h \Re (d) +o(1)$ already verified as stable,
and in the $2\times 2$ upper lefhand block $o(1)$ perturbations of the eigenvalues of the
associated Darcy system. As the later have real parts bounded above by a constant times $-(1+|\rho|^2)$,
hence by a strictly negative constant, this property persists under $o(1)$ perturbation, giving
stability of the corresponding exactly diagonalized eigenvalues.

\subsection{Case (v) $ 1/C\eps^2\leq |\check \sigma|\leq C/\eps^2$ ($1/C\eps\leq |\hat \sigma|\leq C/\eps$)}\label{s:step5}
Defining $\sigma:= \eps^2 \check \sigma$, and factoring out $1/\eps^2$, we have the matrix perturbation 
problem $\eps^{-2}(M_0 + O(\eps))$, where
\be\label{lastM0}
M_0=\bp -\sigma^2\Re(a) & -\sigma^2 \Im(a) & 0 \\ \sigma^2 \Im(a) & \sigma^2 \Re(a) & 0\\
2A_0hi \sigma- \sigma^2 2A_0 \Re(g)& -\sigma^2 2A_0 \Im(g) & fi \sigma- \sigma^2 e_B\ep,
\ee
where $1/C\leq |\sigma|\leq C$, i.e., $\sigma $ varies on a compact range.
Evidently, $M_0$ is lower block-diagonal, with upper lefthand $2\times 2$ block of symmetric part negative
definite with spectral gap $\Re(a) \sigma^2$, and lower lefthand block having negative real part of spectral gap 
$e_B \sigma^2$. Thus, the spectra of $M_0$ have real part uniformly negative over the range under
consideration, and we obtain negativity of the perturbed spectra by continuity of spectra under perturbation.
Combining with previous cases, this verifies stability up to $|\check \sigma|\leq C/\eps^2$ provided
asymptotic diffusive stability \eqref{assdiffstab} holds.

\br\label{compatrmk}
In the vectorial case $m>1$, the real part of the spectrum of
the lower righthand block $(fi \sigma- \sigma^2 e_B)$ is not determined simply by the real part of
the spectrum of $e_B$.  However, $\Re\spec(fi \sigma- \sigma^2 e_B)<0$ by Assumption \eqref{compat},
hence by continuity/compactness we still obtain a uniform spectral gap. 
\er

\subsection{Case (vi) $ |\check \sigma|\geq C/\eps^2$ ($|\hat \sigma|\geq C/\eps$)}\label{s:step6}
In this case, the order $\rho^2=\check \sigma^2\eps^2$ terms dominate order $\check \sigma$ and other
terms, and stability follows in straightforward fashion from parabolicity of the truncated system.
We omit the details of this standard argument; see, e.g., \cite{MZ}.

\subsection{Final result}\label{s:fres}
Combining the results of Sections \ref{s:step1check}-\ref{s:step2check}, we have the following simple condition
for stability, analogous to those for the standard (nonsingular) complex Ginzburg-Landau system.

\begin{proposition}\label{stabcrit}
	Assuming the generic conditions of supercriticality \eqref{supercrit}, 
	noncharacteristicity of effective flux \eqref{fluxcond},
nontrivial Jordan structure \eqref{gencase}, and nonvanishing of $\Re \mu_t^0,\Re \mu_c^0$ in \eqref{genmu's},
	 asymptotic diffusive stability \eqref{assdiffstab}, or
	$\Re \mu_t^0, \Re \mu_c^0>0$ with $\mu_j^0$ as in \eqref{genmu's}, is \emph{necessary and sufficient} 
	for diffusive stability \eqref{dss} in the sense of Schneider of
	periodic (exponential) solutions \eqref{tw} of \eqref{ampeq} 
	with $m=1$, 
	for $0<\eps\leq \eps_0$ sufficiently small, where $\eps_0$ is uniform on compact
	parameter sets satisfying the above assumptions.
\end{proposition}

\begin{proof}
	For fixed model parameters, the result follows for $\eps>0$ sufficiently small, by
	the arguments of cases (i)-(vi) above, where the upper bound $\eps_0$ needed for $\eps$ depends
	only on lower bounds for the various quantities assumed to be nonvanishing.
	As these quantities are continuous, their minima are uniformly bounded from below on compact
	parameter sets where they do not vanish, and so $\eps_0>0$ may be chosen uniformly
	for compact parameter sets on which our hypotheses hold.
\end{proof}

Coefficients $\mu_t$ and $\mu_t$ are explicitly computable, giving simple necessary and sufficient
conditions for stability on a par with those for the classical (nonsingular) complex Ginzburg-Landau equation.

\br\label{remrmk}
The above argument is reminiscent of multi-parameter spectral perturbation computations
carried out in \cite[\S 4]{JNRZ2} and \cite[\S 2]{BJNRZ}, in which,
similarly, stability in successive regions is related back ultimately
to stability of a Taylor expansion about the origin.
In all three of these cases, there is an interesting analogy to relaxation systems and 
the studies of Kawashima, Shizuta, Zeng, and others \cite{ShK,Ze}.
For other examples of multi-parameter expansion as related to spectral stability,
see, e.g., \cite{PZ,BHZ,FS1,FS2}.
\er

\section{Extension to $m$ conservation laws}\label{s:ext}
As noted earlier, the arguments in Section \ref{s:nec} for the scalar case $m=1$
extend for the most part word for word to the vector case $m>1$, with the various symbols now representing
vectors and matrices rather than scalars, modulo the details
indicated in Remarks \ref{compatrmk},  \ref{splitrmk}, and \ref{affinermk}.  
We shall therefore not repeat the full argument here, but only supply the specific computations in cases (i) and
(ii) that are needed to treat the aspects in which the system case requires further analysis beyond what
was given for the scalar case.

\subsection{Region (i): Taylor expansion for vectorial case}\label{s:sysmatched}
There are two main steps in adapting our proof of sufficiency in case (i). The first is to carry out the Kato-style expansions for $\mu_{c,i}$ from the $m$ large eigenvalues and then to use matched asymptotics to obtain $\mu_t$ from the small eigenvalue.\\


We start as before with the preliminary diagonalization. First, we define a vectorized $p$, thought of as a row vector, by $$p:=-\frac{\Re(d)}{2A_0\Re(c)},$$ while keeping the same value of $q$ as the in the scalar case. We further denote by $e_i$ the standard basis, regarded as column vectors, of $\RR^m$. We will let $0$ denote an appropriately sized array whose entries are all $0$. Then our left and right (generalized) eigenvectors become
\be\label{vector-left}
L_s=\bp 1 & 0 & -p\ep, \qquad L_t=\bp q & 1 & -qp\ep,\qquad L_{c,i}=\bp 0 & 0 & e_i^T\ep,
\ee
and
\be\label{vector-right}
R_s=\bp 1 \\ -q\\ 0\ep,\qquad R_t=\bp 0 \\ 1\\ 0 \ep,\qquad R_{c,i}=\bp p_i \\ 0 \\ e_i\ep,
\ee
where $p_i$ denotes the $i$-th entry of $p$. From here, a similar computation to the $m=1$ case shows that this basis block diagonalizes $M_0$, with $\hat{M}_0$ a matrix of the form
\be\label{vector-hatm0}
\hat{M}_0=\bp 0 & A_0(\Im(d)+q\Re(d) )\\
0 & 0\ep.
\ee

In the vector case, the matrix $M_1$ is given by
$$
M_1=i\bp -2\kappa\Im(a) & -2\kappa\Re(a) & 0\\ 2\kappa\Re(a) & -2\kappa\Im(a) & 0\\ 2A_0h\e^{-1}+2A_0\kappa\Im(g) & 0 & \e^{-1} f\ep.
$$
Computing the action of $M_1$ on the left eigenvectors yields
$$
\begin{aligned}
	L_sM_1&=i\bp-2\kappa\Im(a)-2A_0p\big(\e^{-1}h+\kappa\Im(g) \big) & -2\kappa\Re(a) & -\e^{-1}pf\ep,\\
	L_tM_1&=i\bp -2\kappa\Im(a)q+2\kappa\Re(a)-2A_0qp\big(\e^{-1}h+\kappa\Im(g)\big)& -2\kappa\Re(a)q-2\kappa\Im(a) & -\e^{-1}qpf\ep,\\
	L_{c,i}M_1&=i\bp 2A_0h_i\e^{-1}+2\kappa A_0\Im(g_i) & 0 & e_i^Tf\ep. 
\end{aligned}
$$
We can then readily obtain the expression for $M_1$ in the eigenbasis as
$$
M_1=\bp m_1 & s_1\\ s_2 & \hat M_1\ep,
$$
with $m_1$ scalar, $s_1$ an $m+1$-row vector, $s_2$ an $m+1$-column vector, and $\hat M_1$ an $m+1\times m+1$ matrix. More specifically,
\ba\label{vectorM1s}
m_1&= L_sM_1R_s=-\e^{-1}2A_0\vec{p}\vec{h}i+O(1),\\
s_1&=L_sM_1\bp R_t & R_{c,i}\ep=i\bp -2\kappa\Re(a) & -\e^{-1}p(f+2A_0hp)\ep+i\bp 0 & O(1)\ep,\\
s_2&=\bp L_t\\ L_{c,i}\ep M_1R_s=i\bp -2\e^{-1}A_0qph+O(1)\\ 2A_0h\e^{-1}\ep,\\
\hat M_1&= \bp L_t \\ L_{c,i}\ep M_1\bp R_t & R_{c,i}\ep=i\bp -2\kappa\Re(a)q-2\kappa\Im(a) & -\e^{-1}qp(f+2A_0hp)+O(1)\\ 0 & \e^{-1}(f+2A_0hp)+O(1)\ep.
\ea

We can perform a similar first order diagonalization, by choosing
$$
\cT(\hat\sigma,\e)=\bp 1 & 0 \\ \hat\sigma t_2 & Id_{m+1}\ep\bp 1 & \hat\sigma t_1 \\ 0 & Id_{m+1}\ep,
$$
with corresponding inverse
$$
\cT(\hat\sigma,\e)^{-1}=\bp 1 & -\hat\sigma t_1 \\ 0 & Id_{m+1}\ep\bp 1 & 0 \\ -\hat\sigma t_2 & Id_{m+1}\ep.
$$

Setting $N(\e,\hat \s)=\cT M(\e,\hat \s)\cT^{-1}$ as before, we obtain
$$N(\e,\hat \s)=N_0+\hat \s N_1(\e)+\hat \s^2 N_2(\e)+ \hat \s^3 N_3(\e)+O(\hat \s^4),$$
with $N_0=M_0$, and $N_1$ given in block form
$$
N_1=\bp m_1 & s_1+t_1(\hat M_0-m_0)\\s_2-(\hat M_0-m_0)t_2 & \hat M_1\ep,
$$
from which we discover that choosing
\be\label{vectorts}
t_1=- s_1(\hat M_0-m_0)^{-1},\qquad t_2=(\hat M_0-m_0)^{-1}s_2,
\ee
makes $N_1$ block diagonal. Under this choice, we find that the coefficient of $\hat\s^2$, $N_2$, is of the form
\ba
N_2&=-M_2+\bp 0 & t_1\\ t_2 & 0\ep M_1+M_1 \bp 0 & -t_1 \\ -t_2 & 0 \ep\\
&\quad+\bp 0 & t_1 \\ t_2  &0\ep M_0\bp 0 & -t_1\\ -t_2 & 0\ep +M_0\bp t_1t_2 & 0 \\ 0 & 0\ep+\bp 0 &0\\ 0 & t_2t_1\ep M_0.
\ea

As $M_2$ and the left/right eigenvectors are all $O(1)$, we see that to leading order, $N_2$ is given by
$$
N_2=\bp 0 & t_1\\ t_2 & 0\ep M_1+M_1 \bp 0 & -t_1 \\ -t_2 & 0 \ep+\bp 0 & t_1 \\ t_2  &0\ep M_0\bp 0 & -t_1\\ -t_2 & 0\ep+M_0\bp t_1t_2 & 0 \\ 0 & 0\ep+\bp 0 & 0\\ 0 & t_2t_1\ep M_0+O(1).
$$
Computing the requisite matrix products and zooming in on the bottom $m+1\times m+1$ block, we find
$$
\hat N_2= t_2s_1-s_2t_1+t_2t_1 (\hat M_0-m_0)+O(1).
$$
Plugging in \eqref{vectorts}, we then find that
\be\label{vectorhatN2}
\hat N_2=- s_2t_1+O(1).
\ee
We observe that, as before, $\hat M_0^2=0$, and so $(\hat M_0-m_0)^{-1}$ is given by
\be\label{vectorreducedresolvent}
(\hat M_0-m_0)^{-1}=-\frac{1}{m_0}\Big(Id+\frac{1}{m_0}\hat M_0 \Big).
\ee
In particular, we discover that
\be\label{vectort1exp}
t_1=-s_1(\hat M_0-m_0)^{-1}=\frac{1}{m_0}s_1+\frac{4A_0i\kappa\Re(a)}{m_0^2}\bp 0 & \Im(d)+q\Re(d)\ep.
\ee
So to leading order, $\hat N_2$ is $-m_0^{-1}s_2s_1$. Expanding \eqref{vectorhatN2} out, we get
\be\label{vectorhatN2exp}
\hat N_2=\frac{1}{m_0}\bp 4\kappa\Re(a)A_0qph\e^{-1} & 2A_0qphp\big(f+2A_0hp \big)\e^{-2}\\ -4\kappa\Re(a) A_0h\e^{-1} & -2A_0hp(f+2A_0hp)\e^{-2} \ep+ H.O.T.
\ee
The final piece of information we need before we ``balance'' the Jordan block away is $\hat N_3$, the bottom $m+1\times m+1$ block of $N_3$. We begin by collecting cubic terms in our expansion to get
$$
N_3= M_1\bp t_1 t_2 & 0\\ 0 & 0\ep+\bp 0 & 0 \\ 0 & t_2t_1\ep M_1+\bp 0 & t_1 \\ t_2 & 0 \ep M_1\bp 0 & - t_1\\ - t_2 & 0 \ep+ H.O.T.
$$
From this, we can extract the bottom right block as
\be\label{vectorhatN3}
\hat N_3=t_2t_1\hat M_1-m_1t_2t_1+ H.O.T.=t_2t_1(\hat M_1-m_1)+ H.O.T.
\ee
We observe from \eqref{vector-hatm0} and \eqref{vectorM1s} that the reduced spectral problem
$$
\hat N(\hat \s,\e)=\hat M_0+\s \hat M_1+\hat\s^2\hat N_2+\hat\s^3\hat N_3+O(\hat\s^4),
$$
is, to first order in $\hat\s$, upper block triangular. We then observe that we only need the bottom left block of $\hat N_3$ for our balancing transformation. We observe that the first column of $ t_2t_1$ is $O(\e^{-1})$, as the first entry of $ t_1$ is $O(1)$ and $t_2=O(\e^{-1})$. In particular, since $\hat M_1$ is also upper block triangular, we find that the bottom left block of $\hat N_3$ is $O(\e^{-2})$, compared to the expected order of $O(\e^{-3})$.\\

In order to find the bottom left block of $\hat N_3$, we begin by observing that
$$
t_2=-\frac{i}{m_0}\bp O(\e^{-1})\\ 2A_0h\e^{-1}+O(1)\ep.
$$
Hence, we find that
$$
\hat N_3=-\frac{i}{m_0^2}\bp O(\e^{-1})\\ 2A_0 h\e^{-1}\ep\bp 2\kappa\Re(a) & O(\e^{-1})\ep \bp 2A_0ph\e^{-1} & -\e^{-1} qp f \\ 0 & \e^{-1}(f+2A_0hp)+\e^{-1}2A_0ph Id_m\ep+ H.O.T.
$$
Computing the above products, we find
\be\label{vectorhatN3exp}
\hat N_3=-\frac{i}{m_0^2}\bp O(\e^{-2}) & O(\e^{-3})\\ 8\e^{-2}\kappa \Re(a) A_0^2hph+O(\e^{-1}) & O(\e^{-3})\ep.
\ee

We now ``balance'' $\hat N(\e,\hat \s)$ by sending $\hat N(\e,\hat \s)\to O(\e,\hat \s):=\cS(\hat \s)\hat N(\e,\s)\cS(\s)^{-1}$ with $\cS(\hat\s)$ given by
$$
\cS(\hat \s):=\bp i\hat \s & 0 \\ 0 & Id_{m}\ep.
$$
Expanding $O$ in powers of $\hat\s$ as $O(\e,\hat \s)=\hat \s O_1+\hat \s^2 O_2+O(\hat \s^3)$, we obtain
\ba\label{vectorO}
O(\e,\hat\s)&=i\hat\s \bp -2\kappa\big(\Im(a)+q\Re(a)\big) & A_0\big(\Im(d)+q\Re(d) \big)\\ \frac{4\kappa A_0\Re(a)}{m_0}h\e^{-1}+O(1) & \e^{-1}(f+2A_0hp)+O(1)\ep\\
&\quad+\hat \s^2\bp \frac{4\kappa\Re(a)A_0qph}{m_0}\e^{-1}+O(1) & \e^{-1}qp(f+2A_0hp)+O(1)\\ -\frac{8\kappa \Re(a)A_0^2hph}{m_0^2}\e^{-2}+O(\e^{-1}) & -\frac{2A_0}{m_0}hp(f+2A_0hp)\e^{-2}+O(\e^{-1})\ep+O(\hat \s^3).   
\ea

For the purposes of analyticity of the dispersion relations, we require the reduced convection matrix $f+2A_0hp$ to be invertible, with simple eigenvalues. 
Let us recall that the definition of $p$ implies that the reduced convection matrix is independent of $\kappa$ as $p=-\Re(d)/(2A_0\Re(c))$. Under this assumption we obtain $m+1$ analytic in $\hat \s$ dispersion relations which we denote $\l_t$ and $\l_{c,i}$, $i=1,...m$, admitting the expansions
\ba
\l_t(\e,\hat\s)&=i\a_t\hat\s+\mu_t\hat\s^2+O(\hat\s^3),\\
\l_{c,i}(\e,\hat\s)&=i\a_{c,i}\hat\s +\mu_{c,i}\hat\s^2+(\hat\s^3),
\ea
with $\a_{c,i}$ admitting the expansion
\be
\a_{c,i}=\e^{-1}\a_{c,i}^0+O(1),
\ee
where $\a_{c,i}^0\in\spec(f+2A_0hp)$. 

We note that $\a_t$ is real-valued by standard matrix perturbation arguments (as in the scalar case).
However, in the vector case, the $\a_{c,i}$ are not necessarily real-valued, hence we require
the additional, first-order stability condition \eqref{preccond} that $i\a_{c,i}$ are pure imaginary. 
If there are $\a_{c,i}$ with with nonzero real part,
corresponding to a conjugate pair of eigenvalues of the effective flux matrix
$( f+2A_0hp)$, then stability fails in the first-order term of the spectral expansion,
for $\hat \sigma$ of appropriate sign.
If on the other hand all $\a_{c,i}$ are real and distinct, then we may conclude by symmetry that not only the
leading order contribution, but the entire first-order terms of these eigenvalues are pure imaginary,
and thus stability holds (neutrally!) at first order (see Remark \ref{symmrmk}).
Thus, we require that the reduced convection matrix $f+2A_0hp$ be strictly hyperbolic, and noncharacteristic.
Again, this makes very much sense from the point of view of hyperbolic relaxation on systems, 
as the effective flux is none other than the convection matrix for the associated 
Chapman-Enskog equilibrium system, in which 
$\eps^{-1}$ term is required to be in equilibrium, or vanishing to lowest order (see Remark \ref{relaxrmk}).

To evaluate the $\mu_{c,i}$, we let $r_i$ and $\ell_i$ be the right and left eigenvectors of $f+2A_0hp$ respectively, and set
\be\label{vector-lrs}
r_{c,i}=\bp 0 \\ r_i\ep+O(\e),\qquad \ell_{c,i}=\bp c_i & \ell_i \ep+O(\e), 
\ee
where $c_i$ is a constant whose value is chosen so that $\ell_{c,i}$ is a left eigenvector of $O_1$. We remark that we will not need the precise value of $c_i$ in order to determine $\mu_{c,i}$ to leading order. To find $\mu_{c,i}$ to leading order, we compute $\ell_{c,i} O_2 r_{c,i}$, where we find
\be\label{vectormucs}
\mu_{c,i}=\e^{-2}\Big(-\frac{2A_0\a_{c,i}^0 \ell_i h pr_i}{m_0}  \Big)+O(\e^{-1}),
\ee
or after plugging in the definition of $p$ and $m_0$,
\ba\label{vectormucs2}
\mu_{c,i}&=   
\eps^{-2}\ell_i \frac{h\Re(d) }{2A_0^2 \Re(c)}\Big( f- h \frac{\Re(d)}{\Re(c)}\Big)r_i + O(\e^{-1})
\\& =\e^{-2}\frac{\a_{c,i}^0 \ell_i h\Re (d)r_i}{2A_0^2\Re(c)^2}+O(\e^{-1}). 
\ea

The other step in adapting Case (i) is to perform the matched asymptotics to find the coefficients $\a_t$ and $\mu_t$. As before, Kato-style expansions are much more complicated than matched asymptotics due to $\a_t$ and $\mu_t$ being lower order than $\a_{c,i}$ and $\mu_{c,i}$. In a similar fashion to the $m=1$ case, we can factor out an $i\hat \s$ from the last $m$ rows of $M(\e,\hat\s)$ and a copy of $i\hat\s$ from the second column, leading to an expansion of the form
$$
\det(M(\e,\hat\s)-\l_t Id_{m+2} )=(i\hat\s)^{m+1}\big(P_0+i\hat\s P_1+ H.O.T. \big),
$$
with $P_0$ given by
\be\label{vecP0}
P_0:=\det \bp 2A_0^2\Re(c) & -2\kappa\Re(a) & A_0\Re(d)\\ 2A_0^2\Im(c) & -2\kappa\Im(a)-\a_t & A_0\Im(d)\\ 2A_0 h\e^{-1}+2A_0\kappa\Im(g) & 2A_0\Im(g) & \e^{-1}f-\a_t Id_m\ep.
\ee
Let $f_1,...,f_m$ denote the eigenvalues of $f$. By our Kato expansion, we know that we can take $\a_t=O(1)$, and so the bottom right $m\times m$ block is invertible. Hence, by standard identities for block matrices
$$
\begin{aligned}
	P_0&=\Big(\prod_{i=1}^m\big(\e^{-1}f_i-\a_t \big)\Big)\det\Bigg(\bp 2A_0^2\Re(c) & -2\kappa\Re(a)\\ 2A_0^2\Im(c) & -2\kappa\Im(a)-\a_t\ep\\
	&\quad-\bp A_0\Re( d)\\ A_0\Im(d)\ep (\e^{-1}f-\a_t\Id_m)^{-1}\bp 2A_0h\e^{-1}+2A_0\kappa\Im( g) & 2A_0\Im(g)\ep   \Bigg).
\end{aligned}
$$

We note that
$$
\big(\e^{-1}f-\a_t Id_m \big)^{-1}=\e f^{-1}+O(\e^2),
$$
and so $P_0=0$ can be approximated by
$$
P_0=\det\Bigg(\bp 2A_0^2\Re(c) & -2\kappa\Re(a)\\ 2A_0^2\Im(c) & -2\kappa\Im(a)-\a_t\ep-\bp A_0\Re(d)\\ A_0\Im(d)\ep f^{-1}\bp 2A_0 h & 0\ep+ H.O.T.  \Bigg)=0.
$$
We then conclude that
\be\label{vectorat}
\a_t=-2\kappa\Im(a)+2\kappa\Re(a)\frac{\Re(c)-\Re(d)f^{-1}h}{\Im(c)-\Im(d)f^{-1}h }+O(\e), 
\ee
which matches the scalar case as $\hat c=c-df^{-1}h$ in the systems case.\\

Turning to $\mu_t$, we find that $P_1$ is given by
\ba\label{vecP1}
P_1&:=\det \bp -2\kappa\Im(a)-\a_t & -2\kappa\Re(a) & A_0\Re(d)\\ 2\kappa\Re(a) & -2\kappa\Im(a)-\a_t & A_0\Im(d)\\ 2A_0\Re(g) & 2A_0\Im(g) & \e^{-1}f-\a_t Id_m\ep\\
&\quad +
\det \bp 2A_0^2\Re(c) & -\Im(a) & A_0\Re(d)\\ 2A_0^2\Im(c) & \Re(a)+\mu_t & A_0\Im(d)\\ 2A_0 h\e^{-1}+2A_0\kappa\Im(g) & 0 & \e^{-1}f-\a_t Id_m\ep\\
&\quad+ \sum_{j=1}^m\det \bp 2A_0^2\Re(c) & -2\kappa\Re(a) & \widehat{A_0\Re(d)}_j  \\ 2A_0^2\Im(c) & -2\kappa\Im(a)-\a_t & \widehat{A_0\Im(d)}_j\\ 2A_0h\e^{-1}+2A_0\kappa\Im(g) & 2A_0\Im(g) & \cP_{j}\ep,
\ea
where for a vector $v$, $\widehat{v}_j$ denotes the vector $v$ with the $j$-th entry set to zero, and $\cP_j$ denotes the $m\times m$ matrix
$$
\cP_j:=\bp \e^{-1}f_{11}-\a_t & \e^{-1}f_{12} & ... & \e^{-1}f_{1j-1} & [e_B]_{1j} & \e^{-1}_{1j+1} & ... \e^{-1}f_{1m}\\ 
\e^{-1}f_{21} & \e^{-1}f_{22}-\a_t  & ... & \e^{-1}f_{2j-1} & [e_B]_{2j} & \e^{-1}_{2j+1} & ... \e^{-1}f_{2m}\\
\vdots\\
\e^{-1}f_{j1} & \e^{-1}f_{j2}& ... &\e^{-1}f_{jj-1} & [e_B]_{jj}+\mu_t & \e^{-1}f_{jj+1} & ... &\e^{-1}f_{jm}\\
\vdots\\
\e^{-1}f_{m1} & \e^{-1}f_{m2}& ... & \e^{-1} f_{mj-1} & [e_B]_{mj} & \e^{-1}f_{mj+1} & ... &\e^{-1}f_{mm}-\a_t\ep.
$$
In \eqref{vecP1}, we note that the dominant term is given by
$$
\begin{aligned}
	&\det \bp -2\kappa\Im(a)-\a_t & -2\kappa\Re(a) & A_0\Re(d)\\ 2\kappa\Re(a) & -2\kappa\Im(a)-\a_t & A_0\Im(d)\\ 2A_0\Re(g) & 2A_0\Im(g) & \e^{-1}f-\a_t Id_m\ep \\
	&\quad +
\det \bp 2A_0^2\Re(c) & -\Im(a) & A_0\Re(d)\\ 2A_0^2\Im(c) & \Re(a)+\mu_t & A_0\Im(d)\\ 2A_0 h\e^{-1}+2A_0\kappa\Im(g) & 0 & \e^{-1}f-\a_t Id_m\ep,
\end{aligned}
$$
as each entry in the sum over $j$ in \eqref{vecP1} has one less power of $\e$, most readily seen by taking the $j$-th entry and performing cofactor expansion along the $j$-th column. Indeed, we notice that as the first two entries in that column are zero, there are in each cofactor $m-1$ rows of size $\e^{-1}$ where as the first two matrices have $m$ rows of size $\e^{-1}$. Taking advantage of the same trick as we did for $P_0$, we observe that
$$
\begin{aligned}
	P_1&=\Big(\prod_{i=1}^m\big(\e^{-1}f_i-\a_t \big)\Big)\det\Bigg(\bp -2\kappa\Im(a)-\a_t & -2\kappa\Re(a) \\ 2\kappa\Re(a) & -2\kappa\Im(a)-\a_t \ep+O(\e)\Bigg)\\
	&\quad+\Big(\prod_{i=1}^m\big(\e^{-1}f_i-\a_t \big)\Big)\det\Bigg(\bp 2A_0^2\Re(c) & -\Im(a)\\ 2A_0^2\Im(c) & \Re(a)+\mu_t \ep-\bp A_0\Re(d)\\ A_0\Im(d)\ep f^{-1}\bp 2A_0h & 0 \ep+O(\e)\Bigg)\\
	&+O(\e^{m-1}).
\end{aligned}
$$
Setting $P_1=0$, and computing the leading $2\times 2$ determinants and solving for $\mu_t$, we obtain
\be\label{vecmut}
\mu_t=-\frac{\big(2\kappa\Im(a)+\a_t\big)^2+4\Re(a)^2\kappa^2+2A_0^2\big(\Re(a)\Re(\hat c)+\Im(a)\Im(\hat c) \big) }{2A_0^2\Re(\hat c) }+O(\e),
\ee
which we remark matches the scalar case.

\br
Much like the scalar case, the $\a_{c,i}$ and $\mu_{c,i}$ can also be obtained in this manner. However, unlike the scalar case, the procedure is not as simple. Indeed, our trick here was to invert the bottom right block of \eqref{vecP0}, however, if $\a_{c,i}\sim\e^{-1}$, then there is no guarantee of that holding. Instead, we use the invertibility of the top left corner to instead obtain
$$
\begin{aligned}
	P_0&=2A_0^2\Re(c)\det\Bigg(\bp -2\kappa\Im(a) & A_0\Im(d)\\ 2A_0\Im( g) & \e^{-1}f\ep-\frac{1}{2A_0^2\Re(c)}\bp 2A_0^2\Im(c)\\ 2A_0h\e^{-1}+2A_0\kappa\Im(g)\ep\bp-2\kappa\Re(a) & A_0\Re(d)\ep\\
	&\quad-\a_{c,i}Id_{m+1} \Bigg).
\end{aligned}
$$
Collecting the leading order terms, we discover
$$
P_0=2A_0^2\Re(c)\det\Bigg(\bp O(1) & O(1)\\ O(\e^{-1}) & \e^{-1}\Big(f-\frac{h\Re(d)}{\Re(c)}\Big)+O(1)\ep-\a_{c,i}Id_{m+1} \Bigg).
$$
As $P_0$ is approximately the determinant of a lower block triangular matrix, we conclude that $\a_{c,i}$ is an eigenvalue of $f-h\Re(d)/\Re(c)$ to leading order, matching our earlier conclusion.\\

Turning to $\mu_{c,i}$, the determinant independent of $\mu_{c,i}$ in \eqref{vecP1} is to leading order given by
$$
\begin{aligned}
	c_4&:=\det \bp -2\kappa\Im(a)-\a_{c,i} & -2\kappa\Re(a) & A_0\Re(d)\\ 2\kappa\Re(a) & -2\kappa\Im(a)-\a_{c,i} & A_0\Im(d)\\ 2A_0\Re(g) & 2A_0\Im( g) & \e^{-1}f-\a_{c,i} Id_m\ep\\
	&\quad =\a_{c,i}^2\det\big(\e^{-1}f-\a_{c,i}Id_m \big)+O(\e^{-(m+1)}).
\end{aligned}
$$
The first matrix outside of the sum yields a coefficient of $\mu_{c,i}$ of the form
$$
\det \bp 2A_0^2\Re(c) & A_0\Re(d)\\ 2A_0h\e^{-1}+2A_0\kappa\Im(g)& \e^{-1}f-\a_{c,i} Id_m\ep.
$$
The same block determinant trick that gave us $\a_{c,i}$ shows that this determinant is $O(\e^{-(m-1)})$, which is an acceptable error term for our purposes.

The remaining sum can be seen to give a coefficient of $\mu_{c,i}$ of the form
$$
c_3:=-2A_0^2\Re(c)\a_{c,i}\operatorname{Tr}\Big(\operatorname{cof}(\e^{-1}f-\a_{c,i}Id_m-\e^{-1}\frac{h\Re(d)}{\Re(c)})+ H.O.T.\Big),
$$
where $\operatorname{cof}(A)$ denotes the cofactor matrix of $A$. We recall from \cite{DPTZ} that the adjugate matrix, $\operatorname{adj}(A)=\operatorname{cof}(A)^T$, has the same left/right eigenvectors as $A$, with the eigenvectors $(\ell_i,r_i)$ having eigenvalue $\prod_{j\not=i}\l_j$, where $\l_i$ denotes the eigenvalues of $A$. In particular, the cofactor matrix has eigenvalues $\prod_{j\not=i}\l_j$, where $i=1,...,m$, and so by assuming $A$ has a single eigenvalue $\l_i=0$ we have that
$$
\operatorname{Tr}\Big(\operatorname{cof}(\e^{-1}f-\a_{c,i}Id_m-\e^{-1}\frac{h\Re(d)}{\Re(c)})+
H.O.T.\Big)=\prod_{j\not=i}\a_{c,j}+ H.O.T.\sim\e^{-(m-1)}.
$$
Hence, $c_3\sim\e^{-m}$ is the dominant coefficient of $\mu_{c,i}$ in $P_1=0$. To make $c_4$ more closely resemble $c_3$, we notice that $\e^{-1}f-\a_{c,i}Id_m$ is a rank-one update of $\e^{-1}f-\a_{c,i}Id_m-h\Re(d)/\Re(c)$, and so one can use the corresponding update formula for the determinants to get
$$
\det(\e^{-1}f-\a_{c,i}Id_m)=\det\Big(\e^{-1}f-\frac{h\Re(d)}{\Re(c)}-\a_{c,i}Id_m \Big)+\frac{1}{\Re(c)}\Re(d)\operatorname{adj}\big(\e^{-1}f-\frac{h\Re(d)}{\Re(c)}-\a_{c,i}Id_m \big) h.
$$
In particular, the leading order component of $\det(\e^{-1}f-\a_{c,i}Id_m)$ is then, after diagonalizing $\operatorname{adj}(f-h\Re(d)/\Re(c))$, given by
$$
\det(\e^{-1}f-\a_{c,i}Id_m)=\frac{1}{\Re(c)}\Re(d)r_i\ell_ih \prod_{j\not=i}\a_{c,j}+ H.O.T.
$$
from which we recover the Kato-style formula from $\mu_{c,i}=-c_4/c_3$.
\er

\subsection{Region (ii): nonexistence of imaginary eigenvalues}\label{s:transfer}
The first is to show that $M_0$ in case (ii) has no pure imaginary eigenvalues other than the simple
zero eigenvalue in its kernel.
To this end, we note that the zero e-vectors for $M_0$ are
$\ell=(\alpha, 1, \beta), r=(0, 1,0)^T$, which gives complementary projections
$$
\hbox{$L=\bp 1 & 0 & 0\\ 0 & 0 & 1\ep$  and $R=\bp 1 & 0\\ -\alpha & -\beta \\ 0 & 1\ep$}.
$$
This gives the reduced matrix
\be\label{redhatM0}
\hat M_0:= LM_0 R= \bp 2A_0^2 \Re(c) & A_0 \Re(d)\\ 2A_0 h i\hat \hat\sigma & fi\hat \hat\sigma \ep
\ee
complementary to the kernel of $M_0$.

What we need to show is that $\hat M_0$ has no imaginary eigenvalues $i\tau$ for $\hat \sigma\neq 0$,
or, equivalently,
\be\label{transcond}
\hbox{\rm
	$\Delta(\tau, \check \sigma):=
	\det \bp 2A_0^2\Re(c)-i\tau  & A_0 \Re(d)\\ 2A_0h\check \sigma& f\check \sigma-\tau \ep
	$
	has no real roots $(\tau, \check \sigma)$ for $\check \sigma \neq 0$.}
\ee

\bl\label{transferlemma}
Under nondegeneracy conditions \eqref{indcouple} and \eqref{shyp} below, along with $\Re \hat c\neq 0$,
condition \eqref{transcond} holds for all choices of model parameters.
\el

\begin{proof}
	Expanding $\Delta=0$, and setting real and imaginary parts separately to zero, we obtain for the imaginary 
	part
	\be\label{im}
	0= \det \bp -\tau  & A_0 \Re(d)\\ 0 \check \sigma& f\check \sigma-\tau \ep
	= \tau \det(f\check \sigma-\tau).
	\ee
	and (factoring out $2A_0$ from the first column and $A_0$ from the first row)
	\be\label{re}
	0= \det \bp 2A_0^2 \Re(c)  & A_0 \Re(d)\\ 0 \check \sigma& f\check \sigma-\tau \ep
	= \det \bp \Re(c)  &  \Re(d)\\ h \check \sigma& f\check \sigma-\tau \ep
	\ee
	for the real part.
	
	From \eqref{im}, either $\tau=0$, or else $\det(f\check \sigma-\tau)=0$.
	In the first case, \eqref{re} simplifies to
	$$
	0= \check \sigma \det \bp \Re(c)  &  \Re(d)\\ h & f\ep=\det(f) \Re (\hat c),
	$$
	which is excluded by our nondegeneracy condition $\hat c\neq 0$, $\det f\neq 0$,
	and by the assumption $\check \sigma \neq 0$.
	
	Thus, we need only examine the case $\det(f\check \sigma -\tau)=0$, or (using $\tau, \check \sigma\neq 0$)
	$$
	\tau/\check \sigma \in \sigma(f).
	$$
	Taking without loss of generality coordinates such that $f$ is diagonal, 
	$$
	f=\diag\{f_1,\dots, f_m\},
	$$
	impose the additional nondegeneracy conditions of ``individual coupling'':
	\be\label{indcouple}
	\hbox{\rm $h_j \Re (d_j)\neq 0$ for each $j=1,\dots,m$}
	\ee
	and strict hyperbolicity
	\be\label{shyp}
	\hbox{\rm $\spec(f)$ simple.}
	\ee
	Then, taking $\tau/\check \sigma= f_j$, without loss of generality $j=1$,
	and substituting in \eqref{re}, we obtain, factoring out $\check \sigma \neq 0$ from the final row,
	$$
	0= \check \sigma \det \bp \Re(c)  &  \Re(d_1) & \Re(d_2)&\dots & \Re(d_m)\\ h_1 & 0 &0& \dots & 0\\
	h_2 & 0 & f_2-f_1 & \dots & 0\\
	\vdots & \vdots & \vdots & \dots & \vdots\\
	h_m & 0 &  \dots &0& f_m-f_1\ep
	= \check \sigma  h_1 \Re(d_1) \det \bp   f_2-f_1 & \dots & 0\\
	\vdots &  \dots & \vdots\\ 0 &  \dots & f_m-f_1\ep,
	$$
	giving $0= \check \sigma h_1 \Re(d_1) \Pi_{j\neq 1}(f_j-f_1)$, 
	which, by $\check \sigma \neq 0$, \eqref{indcouple}, and \eqref{shyp}, is a contradiction.
\end{proof}

\subsection{Region (ii): affine dependence on $1/\check \sigma$}
The second is to show that the observation for $m=1$ that \eqref{l21} is affine in $\check\sigma^{-1}$ 
remains true in the vector case $m>1$. 
The same argument as in \ref{s:step2check} showing the denominator was proportional to $\check\sigma$ 
is easily adapted to the vectorial case to show that the corresponding denominator is proportional to 
$\check\sigma^m$ and is thus omitted. Focusing on the numerator and recalling \eqref{fcharexp1}, 
we see that the matrix $\tilde{M}(\rho;\check\sigma)$ whose determinant is in the numerator takes the block form
\begin{equation}\label{mblock}
	\tilde{M}(\rho;\check\sigma)=\bp
	\tilde{M}_{11}(\rho) & \tilde{M}_{12}\\
	\tilde{M}_{21}(\rho;\check\sigma) & \tilde{M}_{22}(\rho;\check\sigma)
	\ep,
\end{equation}
where $\tilde{M}_{12}$ is a $2\times m$ constant matrix, $\tilde{M}_{21}$ is a $m\times 2$ matrix taking the form
\begin{equation}\label{mblock21}
	\tilde{M}_{21}(\rho;\check\sigma)=\bp 2A_0i\check\sigma h & 0_{m\times1}\ep +\cO(\rho)_{m\times 2},
\end{equation}
and $\tilde{M}_{22}$ is an $m\times m$ matrix taking the form
\begin{equation}\label{mblock22}
	\tilde{M}_{22}(\rho;\check\sigma)=i\check\sigma f+\cO(\rho)_{m\times m}.
\end{equation}
In \eqref{mblock21} and \eqref{mblock22}, the subscripts denote the shape of the error terms. Of particular interest is the order  $\rho$ term in the determinant of $\tilde{M}(\rho;\check\sigma)$. Let us illustrate the computation in the $m=2$ case, the argument extends naturally to the higher $m$-case. Each column of $\tilde{M}(\rho;\check\sigma)$ may be written as
\begin{equation}\label{mcolumns}
	\tilde{M}_i(\rho;\check\sigma)=\tilde{M}_i^0(\check\sigma)+\rho\tilde{M}_i^1(\check\sigma).
\end{equation}
Hence by multilinearity of $\det$, we may compute the coefficient $\rho$ in $\det(\tilde{M}(\rho;\check\sigma))$ as
\ba\label{detexpanded}
	\frac{\d}{\d \rho}\det\tilde{M}(\rho;\check\sigma)\rvert_{\rho=0}&
	=\det\Big(\begin{bmatrix}\tilde{M}_1^1(\check\sigma) & \tilde{M}_2^0(\check\sigma) & \tilde{M}_3^0(\check\sigma) &\tilde{M}_4^0(\check\sigma)\end{bmatrix} \Big)+...\\
		&\quad +\det\Big(\begin{bmatrix}\tilde{M}_1^0(\check\sigma) & \tilde{M}_2^0(\check\sigma) & \tilde{M}_3^0(\check\sigma) &\tilde{M}_4^1(\check\sigma)\end{bmatrix} \Big),
\ea
where the $\tilde{M}_i^j$ are as in \eqref{mcolumns}. There are three types of terms in the sum in \eqref{detexpanded}, the first column having upper index one, the second column having upper index one, and the final $m$ columns having upper index one. For each summand in \eqref{detexpanded}, we use cofactor expansion along the column whose upper index is one, which we now illustrate for the first three columns when $m=2$. Starting with the first summand, we find
\begin{equation}\label{detexpandedcol1}
	\begin{aligned}
		&\det\Big(\begin{bmatrix}\tilde{M}_1^1(\check\sigma) & \tilde{M}_2^0(\check\sigma) & \tilde{M}_3^0(\check\sigma) &\tilde{M}_4^0(\check\sigma)\end{bmatrix} \Big)=c_{11}\det\bp * & * & *\\
		0 & i\check\sigma f_{11} & i\check\sigma f_{12}\\
		0 & i\check\sigma f_{21} & i\check\sigma f_{22}\ep\\
		&\quad -c_{21}\bp * & * & *\\0 & i\check\sigma f_{11} & i\check\sigma f_{12}\\
		0 & i\check\sigma f_{21} & i\check\sigma f_{22}\ep
		+c_{31}\det\bp * & * & *\\ * & * & *\\ 0 & i\check\sigma f_{21} & i\check\sigma f_{22}\ep-c_{41}\det\bp * & * & * \\ * & * & *\\ 0 & i\check\sigma f_{11} & i\check\sigma f_{12}\ep,
	\end{aligned}
\end{equation}
where the $*$'s denote suitable entries of $\tilde{M}_{11}$ and $\tilde{M}_{12}$, whose precise values do not matter save that they are independent of $\check\sigma$, and the $f_{ij}$ denote the entries of the matrix $f$. Finally, we have denoted the entries of the vector $\tilde{M}_1^1$ by $c_{j1}$, for $j=1,2,3,4$. Observe that we can always factor out exactly one or two powers of $\check\sigma$ in \eqref{detexpandedcol1}. For the second summand, we find
\begin{equation}\label{detexpandedcol2}
	\begin{aligned}
		&\det\Big(\begin{bmatrix}\tilde{M}_1^0(\check\sigma) & \tilde{M}_2^1(\check\sigma) & \tilde{M}_3^0(\check\sigma) &\tilde{M}_4^0(\check\sigma)\end{bmatrix} \Big)=c_{12}\det\bp * & * & *\\
		2A_0h_1i\check\sigma & i\check\sigma f_{11} & i\check\sigma f_{12}\\
		2A_0h_2i\check\sigma & i\check\sigma f_{21} & i\check\sigma f_{22}\ep\\
		&\quad -c_{22}\bp * & * & *\\2A_0h_1 i\check\sigma & i\check\sigma f_{11} & i\check\sigma f_{12}\\
		2A_0h_2i\check\sigma & i\check\sigma f_{21} & i\check\sigma f_{22}\ep
		+c_{32}\det\bp * & * & *\\ * & * & *\\ 2A_0h_2i\check\sigma & i\check\sigma f_{21} & i\check\sigma f_{22}\ep\\
&\quad -c_{42}\det\bp * & * & * \\ * & * & *\\ 2A_0h_1i\check\sigma & i\check\sigma f_{11} & i\check\sigma f_{12}\ep.
	\end{aligned}
\end{equation}
As with, \eqref{detexpandedcol1}, we find that \eqref{detexpandedcol2} is of the form $\check\sigma (a\check\sigma+b)$ for known complex coefficients $a,b$. For our final example column, we find
\begin{equation}\label{detexpandedcol3}
	\begin{aligned}
		&\det\Big(\begin{bmatrix}\tilde{M}_1^0(\check\sigma) & \tilde{M}_2^0(\check\sigma) & \tilde{M}_3^1(\check\sigma) &\tilde{M}_4^0(\check\sigma)\end{bmatrix} \Big)=c_{13}\det\bp * & * & *\\
		2A_0h_1i\check\sigma & 0 & i\check\sigma f_{12}\\
		2A_0h_2i\check\sigma & 0 & i\check\sigma f_{22}\ep\\
		&\quad -c_{23}\bp * & * & *\\2A_0h_1i\check\sigma & 0 & i\check\sigma f_{12}\\
		2A_0h_2i\check\sigma & 0 & i\check\sigma f_{22}\ep
		+c_{33}\det\bp * & * & *\\ * & * & *\\ 2A_0h_2i\check\sigma & 0 & i\check\sigma f_{22}\ep-c_{43}\det\bp * & * & * \\ * & * & *\\ 2A_0h_1i\check\sigma & 0 & i\check\sigma f_{12}\ep.
	\end{aligned}
\end{equation}
There is an additional simplification of \eqref{detexpandedcol3} coming from the observation that $\tilde{M}_{12}$ is independent of $\rho$ and so $c_{13}$ and $c_{23}$ vanish, leading us to conclude \eqref{detexpandedcol3} is proportional to $\check\sigma$. Similar considerations apply to the final term, where the fourth column has upper index one. Hence combining \eqref{detexpandedcol1}, \eqref{detexpandedcol2}, and \eqref{detexpandedcol3}, we find that \eqref{detexpanded} is of the form
\begin{equation}\label{newSigma}
	\frac{\d}{\d \rho}\det\tilde{M}(\rho;\check\sigma)\rvert_{\rho=0}=\check\sigma(A\check\sigma+B),
\end{equation}
for known complex constants $A$ and $B$. To extend to the general case, we observe that the analogs of \eqref{detexpandedcol1}, \eqref{detexpandedcol2}, and \eqref{detexpandedcol3} always have either $m$-rows proportional to $\check\sigma$ or $(m-1)$-rows proportional to $\check\sigma$, leading to the overall conclusion that the numerator is of the form $\check\sigma^{m-1}(A\check\sigma+B)$, which upon division by $\check\sigma^m$, gives us the desired conclusion that $\lambda_2$ is affine in $\check\sigma^{-1}$.

\medskip
This completes the proof of the second and final postponed computation from Remark \ref{affinermk},
thereby completing the proof of sufficiency in the vector case $m>1$.

\subsection{Final result}\label{s:fresvec}
With these modifications, we obtain the following vector version of Proposition \ref{stabcrit}.

\begin{definition}\label{astabdefvec}
	We define \emph{asymptotic diffusive stability} in the vector case as satisfaction of conditions
	\ref{tcond}, \ref{ccond}, and \eqref{preccond}.
	We define \emph{asymptotic instability} by failure of {\it asymptotic neutral stability} defined
	by satisfaction of \eqref{preccond} and
	\be
	\hbox{\rm $\Re \mu_t^0\geq 0$, $\Re \mu_{c,j}^0\geq 0$.}
	\ee
\end{definition}

\begin{proposition}\label{stabcritvec}
	Assuming the generic conditions of supercriticality \eqref{supercrit}, 
	hyperbolicity and noncharacteristicity of effective flux \eqref{splitass} and \eqref{fluxcond},
nontrivial Jordan structure \eqref{gencase}, and nonvanishing of $\Re \mu_t^0,\Re \mu_c^0$ in \eqref{genmu's},
	 asymptotic diffusive stability \eqref{astabdefvec},
	is \emph{necessary and sufficient} 
	for diffusive stability \eqref{dss} in the sense of Schneider of
	periodic (exponential) solutions \eqref{tw} of \eqref{ampeq} 
	for $m>1$, for $0<\eps\leq \eps_0$ sufficiently small, where $\eps_0$ is uniform on compact
	parameter sets satisfying the above assumptions.
\end{proposition}

\section{Rigorous validation}\label{s:validation}
Finally, we discuss the relation between solutions of \eqref{ampeq} and their behavior with 
convective Turing bifurcation in \eqref{bl}, as described in \cite{WZ3}.
Denote by $u=\bp u_1\\u_2\ep$, $f=\bp f_1 \\ f_2\ep$, $B=\bp B_1\\B_2\ep$ the decompositions in
nonconservative (i.e., the first $n-m$) and conservative (i.e., the last $m$) coordinates
in \eqref{bl}, with $f$, $B$, $R$ sufficiently smooth, and $B$ strictly parabolic,
and let $u(x,t)\equiv u^*$ be a constant, equilibrium solution of \eqref{bl} satisfying the Turing
assumptions \cite[Hypothesis 1]{WZ3}.
Note that parabolicity of $B$ implies parabolicity of $e_B$ and $\Re(a)>0$ in \eqref{ampeq},
by the recipe given in \cite{WZ3}, verifying our assumptions on \eqref{ampeq}.

Under \cite[Hypothesis 1(2)]{WZ3}, we have $R_1$ full rank with respect to $u_1$, hence, local to
$u^*$, there is a function $u_1=\phi^*(u_2)$ for which $R_1(\phi^*(u_2), u_2)\equiv 0$, uniquely
determining $u^*_1$ as a function of $u_2^*$. Hence, by a coordinate shift $u_1 \to u_1-\phi^*(u_2)$,
$u_2\to u_2- u_2^0$, where $u_2^0$ is some base point under consideration,
we may assume without loss of generality $u_1^*=0$ for any such equilibrium state, while $u_2^*$ varies
freely within an open set of $u_2^0$.
Linearizing \eqref{bl} about the constant solution $u^*$, we assume that the dispersion relation of
the associated constant-coefficient symbol has strictly negative spectrum for Fourier frequencies $k\neq 0$
for negative values of the bifurcation parameter $\nu$, and at the bifurcaton point $\nu=0$
there exists simple eigenvalues $\lambda =\pm i\tau$ at $k=\pm k_*\neq 0$ and an m-fold semisimple
eiganvalue $\lambda=0$ at $k=0$, with all other spectra strictly negative; moreover, we assume
that the real part of the spectrum has second-order contact in $k$ with the imaginary axis at $k=0, \pm k_*$,
departing to second order as $k$ is varied from $0$ or $\pm k_*$, and first-order contact at $k=\pm k_*$,
growing linearly in $\nu$ as $\nu$ increases through the bifurcation value $\nu=0$.
For the structure assumed in \eqref{bl}, these conditions are equivalent to the Turing assumptions of \cite{WZ3}.

Then, by the results of \cite{WZ3}, there exists a self-consistent and well-posed multi-scale expansion of form
\eqref{ampeq}, \eqref{eq:ansatz} formally governing small-amplitude solutions for $\nu>0$ sufficiently small.
Moreover, this expansion may be continued to all orders.
The coefficients may be determined as described in \cite[\S 3]{WZ3}, with linear terms coming from the
eigenstructure of the linear dispersion relation considered as a function of $k$, $\nu$, and $B_0$; 
however, the details of this recipe will not concern us here, other than the fact that generic coefficients
of \eqref{ampeq} induce generic coefficients of \eqref{ampeq}, so that generic assumptions on \eqref{bl}
may be made via assumptions on \eqref{ampeq}.

\subsection{Existence}\label{s:valexist}
We begin by recalling (a slightly refined version of) the rigorous validation result
established previously in \cite{WZ3} in the context of existence of periodic traveling waves.
Refining a bit the description of \cite[Thm. 6.5]{WZ3}, we have the following result
asserting existence of exact solutions nearby the asymptotic solutions predicted by \eqref{ampeq}, \eqref{tw}.

\begin{theorem}[Existence of exact solutions]\label{main1}
	Under the above-described Turing Hypotheses (encompassing those of \cite{WZ3}),
	with $\nu=\eps^2$, for $\kappa$, $B_0$ lying in any compact subset of domain \eqref{solndom},
	for $0<\eps\leq \eps_0 $ sufficiently small
	there exists a smooth family of traveling-wave solution 
	\be\label{twexact}
	u(x,t)=\bar u^{\eps,\kappa, B_0} (kx+\bar \Omega t), \qquad k=k_*+ \eps \kappa
	\ee
	of amplitude $|u|\sim \eps$ of \eqref{bl}, lying within $O(\eps^2)$ of
	approximate solution \eqref{eq:ansatz}, \eqref{tw} with $\omega$ in \eqref{tw} replaced
	by $\tilde \omega=\omega + O(\eps^3)$ (thus determining $\bar \Omega$ up to $O(\eps^3)$.
	Moreover, up to translation, these are the unique nontrivial small-amplitude periodic traveling waves 
	of \eqref{bl}.
\end{theorem}

\begin{proof}
	The version of this theorem stated in \cite{WZ3} is for $(\kappa, B_0)$ sufficiently small.
	Since choice of the center $B^*$ is arbitrary, that there is no loss of generality in fixing
	$B_0\equiv 0$, so long as all estimates are uniform, so that small $B_0$ is no real restriction.
	Validation on the full existence domain \eqref{solndom}, along with the additional detail given
	here on the speed of the wave, then follows as in \cite[Thm. 1.2]{WZ1} for the case without
	conservation laws. We omit the further details of this involved but by 
	now standard (in particular, nonsingular)
	argument. See also the comments of \cite[\S 1.5]{WZ1} on nonlocal reduction of
	the conservation law case to the standard one, which gives another route to this result
	via \cite[Thm. 1.2]{WZ1}, which as noted in \cite{WZ1} applies also to nonlocal equations. 
\end{proof}

\subsection{Stability}\label{s:valstab}
We next turn to the question of stability of exact periodic solutions of \eqref{bl}, 
and its relation with stability of periodic solutions \eqref{tw} of \eqref{ampeq}.
Following \cite{JNRZ}, 
let $\mathcal{L}^{\eps,\kappa, B_0}$ denote the linearized operator about an exact periodic traveling-wave
solution $\bar u^{\eps,\kappa, B_0}$ considered in a comoving frame for which it becomes stationary.
The $L^2(\R)$ spectrum of $\mathcal{L}^{\eps,\kappa, B_0}$ consists of $\lambda$
such that there exists a solution $v(x)$ on a single period of
\be\label{interior}
\mathcal{L}^{\eps,\kappa, B_0}v=\lambda v, \qquad x\in [0,X]= [0,1/\kappa],
\ee
satisfying for some Bloch-Floquet number $\sigma \in [-\pi/X, \pi/X)$ the boundary conditions
\be\label{twist}
v(X)e^{i\sigma  X}=v(0).
\ee

Diffusive spectral stability in the sense of Schneider is defined in this context as
\be\label{dssBF}
\hbox{\rm $\Re  \lambda \leq c(\eps)| \sigma|^2/(1+| \sigma|^2)$ \; 
for $\lambda \in \spec(\mathcal{L}^{\eps,\kappa, B_0})$,}
\ee
where $\nu=\eps^2$.
In the absence of neutral spectra other than the $m+1$ translational/conservational modes
at $(\sigma, \lambda)=(0,0)$, \eqref{dssBF} is
necessary and sufficient for linearized and nonlinear stability \cite{JNRZ}.
By continuity of spectra for parabolic operators, diffusive spectral stability holds automatically
for $0<\eps\leq \eps_0$ sufficiently small, except possibly for 
\be\label{apriori}
|\lambda|, |\sigma|\leq 1/C, 
\ee
where $C>0$ is arbitrary.

Expanding $ \mathcal{L}^{\eps,\kappa, B_0}= \mathcal{L}^{0,0 , B_0} + O(\eps)$, one may then perform
as in \cite{WZ3} a Lyapunov-Schmidt reduction, following the approach laid out in \cite{M1,M2,M3,S1,S2}, to
obtain a reduced problem consisting in the rescaled $\check \lambda, \check \sigma$ variables 
of an $(m+2)\times (m+2)$ linear system of equations 
\be\label{redeq}
		\Big[-\check\l+\check M(\e,\check \s,\kappa,B_0)+
		\check \cE(\e,\check \s,\kappa,B_0,\check \l) \Big]\bp a_1\\a_2\\ \vec{b}\ep=0,
\ee
with $a_j\in \R$, $\vec{b}\in \R^m$, corresponding
to the $m+2$ dimensional kernel of $\mathcal{L}^{0,0,B_0}$ at $\sigma=0$,
where $\check M$ is as in Section \ref{s:suff}.

That is, first, the question of diffusive stability is reduced to determining diffusive spectral stability
of the $m+2$ neutral eigenvalues branching as $\eps$ increases from zero
from the kernel of the linearized operator about the constant
state $(0,B_0)$ at bifurcation point $\eps=0$.
And, second, these neutral eigenvalues may be identified as solutions of the reduced problem
problem \eqref{redeq}, a nonlinear eigenvalue problem that is a
perturbation by $\cE$ of the matrix perturbation problem corresponding to stability of
traveling waves of \eqref{ampeq}.

We collect here a streamlined version of the results established in \cite[Thm. 7.11 and Thm 8.7]{WZ3}
and in the course of their proofs (in the case of Thm. 8.7, see the discussion just below).\footnote{
	See also the explicit computations of \cite[Appendix A.1]{Wh} for an example model with $m=1$.}

\begin{proposition}[Truncation error bound \cite{WZ3}]\label{truncation_bd}
	For $0<\eps\leq\eps_0$ sufficiently small,
	diffusive stability of periodic solutions of $\bar u^{\eps,\kappa, B_0} (kx+\bar \Omega t)$
	is equivalent to diffusive stability for $\sigma$, $\lambda$ arbitrarily small of 
	$m+2$ ``neutral'' modes $\check \Lambda_j$ satisfying
	the reduced eigenvalue problem \eqref{redeq}, where $\check M$ as in Section \ref{s:suff} is
	the eigenvalue problem associated with \eqref{ampeq}, and 
\be\label{totalerror_check}
|\check\cE|= O(\eps, \eps^2 \check \sigma, \eps^3 \check \sigma^2, \eps^4 \check \sigma^3, 
	(\eps^4 + \eps^6\check \sigma^2)  \check \lambda^2).
\ee
\end{proposition}

\begin{proof}
In \cite{WZ3}, there is demonstrated the stronger truncation error bound
$$
|\check\cE|= O(\eps^2, \eps^2 \check \sigma, \eps^3 \check \sigma^2, \eps^4 \check \sigma^3, 
(\eps^4 + \eps^6\check \sigma^2)  \check \lambda^2)
$$
	for the ``tilde model'', a related, higher-order expansion of \eqref{ampeq} (denoted the
	``truncated model'' in \cite{WZ2}).
	Accounting for the difference in truncated vs. tilde models gives an additional $O(\eps)$ error
	contribution, resulting in \eqref{totalerror_check}.
See \cite[Appendix A.1]{Wh}, for a detailed computation of errors for the truncated model
in an example case with $m=1$.
\end{proof}

\br\label{Ermk}
Expressing \eqref{redeq} in the ``original'' coordinates 
$(\check \lambda, \check \sigma)= \eps^{-2}(\lambda, \sigma)$ natural
to PDE \eqref{bl}, as the system for \eqref{ampeq} perturbed by error $ \cE(\e,\s,\kappa,B_0,\l)$, 
	\eqref{totalerror_check} translates to
\be\label{totalerror}
 |\cE|=O(\eps^3, \eps^2 \sigma, \eps \sigma^2, \sigma^3, (\eps^2 + \sigma^2) \lambda^2).
\ee
The corresponding estimate for the tilde model becomes
 $|\cE|=O(\eps^4, \eps^2 \sigma, \eps \sigma^2, \sigma^3, (\eps^2 + \sigma^2) \lambda^2)$,
 and the difference between truncated and tilde models $O(\eps^3)$.
\er


The estimate \eqref{totalerror_check} is effectively a ``truncation error'' or residual bound.
What is needed to determine stability for \eqref{bl} is to convert this truncation error to ``convergence error,''
or difference between $\Lambda_j$ and the solutions $\lambda_j$ of the unperturbed system corresponding
to stability for \eqref{ampeq}, an issue requiring a detailed analysis of the eigenstructure of $\check M$. 
In this regard, it is worth mentioning that, while the estimates \eqref{totalerror_check} come via
\eqref{totalerror} through Taylor expansion valid on the entire regime $\lambda$, $\sigma$ small, the 
solutions $\check \lambda_j$ have Taylor expansions valid only only on the much smaller
regime $\check \lambda=\sigma/\eps^2 $, $\check \sigma=\sigma/\eps^2$ sufficiently small \cite[Thm. 5.9 and 
discussion in proof]{WZ3}.

Such an analysis was carried out in \cite{WZ3} on the region of analyticity
$|\check \sigma|\leq 1/C$ of Case (i), by essentially the same argument followed in Section \ref{s:nec}.
However, the corresponding treatment of remaining regions was left as an important and apparently difficult
open problem. Indeed, the possibility of completing such an analysis was conjectured in \cite{WZ3},
somewhat optimistically, on the basis of numerical evidence, with a positive outcome far from clear,
either for closeness of $\Lambda_j$ and $\lambda_j$, or for determination of practical stability criteria
for the $\lambda_j$.

The latter problem we have resolved in Section \ref{s:suff}, showing that the stability requirements
already determined in \cite{WZ3} on the analyticity region $|\check\sigma|\leq 1/C$ are in fact
sufficient for stability on all regions.
But, the $\check \cE=0$ case studied there was already a matrix perturbation problem, 
hence in the course of this analysis we have had to analyze the eigenstructure of the principal
parts of $\check M$ in various regimes in order to absorb higher-order truncation errors arising 
from various spectral expansions.
Hence, the analysis already completed for the second problem is also sufficient to resolve the first problem,
provided only that we show that errors \eqref{totalerror_check} are also absorbably small.

By this approach, we obtain as a corollary our second, and final, main theorem resolving stability.


\bt[Stability of exact solutions]\label{main2}
Under the Turing hypotheses described above, with $\nu=\eps^2$ and $\kappa$, $B_0$ satisfying 
\eqref{solndom}, let $\bar u^{\eps,\kappa, B_0} (kx+\bar \Omega t)$ be the exact periodic solutions \eqref{twexact}
of PDE \eqref{bl} guaranteed by Theorem \ref{main1}, and $(\bar A, \bar B)^{\kappa,\omega}=(A_0 e^{i(\kappa x - \omega t)}, B_0)$
the corresponding periodic solutions \eqref{tw} of the associated amplitude equations \eqref{ampeq}.
Then, under the nondegeneracy conditions of Proposition \ref{stabcrit}, diffusive stability of
$\bar u^{\eps,\kappa, B_0}$ is equivalent to diffusive stability of 
$(\bar A, \bar B)^{\kappa,B_0}=(A_0 e^{i(\kappa x - \omega t)}, B_0)$
for $0<\eps\leq \eps_0$ sufficiently small, which in turn is equivalent to the linear algebraic conditions
\eqref{tcond} and \eqref{preccond}-\eqref{ccond}.
Here, as elsewhere, $\eps_0$ may be chosen uniformly for compact parameter sets on which the Turing and
nondegeneracy assumptions are satisfied.
\et

\begin{proof}
Evidently, it suffices to establish necessity and sufficienty of conditions 
\eqref{tcond} and \eqref{preccond}-\eqref{ccond} 
for stability of $\bar u^{\eps,\kappa, B_0}$, 
as we have shown already for stability of $(\bar A, \bar B)^{\kappa,B_0}$.
We study separately the regions described in cases (i)--(vi) of Section \ref{s:suff}.
To prove necessity of conditions \eqref{tcond} and \eqref{preccond}-\eqref{ccond}, 
it is sufficient to show on just one of these regions, that eigenvalues $\Lambda_j$ and $\lambda_j$ 
are sufficiently close that their real parts have the same sign. 
To prove sufficiency, we must show this on each of the regions of cases (i)-(vi). 

\smallskip \noindent
	\textbf{Case (i)}, ($|\sigma|\leq \eps^2/C$):  
	In \cite[Thm. 8.7]{WZ3}, it is shown for the tilde model that 
	$ |\Lambda_j-\lambda_j|=O(\eps^4, \eps \sigma, \eps \sigma^2, \sigma^3)  $
	for all $j$; moreover \cite[Thm. 8.7, final line]{WZ3}, for $\lambda_t$ and $\lambda_c$, for
	which both $\lambda_j(\eps, 0)=0$ and $\Lambda_j(\eps, 0)=0$,
	$ |\Lambda_j-\lambda_j|=O(\eps \sigma, \eps \sigma^2, \sigma^3).  $
	For, both $\lambda_j$ and $\Lambda_j$ are analytic on this regime, hence there can be
	no error term of order $\eps^4$, nonvanishing at $\sigma=0$.
	The same argument applies for the truncated model \eqref{ampeq} studied here, but
	with $\eps^4$ order errors replaced by $\eps^3$.
	Since both $\lambda_j$ and $\Lambda_j$ have first-order Taylor coefficient
	pure imaginary (also shown in \cite{WZ3}), we have by similar reasoning that 
	\be\label{ierrest}
	|\Re	\Lambda_j-\Re \lambda_j|=O(\eps \sigma^2, \sigma^3)=o(\min\{\eps^2 , |\sigma|^2\}),
	\ee
	as there can be no error term in real parts of order $\eps \sigma$ nonvanishing to order $\sigma$.

	On the other hand, the analysis of Section \ref{s:suff}, case (i) gives under
	our nondegeneracy assumptions that the stable eigenvalue
$\lambda_s$ has real part negative of order $-\eps^2$, the translational eigenvalue $\lambda_t$
	has real part of size $ \eps^{-2} \sigma^2$, and the conservational eigenvalues $\lambda_c$ have real part of
	size $\eps^{-4}\sigma^2$ in the (original, pde) coordinates
	$(\lambda, \sigma)=\eps^2(\check \lambda, \check \sigma)$;
	see \eqref{rescdisp1}-\eqref{tgenmu's1}.
	(Note, as $|\sigma|\ll \eps^2$ on this region, that all $|\lambda_j|\ll 1$, despite large coefficients.)
	Comparing with \eqref{ierrest}, we thus see that the real parts of $\lambda_j$ and $\Lambda_j$
	have common sign in each case.

	\medskip
	The analysis of case (i), carried out previously in \cite{WZ3}, both establishes {\it necessity}
	of conditions \eqref{tcond} and \eqref{preccond}-\eqref{ccond}, and reduces the study of
	{\it sufficiency} to the examination of cases (ii)-(vi).
	In the rest of the proof, we carry out this remaining part, showing sufficiency
	in each of the cases (ii)-(vi).

	\medskip

	\noindent
	\textbf{Case (ii)}.  ($1/C\leq |\check \sigma|\leq C $):  
	This case for \eqref{ampeq} is about continuity of spectra of $M_0$ with respect
	to ``errors'' $\rho M_1$ and $\rho^2M_2$ that are merely $o(1)$.
	So, to accomodate also the additional error $\check \cE$, we have only to show that
\eqref{totalerror_check} is small on this regime, and absorbably in the errors induced by $M_1$ and
$M_2$.
	We first observe, by $\check \lambda V= o(|\check \lambda|)V + (M_0 + o(1))V$ for $V\neq 0$ that 
	$(1+o(1))|\check \lambda|\leq |M_0+o(1)|\lesssim C$, hence $\check \lambda= O(1)$,
	and thus by \eqref{totalerror_check} 
	\be\label{crudecheck}
	|\check \cE|= 
	O(\eps, \eps^2 \check \sigma, \eps^3 \check \sigma^2, \eps^4 \check \sigma^3, 
	(\eps^4 +\eps^6\check \sigma^2) \check \lambda^2)=
	O(\eps^2, \eps^2 \check \sigma, \eps^3 \check \sigma^2, \eps^4 \check \sigma^3, o(\check \lambda) )= o(1).
	\ee

	Thus, the induced errors are indeed small, preserving the order one separation between the kernel of
	$M_0$ and remaining eigenvalues featuring a spectral gap. It follows that the ``small'' eigenvalue
	bifurcating from the kernel remains analytic in $\rho$, uniformly in $\check \sigma$, and the remaining
	``order one'' eigenvalues possess a spectral gap, having real parts negative and uniformly bounded from zero.

	By the $o(1)$ error estimate \eqref{crudecheck}, the order one eigenvalues with spectral gap retain
	this spectral gap under perturbation, by simple continuity of spectra, and so $\Re\Lambda_j$ and 
	$\Re \lambda_j$ have the same signs for these modes, carrying the same stability information.

	For the remaining ``small'' eigenvalue $\lambda_t$, we have
	$\Re \lambda_t\sim \rho^2=\eps^2\check \sigma^2$, we must look more carefully,
	as potential errors of order $\eps^2$ or $\eps^2\check \sigma$ are much greater than
	$\Re \lambda_t$ near the lower boundary $|\check  \sigma|=1/C$ where $|\check \sigma|\ll 1$.
	However, the same reasoning shows that such error terms therefore cannot occur, by matching
	at the boundary with region (i).  For, as we have demonstrated analyticity of the projector
	onto the small eigenmode, and as the initial error $\check \cE$ is analytic on all $|\sigma|\ll 1$, we 
	have that the projected error $\check \cE_t$ in the reduced problem for $\Lambda_t$ is analytic
	as well, as a function of $\check \sigma, \eps, \rho$, for $\s$ in region (ii), and indeed
	for complexified $\check \sigma$, $\eps$, for $1/C\leq |\s|\leq C$ and $|\eps|\ll 1$.
	But, by the balancing transformation argument in region (i), we have already shown that the projector
	onto the small, $\lambda_t$ mode is analytic in (complexified) $\s$, $\eps$ for 
	for $ |\s|\leq 1/C$ and $|\eps|\ll 1$, hence the projector, and projected error are analytic for
	$|\s|\leq C$ and $|\eps|\ll 1$, so that {\it $\lambda_t$ and $\Lambda_t$ are analytic as well, with
	convergent power series representations about $(\s,\eps)=(0,0)$.}
	We may thus obtain the desired sharpened estimates from the power series analysis of region (i), which shows
	that ``harmful'' order $\eps^4$ or $\eps^2\s$ terms do not occur.

	Note, in showing analyticity of the complexified $\lambda_t$ projector on region (ii), we require that
	$\hat M_0$ in \eqref{Ms1} have a kernel of dimension one, i.e., that the complementary matrix
$$ \hat M_0:= LM_0 R= \bp 2A_0^2 \Re(c) & A_0 \Re(d)\\ 2A_0 h i\hat \hat\sigma & fi\hat \hat\sigma \ep$$
derived in \eqref{redhatM0} be invertible for $\s\neq 0$. But, direct computation gives
$$
\det \hat M_0= 2A_0^2i\s \Re(c)\det\Big( f - h\Re(d)\Re(c) \Big)\neq 0
$$
for $\s\neq 0$, since $\Re(c)=0$ and $\det ( f - h\Re(d)\Re(c) )=0$ are excluded 
by our previous nondegeneracy assumptions.

\smallskip\noindent
	\textbf{Case (iii)}. ($C\leq |\check \sigma|\leq 1/C\eps $):  
	Similarly as in Case (ii), the treatment of Case (iii) in Section \ref{s:suff}
	requires only that errors be $o(1)$, much smaller than spectral separation between small
	eigenvalue $\lambda_t$ and remaining eigenvalues.
	Starting with $\check \lambda V= o(|\check \lambda|)V + (M_0 + o(1))V$ for $V\neq 0$, giving 
	$(1+o(1))|\check \lambda|\leq |M_0+o(1)|\lesssim \check \sigma$, hence $\check \lambda= O(\check \sigma)$,
	we find that 
	$$
	\eps^4 |\check \lambda|^2 =O(\eps^4 \check \sigma^2)=O(1/C)=o(1).
	$$

	Likewise,
	$$
	O(\eps, \eps^2 \check \sigma, \eps^3 \check \sigma^2, \eps^4 \check \sigma^3, \eps^4\check \sigma^2) )=
	O(\eps, \eps/C, \eps/C^2 , \eps/C^3, \eps^2/C^2)=o(1),
	$$
confirming smallness of remaining error terms.	
Finally, we note that the ``small'' $\lambda_t$ mode has real part 
$\sim \rho^2= \eps^2|\check \sigma|^2\gg \eps^2 |\check \sigma|$
on this region, hence can absorb all error terms other than constant, $O(\eps)$ ones. 
Observing that analyticity of $\lambda_t$ and $\Lambda_t$ holds up to $|\rho|\ll 1$, by the same
argument as in region (ii), we find again that such an error term cannot exist.

\smallskip\noindent
	\textbf{Case (iv)}.  ($1/C\eps \leq |\check \sigma|\leq 1/C\eps^2$): 
	This case is straightforward. For, both symbol and real parts of eigenvalues are of the same
	order $(1+|\rho|^2)$, hence we need only show that truncation errors are order $o(1+|\rho|^2)$
	in order to see by simple continuity of eigenvalues under perturbation
	that $\Re\Lambda_j$ and $\Re \lambda_j$ have the same signs.

	Here, $|\rho|=\eps |\check \sigma|\geq 1/C$,  while $|\check \sigma|\gg 1$
	and $\eps |\rho|=\eps^2\check \sigma|\ll 1$.
	Thus, $\lambda$-errors in $\check \cE$ are bounded by $\eps^4 |\check \lambda|^2=|\lambda|^2=o(1)$
	and $\eps^6|\check \sigma|^2 |\check \lambda|^2=|\rho|^2\eps^4 |\check \lambda|^2=o(\rho^2)$,
	both $o(1+|\rho|^2)$.
	Likewise, the remaining errors in $\check \cE$ are bounded by
	$$
	O(\eps, \eps^2 \check \sigma, \eps^3 \check \sigma^2, \eps^4 \check \sigma^3, \eps^4\check \sigma^2) )=
	o(1,1,|\rho|^2, |\rho|^2, |\rho|^2)=o(1+|\rho|^2),
	$$
	so again are absorbable.

\smallskip\noindent
	\textbf{Case (v)}.  $ 1/C\eps^2\leq |\check \sigma|\leq C/\eps^2$:
	This case was carried out for \eqref{ampeq} in the ``original PDE coordinates'' 
	$\lambda, \sigma$ coordinates with $1/C\leq | \sigma|\leq C$,
featuring a spectral gap of order $\sigma^2\geq 1/C^2$ for eigenvalues $\lambda_j$.
The error \eqref{totalerror} of
 $O(\eps^3, \eps^2 \sigma, \eps \sigma^2, \sigma^3, \lambda^2)$ need thus be only $o(\sigma^2)$, or,
 sufficiently, $o(1)$.
 Recalling that $|\lambda|=o(1)$, we see that this is evidently so.

\smallskip\noindent
	\textbf{Case (vi)}.  $ C/\eps^2\leq |\check \sigma|$:
In this case, working with the original PDE coordinates $(\sigma, \lambda)=\eps^2(\check \sigma , \check \lambda)$,
we have $|\sigma|\geq C$.
Recall that this case is always stable for \eqref{ampeq}.
Likwise, as noted in \eqref{apriori}, the spectra for the exact PDE problem is also automatically stable
in this regime. Thus, stability of $\lambda_j$ and $\Lambda_j$ are (trivially) equivalent.
\end{proof}

\br\label{lambda_texp}
An examination of the analyticity argument for the projector onto the $\lambda_t$ mode in region (ii) 
shows that the argument can in fact be extended also to all of region (iii).
This shows that $\lambda_t$ and $\Lambda_t$ have convergent analytic expansions
in $\hat \sigma$, $\eps$ for $|\hat \sigma|\leq 1/C$ and $\eps\ll 1$, similarly as
in the nonconservative case.
In particular, it gives an additional verification that the second order coefficient $\mu_t^0$
derived on region (i) is valid also on region (iii), for which the Dary approximation is valid.
This gives an additional proof that the descriptions of $\lambda_t$ behavior for Eckhaus and Darcy
approximations agree.
\er


%

\appendix

\section{Turing bifurcation for example model}\label{s:Teg}
For the model problem \eqref{MOeq}. 
homogeneous equilibrium states take the form $(\rho_0, u_0, c_0)$, with  
\be\label{eqeg}
u_0=0,\quad  c_0=\alpha \tau \rho_0,
\ee
without loss of generality (rescaling parameters if necessary) $\rho_0=1$.
For simplicity in writing, let us fix $\mu=\nu=0$; as we note at the end, this
does not affect the end result.

Linearized about such a state, \eqref{MOeq} then becomes
\be\label{lineg}
\bp \rho\\u\\c\ep_t= \bp 0&0&0\\0&-\gamma&0\\\alpha&0&-1/\tau\ep \bp \rho\\u\\c\ep
+
\bp 0&-1&0\\-2A&0& \beta\\ 0&0&0 \ep \bp \rho\\u\\c\ep_{x}
+
\bp 0&0&0\\ 0&0&0\\ 0&0&D\ep \bp \rho\\u\\c\ep_{xx}.
\ee

The associated dispersion relation $\lambda=\lambda(k)$ is thus given by
\be\label{dispeg}
0=\det
\bp -\lambda & -ik & 0\\ -2Aik & -\gamma -\lambda & \beta ik\\ \alpha & 0 & -k^2D-1/\tau -\lambda
\ep.
\ee
Examining \eqref{dispeg} at $k=0$, we see that the eigenvalues are $0, -\gamma, -1/\tau$, 
so that there is a single ``critical'' analytic branch $\lambda_*(k)$ with $\lambda_*(0)=0$.
Moreover, it is readily seen by matching common orders of $k$ that $\lambda_*$ vanishes to first order as well.
Substituting 
$$
\lambda_*(k)= \theta k^2 + O(k^3)
$$
into \eqref{dispeg} and collecting terms of order $k^2$,
we find that $(-1/\tau)(\theta k^2 \gamma + Ak^2)+\beta\alpha k^2=0$, or
\be\label{thetaform}
\theta = \alpha \beta \tau -A.
\ee
Thus, diffusive stability of the dispersion relation near $k=0$ is equivalent to 
\be\label{disscriteg}
A>\alpha \beta \tau.
\ee

On the other hand, for $\alpha\beta=0$, the linearized matrix becomes block triangular, and
we can solve the dispersion relation explicitly as
\be\label{abzero}
\lambda= -1/\tau - Dk^2, \frac{-\gamma \pm \sqrt{\gamma^2 -8Ak^2}}{2},
\ee
yielding strict diffusive stability by inspection. Thus, strict diffusive stability holds for $\alpha\beta$
sufficiently small, while, by \eqref{disscriteg}, fails for $\alpha\beta$ sufficiently large.

We may conclude therefore that a Turing bifurcation occurs for some value $(\alpha \beta)_*>0$.
Essentially the same argument yields the result for arbitrary $\mu, \nu >0$ as well.
It would be very interesting to determine the nature and properties of this bifurcation, in the
spirit of the current analysis.

\section{Numerical illustration, case $m=1$}\label{s:num}
In the below figures, we display the results of a numerical comparison of the
spectra of the linearized equations for \eqref{ampeq} and the results predicted by Taylor expansion,
for a generic example system with $m=1$.
The parameters chosen are $a=1+i$, $b=1$, $c=-3+2i$, $d=-1+2i$, $e_B=1$, $f=1$, $g=2+2i$, $h=2$, $\eps=10^-2$.
In the illustrations, blue dots denote the real parts of (numerically approximated) true eigenvalues, 
	green, predictions from $\mu_t$, and red, predictions from $\mu_c$, with lefthand
	panel displaying results for small $ \sigma$ and righthand large $ \sigma$,
	and each pair of panels associated with a different wave-number $\kappa$.\\
	
	In the left column, we've plotted the curves on $\abs{\hat\sigma}\leq10\e$ and in the right column, we have the curves on $\abs{\hat\sigma}\leq1$. We note the strong agreement between $\mu_c\hat\sigma^2$ and $\Re\lambda_c(\hat\sigma)$ on $\hat\sigma=o(\e)$ and good agreement between $\mu_t\hat\sigma^2$ and $\Re\lambda_t(\hat\sigma)$ on the region where $\hat\sigma=o(1)$, as is expected by Remark \ref{lambda_texp}. We emphasize that the singularity in the dispersion relations are quite prominent, as $\Re\lambda_c(\hat\sigma)$ reaches $O(1)$ when $\hat\sigma\sim \e$, in contrast to the classical Ginzburg-Landau or Matthews-Cox cases where $\Re\lambda_c(\hat\sigma)\sim\e^2$ when $\hat\sigma\sim\e$. Note also that the effects of singularity are visible
	in the right-hand column of figures through the ``spikes'' near $\hat\sigma=0$. The final point to which we wish to draw emphasis is that it is $\lambda_c$ and $\lambda_s$ that intersect 
	and lose analyticity within an $O(\eps)$ domain, 
	while $\lambda_t$ remains analytic on a domain of $o(1)$. 
	We recall that this played a key role in the analysis of Cases (ii)-(iv). One final remark is that, for fixed $\eps$, the agreement between $\mu_t\hat\sigma^2$ and $\lambda_t(\hat\sigma)$ gets worse as $\abs{\kappa}$ 
increases.
This is to be expected, as the formula for $\mu_t$ in \eqref{mut} goes to $\infty$ as $\abs{\kappa}^2\to \kappa_E^2$.\\
	
	In the figures below, we have chosen the frequencies $\kappa=0$, $\kappa=\kappa_E/4(=0.25)$ and $\kappa=\kappa_E/2=(0.5)$ as the numerically observed stability boundary occurs at around $\kappa\approx\pm 0.3$.\footnote{
		Close to the real Ginzburg-Landau boundary of $1/3$.}


\begin{figure}[htbp]
\begin{center}
$
\begin{array}{ll}
\includegraphics[scale=0.6]{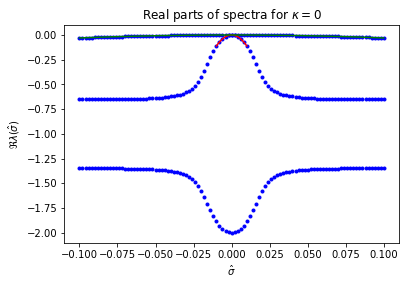}\;
\includegraphics[scale=0.6]{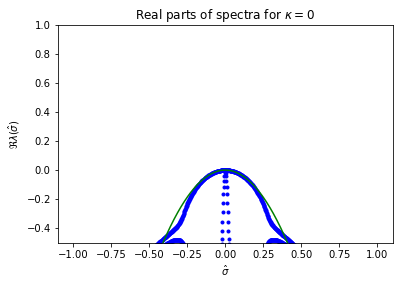}\;
\end{array}
$
\end{center}
\caption{ Wavenumber $\kappa=0$.}
\label{fig_1}\end{figure}

\begin{figure}[htbp]
\begin{center}
$
\begin{array}{ll}
\includegraphics[scale=0.6]{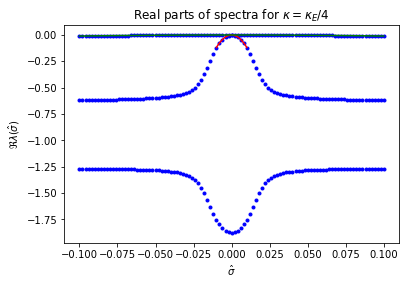}\;
\includegraphics[scale=0.6]{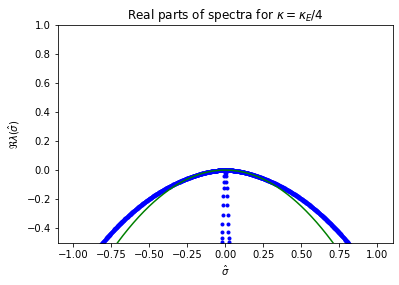}\;
\end{array}
$
\end{center}
\caption{Wavenumber $\kappa=\kappa_E/4$.}
\label{fig_2}\end{figure}

\begin{figure}[htbp]
	\begin{center}
		$
		\begin{array}{ll}
		\includegraphics[scale=0.6]{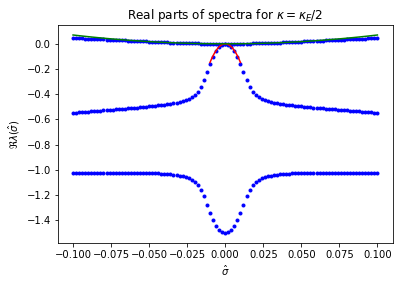}\;
		\includegraphics[scale=0.6]{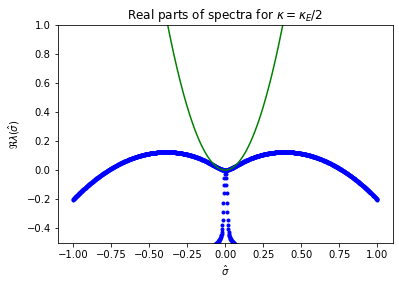}\;
		\end{array}
		$
	\end{center}
	\caption{ Wave number $\kappa=\kappa_E/2$.}
	\label{fig_3}\end{figure}

%
%
%
%
%
%
%
%
%
%

\section{Matched asymptotics for complex Ginzburg-Landau}\label{s:matched}
Here, we compute the expansion for the neutral translational mode for complex Ginzburg-Landau using matched asymptotics. To do so, we let $a,b,c$ be known complex numbers with $\Re(a),\Re(b)>0$ and $\Re(c)<0$ and consider a generic Ginzburg-Landau equation of the form
$$
A_T=aA_{XX}+bA+c|A|^2A.
$$
As before, we have a one-parameter family of periodic traveling waves parameterized by $\kappa$ of the form $A(X,T)=A_0e^{i(\kappa X-\omega T) }$ with $A_0>0$. The same linearization procedure as in \cite{WZ2} leads us to the spectral problem
$$
\lambda \bp u\\v\ep=-\s^2\bp \Re(a) & -\Im(a)\\ \Im(a) & \Re(a)\ep +2i\kappa \s \bp -\Im(a) & -\Re(a)\\ \Re(a) & -\Im(a)\ep+\bp 2A_0^2\Re(c) & 0 \\ 2A_0^2\Im(c) & 0 \ep=:m(\s)\bp u\\ v\ep.
$$
We then seek an eigenvalue $\l_t$ of the form $\l_t(\s)=i\a_t \s+\mu_t\s^2+O(\s^3)$. We notice that the second column of $m(\s)-\l_t(\s)Id$ is proportional to $i\s$, and so our expansion of the determinant is of the form
$$
\det(m(\s)-\l_t(\s)Id_2)=i\s\Big(\det P_0+i\s \det P_1+ H.O.T. \Big),
$$
with
$$
P_0=\bp 2A_0^2\Re(c) & -2\kappa\Re(a)\\ 2A_0^2\Im(c) & -2\kappa\Im(a)-\a_t\ep,
$$
and
$$
P_1=\bp -2\kappa\Im(a)-\a_t & -2\kappa\Re(a)\\ 2\kappa\Re(a) & -2\kappa\Im(a)-\a_t\ep+\bp 2A_0^2\Re(c) & -\Im(a)\\ 2A_0^2\Im(c) & \Re(a)+\mu_t\ep.
$$
Setting $\det P_0=\det P_1=0$ and solving the corresponding equations gives
$$
\a_t=-2\kappa\Im(a)+2\kappa\Re(a)\frac{\Im(c)}{\Re(c)} ,
$$
and
$$
\mu_t=-\frac{\big(-2\kappa\Im(a)-\a_t\big)^2+4\kappa^2\Re(a)^2+2A_0^2\big(\Re(a)\Re(c)+\Im(a)\Im(c) \big) }{2A_0^2\Re(c)} .
$$



\begin{thebibliography}{GMWZ7}



\bibitem[AGS]{AGS} D. Ambrosi, A. Gamba, and G. Serini, 
{\it Cell directional and chemotaxis in vascular morphogenesis,}
Bull. Math. Biol. 66 (2004), no. 6, 1851--1873.

\bibitem[BHZ]{BHZ} B. Barker, J. Humpherys, and K. Zumbrun, {\it One-dimensional stability of parallel shock layers in isentropic magnetohydrodynamics,} J. Differential Equations 249 (2010), no. 9, 2175--2213. 

\bibitem[BJNRZ]{BJNRZ} B. Barker, M.A. Johnson, P. Noble, L.M. Rodrigues, and K. Zumbrun, 
{\it Stability of viscous St. Venant roll waves: from onset to infinite Froude number limit,} 
J. Nonlinear Sci. 27 (2017), no. 1, 285--342.	

\bibitem[BJZ1]{BJZ1} B. Barker, S. Jung and K. Zumbrun, 
{\it Turing patterns in parabolic systems of conservation laws and numerically observed stability of periodic waves,}
Phys. D 367 (2018), 11--18.

\bibitem[BJZ2]{BJZ2} J.C. Bronski, M.A. Johnson, and K. Zumbrun,{\it On the modulation equations and stability of periodic generalized Korteweg-de Vries waves via Bloch decompositions,}
Phys. D 239 (2010), 2057--2065.

\bibitem[BLWZ]{BLWZ} K. Buck, V.T. Luan, A. Wheeler, and K. Zumbrun,
	{\it Numerical approximation of singular amplitude equations arising in biomorphogenesis}, Work in progress.














\bibitem[dRS1]{dRS1} B. de Rijk and G. Schneider {\it Global existence and decay in nonlinearly coupled reaction-diffusion-advection equations with different velocities,} J. Differential Equations 268 (2020), no. 7, 3392-3448.

\bibitem[dRS2]{dRS2} B. de Rijk and G. Schneider {\it Global existence and decay in multi-component reaction-diffusion-advection systems with different velocities: oscillations in time and frequency,} NoDEA Nonlinear Differential Equations Appl. 28 (2021), no. 1, Paper no. 2.

\bibitem[DPTZ]{DPTZ} P. Denton, S. Parke, T. Tao, and X. Zhang, {\it Eigenvectors from eigenvalues: a survey of a basic identity in linear algebra,} Bull. Amer. Math. Soc. (N.S.) 59 (2022), no. 1, 31-58.


\bibitem[DJMR]{DJMR} P. Donnat, J.-L. Joly, G. M\'etivier, and J. Rauch {\it Diffractive nonlinear geometric optics, } S\'eminaire sur les \'Equations aux D\'eriv\'ees Partielles, 1995-1996, Exp. No. XVII.

\bibitem[E]{E} W. Eckhaus, {\it Studies in nonlinear stability theory,} Springer tracts in Nat. Phil. Vol. 6, 1965.





\bibitem[FS1]{FS1} H. Freist\"hler and P. Szmolyan, {\it Spectral stability of small shock waves,} Arch. Ration. Mech. Anal. 164 (2002), no. 4, 287--309.

\bibitem[FS2]{FS2} H. Freist\"hler and P. Szmolyan, {\it Spectral stability of small-amplitude viscous shock waves in several space dimensions,} Arch. Ration. Mech. Anal. 195 (2010), no. 2, 353--373.


\bibitem[GAC]{GAC} A Gamba, D. Ambrosi, A. Coniglio, A.D. Candia, S.D. Taia, E. Geraudo, G. Serini, 
L. Preziosi, F. Bussolino, {\it Percolation, morphogenesis, and Burgers dynamics in blood vessels formation,}
Phys.  Rev. Lett., 90(2003).




	
\bibitem[H]{H} B. Hilder, {\it Modulating traveling fronts in a dispersive Swift-Hohenberg equation coupled to an additional conservation law, } J. Math. Anal. Appl. 513 (2022), no. 2.

\bibitem[HSZ]{HSZ} T. H\"acker, G. Schneider, and D. Zimmermann {\it Justification of the Ginzburg-Landau approximation in case of marginally stable long waves.} J. Nonlinear Sci. 21 (2011), no. 1, 93-113.



\bibitem [JNRYZ]{JNRYZ} M. Johnson, P. Noble, L.M. Rodrigues,  Z. Yang, and K. Zumbrun,
	{\it Spectral stability of inviscid roll waves,} Comm. Math. Phys. 367 (2019), no. 1, 265--316. 

\bibitem [JNRZ]{JNRZ} M. Johnson, P. Noble, L.M. Rodrigues,  and K. Zumbrun,
{\it Behavior of periodic solutions of viscous conservation laws under localized and nonlocalized perturbations,}
Invent. Math. 197 (2014), no. 1, 115--213. 

\bibitem [JNRZ2]{JNRZ2} M. Johnson, P. Noble, L.M. Rodrigues,  and K. Zumbrun,
	{\it Spectral stability of periodic wave trains of the Korteweg-de Vries/Kuramoto-Sivashinsky
	equation in the Korteweg-de Vries limit,} Trans. Amer. Math. Soc. 367 (2015), no. 3, 2159--2212.




\bibitem[JZ1]{JZ1} M. Johnson and K. Zumbrun, {\it Rigorous justification of the Whitham modulation equations for the generalized Korteweg-de Vries equation, } Stud. Appl. Math. 125 (1) (2010) 69-89.

\bibitem[JZ2]{JZ2} M. Johnson and K. Zumbrun, {\it Nonlinear stability of periodic traveling wave solutions of systems of viscous conservation laws in the generic case,} J. Differential Equations 249 (5) (2010) 1213-1240.

\bibitem[JZN]{JZN} M. Johnson, K. Zumbrun, and P. Noble, {\it Nonlinear stability of viscous roll waves,}
SIAM J. Math. Anal. 43 (2011), no. 2, 577--611.

\bibitem [K]{K} T. Kato,
{\it Perturbation theory for linear operators},
Springer--Verlag, Berlin Heidelberg (1985).

\bibitem [KSM]{KSM} P. Kirrmann, G. Schneider, and A. Mielke,
	{\it The validity of modulation equations for extended systems with cubic nonlinearities,}
Proc. Roy. Soc. Edinburgh Sect. A 122 (1992), no. 1-2, 85--91.

\bibitem [ShK]{ShK} Y. Shizuta, Yasushi and S. Kawashima, {\it  Systems of equations of hyperbolic-parabolic type with applications to the discrete Boltzmann equation,} Hokkaido Math. J. 14 (1985), no. 2, 249--275.





\bibitem[L]{L} T.-P. Liu, {\it Hyperbolic conservation laws with relaxation,} Comm. Math. Phys. 108 (1987), no. 1, 153--175.

\bibitem[LBK]{LBK}L\"ucke, M., W. Barten, and M. Kamps, 
{\it Convection in binary mixtures: the role of the concentration field},
Physica D 61 (1992) 183--196.

\bibitem[LW]{LW} Q.Q. Liu and X.L.Wu, {\it Stability of rarefaction wave for viscous vasculogenesis model,}
Discrete Contin.  Dyn. Syst. Ser. B, 27(2022), 7089--7108.

\bibitem[LWZ]{LWZ} O. Lafitte, M. Williams, and K. Zumbrun, {\it Block-diagonalization of ODEs in the semiclassical limit and $C^\omega$ versus $C^\infty$ stationary phase,} SIAM J. Math. Anal. 48 (2016), no. 3, 1773--1797.

\bibitem [Mai] {Mai} P. K. Maini,
{\it Applications of mathematical modelling to biological patern formation,} 
Coherent Strutures in Complex Systems (Sitges, 2000). Lecture Notes in Phys. Vol. 567,
Springer-Verlag, Berline (2001) 205--217.

\bibitem [Ma] {Ma} D. Manoussaki,
{\it Modeling and simulation of the formation of vascular networks,}
Math. Model. Mumer. Anal. 37 (2003) 581--600.

\bibitem[MZ1]{MZ1} C. Mascia and K. Zumbrun, 
{\it Pointwise Green's function bounds and stability of relaxation shocks,} 
Indiana Univ. Math. J. 51 (2002), no. 4, 773--904. 

\bibitem[MZ]{MZ} C. Mascia and K. Zumbrun,
{\it Stability of large-amplitude viscous shock profiles
of hyperbolic--parabolic systems,} 
Arch. Ration. Mech. Anal. 172 (2004), no.1, 93--131.

\bibitem[MC]{MC} P.C. Matthews and S.M. Cox, 
{\it Pattern formation with a conservation law,} 
Nonlinearity 13 (2000), no. 4, 1293--1320.

\bibitem[MetZ]{MetZ} G. M\'etivier and K. Zumbrun,
{\it Large-amplitude modulation of periodic traveling waves,}
Discrete Contin. Dyn. Syst. Ser. S 15 (2022), no. 9, 2609--2632.


\bibitem [MO] {MO} J.D. Murray and G.F. Oster,
{\it Generation of biological pattern and form,} J. Math.  Appl. Med. Biol. 1 (1984) 51--75.


\bibitem[M1]{M1} A. Mielke, 
{\it A new approach to sideband-instabilities using the principle of 
reduced instability,}
 Nonlinear dynamics and pattern formation in the natural environment 
(Noordwijkerhout, 1994), 206--222, 
Pitman Res. Notes Math. Ser., 335, Longman, Harlow, 1995. 

\bibitem[M2]{M2} A. Mielke, 
{\it Instability and stability of rolls in the Swift-Hohenberg equation,}
Comm. Math. Phys. 189 (1997), no. 3, 829--853. 


\bibitem[M3]{M3} A. Mielke, 
{\it The Ginzburg-Landau equation in its role as a modulation equation,}
Handbook of dynamical systems, Vol. 2, 759--834, 
North-Holland, Amsterdam, 2002.

%








\bibitem[PZ]{PZ} R. Plaza and K. Zumbrun, {\it An Evans function approach to spectral stability of small-amplitude shock profiles,} Discrete Contin. Dyn. Syst. 10 (2004), no. 4, 885--924.


\bibitem[S1]{S1} G. Schneider, 
{\it Diffusive stability of spatial periodic solutions of the Swift-Hohenberg equation,}
Commun. Math. Phys. 178, 679--202 (1996). 

\bibitem [S2]{S2} G. Schneider, {\it Nonlinear diffusive stability
of spatially periodic solutions-- abstract theorem and higher space
dimensions},
Proceedings of the International Conference on Asymptotics
in Nonlinear Diffusive Systems (Sendai, 1997),  159--167,
Tohoku Math. Publ., 8, Tohoku Univ., Sendai, 1998.


\bibitem[SZ]{SZ} W. Schopf, and W. Zimmermann, 
{\it Convection in binary mixtures: the role of the concentration field,}
1990, Phys. Rev. A 41, 1145.


\bibitem [SBP] {SBP} M. Scianna, C.G. Bell, L. Preziosi,
{\it A review of mathematical models for the formation of vasular networks,} in: J. Theoretical Biol. v. 333 (2013)
174--209 ISSN 0022-5193.



\bibitem[S]{S} A. Sukhtayev, 
{\it Diffusive stability of spatially periodic patterns with a conservation law,} 
Preprint, arXiv:1610.05395.



\bibitem[T]{T} A. Turing, 
{\it The chemical basis of morphogenesis,} Philos. Trans. Roy. Soc. Ser. B 237 (1952) 37--72.

\bibitem [vH]{vH} A. van Harten, 
{\it On the validity of the Ginzburg-Landau's equation,}
J. Nonlinear Sci. 1 (1991), pp. 397--422.

\bibitem[vSH]{vSH} W. van Saarloos and P.C. Hohenberg, 
{\it Fronts, pulses, sourc`es, and sinks in generalized complex Ginsburg-Landau equations,}
Physica D , 56 (1992) pp. 303--367.


\bibitem[WZ1]{WZ1} A. Wheeler and K. Zumbrun,
	{\it Convective Turing bifurcation}, preprint; arXiv:2101.07239. 

\bibitem[WZ2]{WZ2} A. Wheeler and K. Zumbrun,
{\it Diffusive stability of convective Turing patterns,} 
preprint; arXiv:2101.08360.

\bibitem[WZ3]{WZ3} A. Wheeler and K. Zumbrun,
{\it Convective Turing bifurcation with conservation laws,} 
preprint; arXiv:2305.16457. 

\bibitem[Wh]{Wh} A. Wheeler, {\it Convective Turing bifurcation with conservation laws and
applications to biomorphology,} PhD. Thesis, Indiana University, April 2023.

\bibitem[W]{W} G. B. Whitham, \emph{Linear and Nonlinear Waves}, Pure and Applied Mathematics (New York), John Wiley \& Sons Inc., New York, 1999.
Reprint of the 1974 original, A Wiley-Interscience Publication.





\bibitem[Ze]{Ze} Y. Zeng, {\it Gas dynamics in thermal nonequilibrium and general hyperbolic systems with relaxation,}
Arch. Ration. Mech. Anal. 150 (1999), no. 3, 225--279.

\end{thebibliography}
\end{document}